\documentclass[11pt]{article}
\usepackage{xcolor}
\usepackage{graphicx}
\usepackage{amsmath,amsfonts,amssymb,graphics,amsthm}
\usepackage{hyperref}
\usepackage{comment}
\usepackage{tabularx}
\usepackage{enumerate}
\usepackage{bbm}
\usepackage{mathrsfs}
\usepackage[margin=1in]{geometry}
\usepackage[shortlabels]{enumitem}

\hypersetup{
	colorlinks=false,
	linktocpage,
}

\numberwithin{equation}{section}

\newtheorem{theorem}{Theorem}[section]

\newtheorem{lemma}[theorem]{Lemma}
\newtheorem{proposition}[theorem]{Proposition}

\newtheorem{remark}[theorem]{Remark}
\newtheorem{definition}[theorem]{Definition}

\theoremstyle{remark}

\def\@rst #1 #2other{#1}
\newcommand\MR[1]{\relax\ifhmode\unskip\spacefactor3000 \space\fi
	\MRhref{\expandafter\@rst #1 other}{#1}}
\newcommand{\MRhref}[2]{\href{http://www.ams.org/mathscinet-getitem?mr=#1}{MR#2}}

\def\MR#1{\href{http://www.ams.org/mathscinet-getitem?mr=#1}{MR#1}}

\newcommand{\C}{\mathbbm{C}}

\newcommand{\E}{\mathbbm{E}}

\newcommand{\R}{\mathbbm{R}}
\renewcommand{\P}{\mathbbm{P}}
\newcommand{\bbH}{\mathbbm{H}}

\newcommand{\eps}{\varepsilon}

\newcommand{\LF}{{{\mathrm{LF}}}}

\newcommand{\pZ}{{\cZ_\kappa^{ \alpha}}}

\let\Re\undefined
\DeclareMathOperator{\Re}{Re}
\let\Im\undefined
\DeclareMathOperator{\Im}{Im}

\DeclareMathOperator{\Cov}{Cov}
\DeclareMathOperator{\Var}{Var}
\DeclareMathOperator{\SLE}{SLE}
\DeclareMathOperator{\rSLE}{rSLE}

\def\cZ{\mathcal{Z}}

\def\cW{\mathcal{W}}

\def\cS{\mathcal{S}}

\def\cP{\mathcal{P}}

\def\cM{\mathcal{M}}
\def\cL{\mathcal{L}}

\def\cF{\mathcal{F}}

\def\cD{\mathcal{D}}
\def\cC{\mathcal{C}}
\def\cB{\mathcal{B}}
\def\cA{\mathcal{A}}

\def\alb#1\ale{\begin{align*}#1\end{align*}}
\def\allb#1\alle{\begin{align}#1\end{align}}

\newcommand{\aryb}{\begin{eqnarray*}}
	\newcommand{\arye}{\end{eqnarray*}}
\def\alb#1\ale{\begin{align*}#1\end{align*}}
\newcommand{\eqb}{\begin{equation}}
	\newcommand{\eqe}{\end{equation}}
\newcommand{\eqbn}{\begin{equation*}}
	\newcommand{\eqen}{\end{equation*}}

\newcommand{\ol}{\overline}

\newcommand{\wt}{\widetilde}
\newcommand{\wh}{\widehat}

\let\originalleft\left
\let\originalright\right
\renewcommand{\left}{\mathopen{}\mathclose\bgroup\originalleft}
\renewcommand{\right}{\aftergroup\egroup\originalright}

\DeclareMathAlphabet{\mathpzc}{OT1}{pzc}{m}{it}

\begin{document}

	\title{Liouville conformal field theory and the quantum zipper}
	\author{
		Morris Ang
	}

	\date{\today}
	
	\maketitle

\begin{abstract}
	Sheffield showed that conformally welding a $\gamma$-Liouville quantum gravity (LQG) surface to itself gives a Schramm-Loewner evolution (SLE) curve with parameter $\kappa = \gamma^2$ as the interface, and Duplantier-Miller-Sheffield proved similar results for $\kappa  = \frac{16}{\gamma^2}$ for $\gamma$-LQG surfaces with boundaries decorated by looptrees of disks or by continuum random trees. We study these dynamics for LQG surfaces coming from Liouville conformal field theory (LCFT). At stopping times depending only on the curve, we give an explicit description of the surface and curve in terms of LCFT and SLE. This has applications to both LCFT and SLE. We prove the boundary BPZ equations for LCFT, 
	a crucial input for subsequent work with Remy, Sun and Zhu deriving the structure constants of boundary LCFT. With Yu we  prove the reversibility of whole-plane $\SLE_\kappa$ for $\kappa > 8$ via a novel radial mating-of-trees, and will show the space of LCFT surfaces is closed under conformal welding.
\end{abstract}
	
	\setcounter{tocdepth}{1}
	
	\tableofcontents

\section{Introduction}
Polyakov introduced a one-parameter family of random surfaces called \emph{Liouville quantum gravity (LQG)} to make sense of summation over surfaces \cite{polyakov-qg1}. The \emph{mating-of-trees} framework studies LQG through its coupling with random curves called \emph{Schramm-Loewner evolution (SLE)}.  
Let $\kappa > 0$ and $\gamma = \min (\sqrt\kappa, \frac4{\sqrt\kappa})$. $\SLE_\kappa$ is a simple curve when $\kappa \leq 4$, self-intersecting when $\kappa \in (4,8)$, and space-filling when $\kappa \geq 8$. 
When $\kappa \in (0,4]$ there is an infinite-volume $\gamma$-LQG surface which, when decorated by an independent $\SLE_\kappa$ curve, is invariant in law under the operation of conformally welding the two boundary arcs according to their random length measures; this is called the \emph{quantum zipper} \cite{shef-zipper, hp-welding, kms-sle4}. 
Similar results hold for other ranges of $\kappa$ when the boundary of the LQG surface is modified to have non-trivial topology \cite{wedges}.

Starting with these stationary  quantum zippers, the mating-of-trees approach develops a theory of conformal welding of special LQG surfaces, culminating in landmark results such as the equivalence of the Brownian map and LQG \cite{lqg-tbm1} and the convergence of random planar maps to LQG \cite{gms-tutte, hs-cardy-embedding}. A recent program \cite{AHS-SLE-integrability, ARS-FZZ, AS-CLE, ASY-triangle} extends the conformal welding theory to a larger class of LQG surfaces which arise from \emph{Liouville conformal field theory (LCFT)}. In these conformal weldings, whole boundary arcs are glued at once, in contrast to the quantum zipper where the gluing is incremental.

We study the quantum zipper dynamics of \cite{shef-zipper, wedges} applied to random surfaces arising from LCFT, applying a different zipping mechanism for each parameter range of $\kappa$. 
For $\kappa \in (0,4]$, we conformally weld the left and right boundaries of the LQG surface. For $\kappa \in (4,8)$, we add a Poissonnian collection of looptrees of LQG disks to the boundary, then mate the forested boundaries. For $\kappa > 8$, we attach a pair of correlated continuum random trees to the boundary arcs of the LQG surface, then mate the continuum random trees. This gives an LQG surface with an interface curve, see figure above. In all three regimes, when the process is run until a stopping time depending only on the curve, we give an explicit description of the joint law of the field and curve. Roughly speaking, we show the curve is described by reverse $\SLE_\kappa$ and  the field is described by LCFT; see Theorems~\ref{thm-simple-zipper},~\ref{thm-simple-zipper-rho},~\ref{thm-forest-zipper} and~\ref{thm-sf-zipper} for details. The $\kappa = 8$ case  requires further arguments and is not treated in this work. 
\begin{figure}
	\begin{center}
		\includegraphics[scale=0.37]{figures/all-zippers}
	\end{center}
\end{figure}

We then give an application of the LCFT zipper. Using representation theoretic methods, Belavin, Polyakov and Zamolodchikov (BPZ) proposed that some correlation functions of conformal field theories should satisfy certain differential equations \cite{bpz-conformal-symmetry}. This was rigorously proved for LCFT in \cite{krv-local} and used in the landmark computation of the LCFT three-point function \cite{krv-dozz}. There are substantial conceptual and technical difficulties in adapting the argument of \cite{krv-local} to boundary LCFT. We instead use SLE martingales arising from the LCFT zipper to obtain the boundary LCFT BPZ equations. These equations require a non-trivial coupling of cosmological constants~\eqref{eq-cosmocoupling} initially conjectured from special cases \cite{FZZ} and later obtained in physics  under the ansatz of uniqueness of the primary field of given conformal dimension \cite{bb-hem-lcft}; our argument gives a novel conceptual explanation.

Our results have consequences for both LCFT and SLE. With Guillaume Remy, Xin Sun and Tunan Zhu \cite{ARSZ-PT}, we use the BPZ equations to prove that the boundary three-point function and the  bulk-boundary function of LCFT agree with formulas proposed by  Ponsot-Teschner~\cite{pt-formula} and Hosomichi~\cite{hosomichi} respectively; this is the boundary analog of~\cite{krv-dozz} and provides the initial data for the boundary LCFT conformal bootstrap. 
With Pu Yu \cite{ay-reversibility} we establish a radial mating-of-trees  via the LCFT zipper, and use it to prove the reversibility of whole-plane $\SLE_\kappa$ for $\kappa> 8$. Whole-plane reversibility was shown for $\kappa \in (0,4]$ and $\kappa \in (4,8)$ by \cite{zhan-rev-whole-plane} and \cite{ig4} respectively, and our result  resolves the remaining case. This answers two conjectures of~\cite{vw-loewner-kufarev}. Finally, with Pu Yu we will prove that the conformal welding of LCFT surfaces of arbitrary genus gives a curve-decorated LCFT surface, extending the program of \cite{AHS-SLE-integrability, ARS-FZZ, AS-CLE, ASY-triangle}. See Section~\ref{sec-outlook} for details.

\subsection{Liouville quantum gravity and Liouville conformal field theory}

We first briefly recall some preliminaries, see Sections~\ref{subsec-gff} and~\ref{sec-LF} for details. The free boundary Gaussian free field (GFF) on a simply-connected domain $D \subset {\mathbb C}$ is the Gaussian process on $D$ whose covariance kernel is the Green function; it can be understood as a random generalized function $h$ \cite{shef-gff}.  
For $\gamma \in (0,2]$ and $h$ a variant of the GFF, the $\gamma$-LQG area measure $\cA_h$ on $D$ and boundary measure $\cL_h$ on $\partial D$ are heuristically defined by $\cA_h(dz) = e^{\gamma h(z)}dz$ and $\cL_h(dx) = e^{\frac\gamma2 h(x)}dx$. These definitions are made rigorous via regularization and renormalization \cite{shef-kpz, shef-deriv-mart}.

Liouville conformal field theory (LCFT) is the quantum field theory 
introduced by Polyakov in his seminal work on quantum gravity and string theory \cite{polyakov-qg1}. LCFT can be defined on general surfaces with or without boundary; for concreteness we consider the surface $(D, g)$ where $D \subset \C$ is a 
simply connected domain with smooth boundary and $g$ is a Riemannian metric. The \emph{Liouville action} is
\eqb
S(\phi, g) = \frac1{4\pi} \int_D (|\partial^g \phi|^2 + Q R_g \phi + 4\pi \mu e^{\gamma \phi})d \lambda_g  + \frac1{2\pi} \int_{\partial D} (Q K_g \phi + 2\pi \mu_\partial e^{\frac\gamma2 \phi})d\lambda_{\partial g}
\eqe
where $\partial^g, R_g, K_g, \lambda_g, \lambda_{\partial g}$ denote the gradient, Ricci scalar curvature, geodesic curvature of the boundary, volume form and line element along $\partial D$, all with respect to $g$, and the \emph{cosmological constants} $\mu, \mu_\partial$ are complex numbers with $\Re \mu, \Re \mu_\partial \geq 0$ and $\Re (\mu + \mu_\partial) > 0$. We will soon see that $\phi$ should be a variant of the GFF, so  $e^{\gamma \phi} d\lambda_g$ and $e^{\frac\gamma2 \phi} d\lambda_{\partial g}$ can be interpreted as the $\gamma$-LQG area and boundary measures.  For insertions $(\alpha_i, z_i) \in \R \times \ol D$, the \emph{LCFT correlation function} is (nonrigorously) defined by the path integral
\eqb\label{eq-corr-fn-physics}
\langle \prod_{i=1}^n e^{\alpha_i \phi(z_i)} \rangle_{\mu, \mu_\partial} = \int_{\phi: D \to \R} \prod_{i=1}^n e^{\alpha_i \phi(z_i)} e^{-S(\phi, g)} D\phi
\eqe
where $D\phi$ represents a formal ``Lebesgue measure'' on the space of functions $\phi: D \to \R$. Correlation functions can be defined on all surfaces (with or without boundary). A fundamental goal in LCFT is to solve for all the correlation functions. This was achieved at the physics level of rigor by \cite{bpz-conformal-symmetry, do-dozz, zz-dozz} for surfaces without boundary, and \cite{bpz-conformal-symmetry, hosomichi, pt-formula} for surfaces with boundary. 

Mathematically, the correlation functions~\eqref{eq-corr-fn-physics} were rigorously constructed by interpreting $\exp(\frac1{4\pi} \int_D |\partial^g \phi|^2) D\phi$ as the law of $h+\mathbf c$ where $\mathbf c$ is a constant ``sampled'' from Lebesgue measure on $\R$ and $h$ is an independent GFF. This was first carried out on the sphere  \cite{dkrv-lqg-sphere} and subsequently extended to other surfaces  \cite{hrv-disk, remy-annulus, drv-torus, grv-higher-genus}. 
In a series of recent breakthroughs, the correlation functions of LCFT on  surfaces without boundary were rigorously computed \cite{krv-dozz, gkrv-bootstrap, gkrv-segal}. A similar program for surfaces with boundary is currently in progress \cite{ARS-FZZ, ARSZ-PT, wu-bootstrap-annulus}; as we explain in Section~\ref{sec-outlook} the present work gives a key ingredient for \cite{ARSZ-PT}. 

We will focus on LCFT on the disk \cite{hrv-disk}, parametrized by the upper half-plane ${\mathbb H}$.  It is rigorously defined via an infinite measure $\LF_{\mathbb H}$ on the space of generalized functions on ${\mathbb H}$ obtained by an additive perturbation of the GFF, which we call the \emph{Liouville field}; see Definition~\ref{def-LF-bare}.  
For $\delta>0$ and finitely many $(\alpha_j, z_j) \in \R \times \ol {\mathbb H}$ we can make sense of the  measure $\LF_{\mathbb H}^{(\alpha_j, z_j)_j, (\delta, \infty)} =  \prod_j e^{\alpha_j \phi(z_j)} e^{\delta \phi(\infty)}\LF_{\mathbb H}(d\phi)$ via regularization and renormalization. This is the Liouville field with \emph{insertions} of size $\alpha_i$ at $z_i$ and an insertion of size $\delta$ at $\infty$. 
The correlation functions~\eqref{eq-corr-fn-physics} are rigorously defined as expected values over  $\LF_{\mathbb H}^{(\alpha_j, z_j)_j, (\delta, \infty)}$, see for instance~\eqref{eq-corr}.

\subsection{The mating-of-trees approach: quantum surfaces and Schramm-Loewner evolution} \label{sec-random-geom}
We now review a second approach to LQG pioneered by Sheffield \cite{shef-zipper} and Duplantier-Miller-Sheffield \cite{wedges} called \emph{mating-of-trees}, which is motivated by the interpretation of LQG decorated by SLE as the scaling limit of random planar maps decorated by statistical physics models.

For simply connected $D, \wt D \subset \C$ and a distribution $h$ on $D$, suppose $g: D \to \wt D$ is a conformal map. We define the distribution $g \bullet_\gamma h$ on $\wt D$ by 
\eqb\label{eq-coordinate-change}
g \bullet_\gamma h := h \circ g^{-1} + Q \log |(g^{-1})'|, \qquad Q = \frac\gamma2 + \frac2\gamma. \eqe
A \emph{$\gamma$-LQG surface} or  \emph{quantum surface} is an equivalence class of pairs $(D,h)$ where $(D,h) \sim_\gamma (\wt D, \wt h)$ if there exists a conformal map $g: D \to \wt D$ such that $\wt h = g \bullet_\gamma h$ \cite{shef-zipper}. The pair $(D, h)$ is called an embedding of the quantum surface. The $\gamma$-LQG area and boundary measure are intrinsic to the quantum surface: If $h$ is a variant of the GFF on $D$, then writing $g_*$ for the pushforward under $g$, we have $g_*\cA_h = \cA_{\wt h}$ and $g_*\cL_h = \cL_{\wt h}$ where $\wt h = g \bullet_\gamma h$. 
One can further define \emph{beaded quantum surfaces} which may have disconnected interior, and 
define (beaded) quantum surfaces decorated by curves and marked points. See Section~\ref{subsec-wedges-cells}.

Schramm-Loewner evolution (SLE) \cite{schramm0, schramm-sle} is a canonical random planar curve which describes the scaling limits of many critical 2D lattice models, e.g.\ \cite{smirnov-cardy, lsw-lerw-ust, ss-dgff, cdchks-ising}. The parameter $\kappa>0$ describes the ``roughness'' of the $\SLE_\kappa$ curve: the curve is simple when $\kappa \in (0,4]$, self-intersecting (but not self-crossing) when $\kappa \in (4,8)$, and space filling when $\kappa \geq 8$. There are two versions of $\SLE_\kappa$. Forward $\SLE_\kappa$ is a random curve defined for each simply connected domain with two marked boundary points, characterized by conformal invariance and a domain Markov property. See for instance \cite{lawler-book, bn-sle-notes-alt} for introductory expository works. On the other hand, reverse $\SLE_\kappa$ is \emph{not} a curve, rather, it is a random process of curves $(\eta_t)$ in the upper half-plane $\bbH$ where for each time $t$ we have a curve $\eta_t$ from the boundary to an interior point. In a sense, reverse SLE describes a curve growing from its base.

Starting with the work of Sheffield \cite{shef-zipper} and Duplantier-Miller-Sheffield \cite{wedges}, there are now numerous results in the literature saying that for $\kappa > 0$ and $\gamma = \min(\sqrt\kappa, 4/\sqrt\kappa)$, when certain $\gamma$-LQG surfaces are cut by independent $\SLE_\kappa$-type curves, the resulting objects are independent. 
We recall examples of such results for the three $\kappa$ parameter ranges $(0,4], (4,8)$ and $[8, \infty)$ (Propositions~\ref{prop-wedge-welding},~\ref{prop-mate-forests} and~\ref{prop-mot}).

For $W \geq \frac{\gamma^2}2$ there is a $\gamma$-LQG surface called a \emph{weight $W$ quantum wedge} which has the half-plane topology. It has two boundary marked points  which we call $0$ and $\infty$; neighborhoods of $\infty$ have infinite quantum area, whereas regions bounded away from $\infty$ have finite quantum area. See Definition~\ref{def-wedge} for details. The following result is due to \cite[Theorem 1.8]{shef-zipper} for $\kappa < 4$ and  \cite[Theorem 1.2]{hp-welding} for $\kappa = 4$, see Figure~\ref{fig-orig-zipper1} (left).

\begin{proposition}\label{prop-wedge-welding}
	Let $\kappa \leq 4$ and $\gamma = \sqrt\kappa$. Let $(\bbH, \phi, 0, \infty)$ be an embedding of a weight $4$ quantum wedge, and let $\eta$ be an independent $\SLE_\kappa$ curve in $\bbH$ from $0$ to $\infty$. Let $D_\ell$ and $D_r$ be the connected components of $\bbH \backslash \eta$ to the left and right of $\eta$ respectively. Then the quantum surfaces $\cS_\ell = (D_\ell, \phi, 0, \infty)/{\sim_\gamma}$ and $\cS_r = (D_r, \phi, 0, \infty)/{\sim_\gamma}$ are independent weight $2$ quantum wedges. Moreover, the quantum boundary measures along $\eta$ as defined by $\cS_\ell$ and $\cS_r$ agree. 
\end{proposition}
In Proposition~\ref{prop-wedge-welding}, the quantum surface $(\bbH, \phi, 0, \infty)/{\sim_\gamma}$ can be recovered from $(\cS_\ell, \cS_r)$, as the \emph{conformal welding} of $\cS_\ell$ and $\cS_r$ according to quantum length measure on their respective boundary arcs.  See Section~\ref{sec-conf-weld} for details. 

\begin{figure}
	\begin{center}
		\includegraphics[scale=0.45]{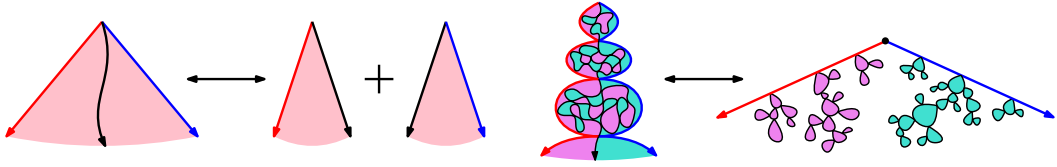}
	\end{center}
	\caption{\label{fig-orig-zipper1} \textbf{Left:} For $\kappa \leq 4, \gamma =\sqrt\kappa$, cutting a weight 4 quantum wedge by an independent $\SLE_\kappa$ curve gives independent weight 2 quantum wedges (Proposition~\ref{prop-wedge-welding}).  \textbf{Right:} For $\kappa \in (4,8), \gamma = \frac4{\sqrt\kappa}$, cutting a weight $2-\frac{\gamma^2}2$ quantum wedge by an independent $\SLE_\kappa(\frac\kappa2-4; \frac\kappa2 - 4)$ curve gives independent forested lines  (Proposition~\ref{prop-mate-forests}).}
\end{figure}

The setup for Proposition~\ref{prop-mate-forests} is more involved. 
One can also define weight $W$ quantum wedges when $W \in (0,\frac{\gamma^2}2)$. In this regime, the quantum wedge is called \emph{thin} and is the concatenation of a chain of countably many ``beads''. Each quantum wedge comes with two marked points, $0$ and $\infty$. See Definition~\ref{def-thin-wedge}. There is also a variant of forward $\SLE_\kappa$ called forward $\SLE_\kappa(\rho_-; \rho_+)$ (with force points immediately to the left and right of the curve starting point), introduced in \cite{lsw-restriction} and studied in, e.g., \cite{dubedat-rho, ig1}.

A \emph{forested line} is a Poisson point process of pairs $(\cD, t)$, where each $t>0$ and each $\cD$ is a looptree of quantum surfaces with the disk topology; it may be visualized as the positive real line $(0,\infty)$ with a looptree $\cD$ attached to the point $t \in (0,\infty)$ for each pair $(\cD, t)$. 
We will not need the precise definition in this paper so we omit it, but see \cite[Section 1.4.2]{wedges} for details.

\begin{proposition}[{\cite[Theorem 1.4.7]{wedges}}]\label{prop-mate-forests}
	Let $\kappa \in (4,8)$ and $\gamma = \frac4{\sqrt\kappa}$. 
	Sample a weight $(2 - \frac{\gamma^2}2)$ (thin) quantum wedge and let $\eta$ be the curve from $0$ to $\infty$ obtained by concatenating  independent $\SLE_{\kappa}(\frac\kappa2 - 4; \frac\kappa2-4)$ curves in each bead. Then $\eta$ divides the quantum wedge into two independent forested lines. Moreover the curve-decorated quantum wedge is  measurable with respect to the pair of forested lines. 
\end{proposition}
See Figure~\ref{fig-orig-zipper1} (right). We note that 
\cite[Theorem 1.4.7]{wedges} further asserts that the LQG length measure on $\eta$ agrees with a \emph{generalized quantum length measure} intrinsically defined for each forested line, so the quantum wedge decorated by SLE in Proposition~\ref{prop-mate-forests} is a ``mating'' of the forested lines according to generalized quantum length. 

When $\kappa \geq 8$, the $\SLE_\kappa$ curve is space-filling. In this regime, the seminal mating-of-trees theorem \cite[Theorem 1.4.1]{wedges} can be stated as follows. See Figure~\ref{fig-mot}. 
\begin{proposition}\label{prop-mot}
	Let $\kappa \geq 8$ and $\gamma = \frac{4}{\sqrt\kappa}$. 
	Let $({\mathbb H}, \phi, 0, \infty, \eta)$ be an embedding of a weight $(2-\frac{\gamma^2}2)$ quantum wedge decorated by an independent $\SLE_\kappa$ curve $\eta$. Parametrize $\eta$ by quantum area, so $A_\phi(\eta([0,s])) = s$. On the counterclockwise (resp.\ clockwise) boundary arc of $\eta([0,s])$ from $0$ to $\eta(s)$, let $X_s^-$ and $X_s^+$ (resp.\ $Y_s^-$ and $Y_s^+$) be the quantum lengths of the boundary segments in $\R$ and ${\mathbb H}$ respectively. Then $(X_s, Y_s) := (X_s^+ - X_s^-,Y_s^+ - Y_s^-)$ is correlated Brownian motion with covariance
	\[\Var(X_s) = \Var(Y_s) = \mathbbm a^2 s, \qquad  \Cov(X_s, Y_s) = -\mathbbm a^2 \cos(\frac{4\pi}\kappa)s \qquad  \text{ where } \mathbbm a^2 = 2/\sin(\frac{4\pi}\kappa).\]
	Moreover, the curve-decorated quantum surface $({\mathbb H}, \phi, 0, \infty, \eta)/{\sim_\gamma}$ is measurable with respect to $(X_s, Y_s)_{s \geq 0}$. 
\end{proposition}
The measurability statement and the fact that $(X_s, Y_s)$ evolve as correlated Brownian motion was shown in \cite{wedges}, the correlation $\cos(\frac{\pi \gamma^2}4)$ was obtained in \cite{kappa8-cov}, and the value of $\mathbbm a^2$ was identified in \cite{ARS-FZZ}.
(For the $\kappa \in (4,8)$ mating-of-trees theorem of \cite{wedges}, see Proposition~\ref{prop-mot-forr}.) 
Proposition~\ref{prop-mot} has the interpretation that the quantum wedge decorated by independent SLE corresponds to a mating of a pair of continuum random trees, where each process $(X_s)$ and $(Y_s)$ corresponds to a continuum random tree; see Figure~\ref{fig-mot} for a brief explanation, or \cite[Sections 1.3, 1.4]{wedges} for more details.

Proposition~\ref{prop-mate-forests} shows that there is a way to mate forested lines to get LQG decorated by SLE, but \cite{wedges} does not give an explicit description of this procedure. Similarly, the mating procedure for continuum random trees in Proposition~\ref{prop-mot} doe not have an explicit description in \cite{wedges}. Our $\kappa \in (4,8)$ and $\kappa > 8$ LCFT zipper theorems can be equivalently formulated in terms of these nonexplicit matings, but we will instead give more concrete statements. 

\begin{figure}
	\begin{center}
		\includegraphics[scale=0.45]{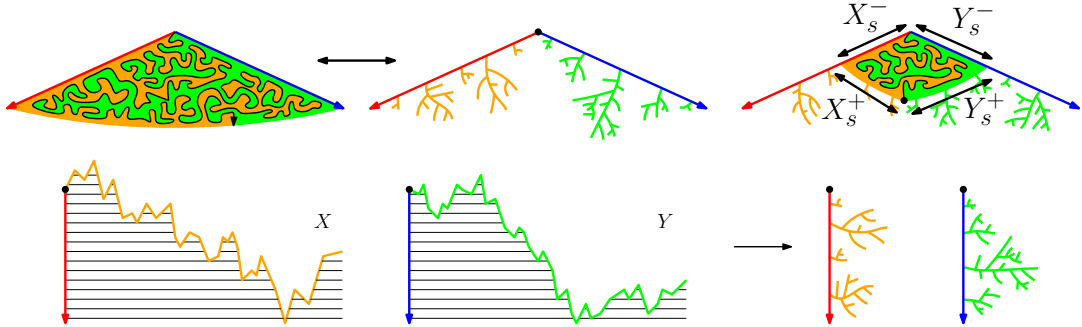}
	\end{center}
	\caption{\label{fig-mot} Interpretation of Proposition~\ref{prop-mot} via continuum random trees. \textbf{Top left:} For $\kappa \geq 8$ and $\gamma = \frac4{\sqrt \kappa}$, cutting a weight $2-\frac{\gamma^2}2$ quantum wedge by an independent $\SLE_\kappa$ curve $\eta$ gives a pair of correlated continuum random trees, which we now describe. \textbf{Top right:} For each $s> 0$, the quantities $X_s^-, X_s^+, Y_s^-, Y_s^+$ describe the quantum boundary lengths of $\eta([0,s])$ (region explored by curve on quantum wedge at time $s$).  \textbf{Bottom:} Let $(X_s, Y_s) := (X_s^+ - X_s^-,Y_s^+ - Y_s^-)$. We can construct a continuum random tree from $X$ by plotting the graph of $s \mapsto X_s$, letting $A \subset \R^2$ be the set of points lying on or below the graph, and identifying points in $A$ which lie on a horizontal chord which stays below the graph. Likewise we construct a continuum random tree using $Y$.}
\end{figure}

For context, we mention some breakthroughs accomplished via the mating-of-trees framework. The Brownian map is the metric space scaling limit of uniform random planar map \cite{mm-brownian-map, legall-uniqueness, miermont-brownian-map}; building on a variant of Proposition~\ref{prop-mate-forests}, \cite{lqg-tbm1} constructed a metric on $(\gamma =\sqrt{8/3})$-LQG and showed it agrees with the Brownian map. Next, it is a longstanding belief that under conformal embedding in the plane, many random planar maps should converge to LQG; this has been proved for the case of the uniform random planar map \cite{hs-cardy-embedding} and the so-called mated-CRT maps \cite{gms-tutte}. These scaling limit results depend fundamentally on the mating-of-trees theorem (stated as Propositions~\ref{prop-mot} and~\ref{prop-mot-forr}).

In recent years the author and collaborators have developed a program where inputs coming from mating-of-trees/conformal welding and LCFT are combined to obtain results in each of these domains \cite{AHS-SLE-integrability, ARS-FZZ, ARS-annulus, AS-CLE, ASY-triangle, ARSZ-PT}. The present work is part of this program.

\subsection{The $\kappa \leq 4$ LCFT  zipper}

Let $\kappa \in (0,4]$ and $\gamma = \sqrt\kappa$. 
See Figure~\ref{fig-zipper-simple} for a summary of the LCFT zipper in this range. Note that throughout the paper, we will use the language of probability theory in the setting of non-probability measures, see Section~\ref{sec-prelim} for a careful explanation. Further note that our Liouville field notation differs from e.g.\ \cite{AHS-SLE-integrability, ARS-FZZ, AS-CLE, ASY-triangle}, see Remark~\ref{rem-lf-notation}.

Let $n \geq 0$, let $(\alpha_j, z_j) \in \R\times \ol {\mathbb H}$ such that $z_1, \dots, z_n$ are distinct, and let $\delta \in \R$. Sample a field $\phi_0$ from the infinite measure $\LF_{\mathbb H}^{(-\frac1{\sqrt\kappa}, 0), (\alpha_j, z_j)_j, (\delta, \infty)}$. Let $s > 0$ satisfy $\cL_{\phi_0}((-\infty, 0)), \cL_{\phi_0}((0, \infty)) > s$. For each $u \in (0, s]$ let $p_u \in (0,\infty)$ and $q_u \in (-\infty, 0)$ be the points such that $\cL_{\phi_0}([0,p_u]) = \cL_{\phi_0}([q_u, 0])= u$. We want to glue the boundary arcs $[q_s, 0]$ and $[0,p_s]$ of ${\mathbb H}$ together, identifying $q_u$ with $p_u$ for $u \in (0,s]$. Almost surely there is a simple curve $\hat \eta_s:[0,s] \to \ol {\mathbb H}$ such that $\hat \eta_s \cap \R = \hat \eta_s(s)$ and a conformal map $\hat g_s: {\mathbb H}\to {\mathbb H}\backslash \hat \eta_s$ fixing $\infty$ such that $\hat g_s(p_u) = \hat g_s (q_u) = \hat \eta_s (u)$ for all $u \leq s$; this is called a \emph{conformal welding}.  
The pair $(\hat \eta_s, \hat g_s)$ is unique modulo conformal automorphisms of ${\mathbb H}$, so specifying the \emph{hydrodynamic normalization} $\lim_{z\to\infty} \hat g_s(z) - z = 0$ uniquely defines $(\hat \eta_s, \hat g_s)$. The existence and uniqueness of $(\hat \eta_s, \hat g_s)$ was shown in \cite{shef-zipper} for $\kappa < 4$; for $\kappa = 4$ existence was established by \cite{hp-welding} and uniqueness by \cite{kms-sle4}. See Section~\ref{sec-conf-weld} for more details on conformal welding of LQG surfaces.

\sloppy
Thus, for $\phi_0 \sim \LF_{\mathbb H}^{(-\frac1{\sqrt\kappa},0), (\alpha_j, z_j)_j, (\delta, \infty)}$, we can define a process $(\hat \eta_s, \hat g_s)$ for $s < \min (\cL_{\phi_0}(-\infty, 0), \cL_{\phi_0}(0, \infty))$. The \emph{half-plane capacity} of $\hat \eta_s$ is $\mathrm{hcap}(\hat\eta_s) := \lim_{z \to \infty} z (z-\hat g_s(z))$. We reparametrize time to get a process $( \eta_t, g_t)$ such that $\mathrm{hcap}(\eta_t) = 2t$. Define $\phi_t = g_t \bullet_\gamma \phi_0$ and  $W_t = \eta_t \cap \R$. See Figure~\ref{fig-zipper-simple}. 

\fussy

We first give a description of the law of the field and curve when the process is run until a stopping time before any marked points are zipped into the curve.

\begin{figure}
	\begin{center}
		\includegraphics[scale=0.36]{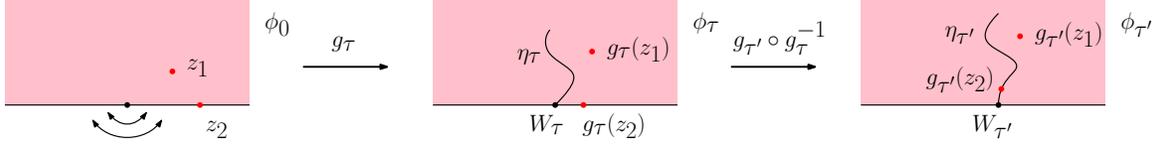}
	\end{center}
	\caption{\label{fig-zipper-simple} Let $\kappa \leq 4$ and $\gamma = \sqrt\kappa$. \textbf{Left:} Liouville field $\phi_0$ with insertions at $z_1 \in {\mathbb H}$ and $z_2 \in \R$. \textbf{Middle:} We conformally weld the left and right boundary arcs of $\phi_0$ according to quantum length to get $(\phi_t, \eta_t)_{t \geq 0}$. Time is parametrized by half-plane capacity of the curve. For a stopping time $\tau$ for $\cF_t = \sigma(\eta_t)$ such that $g_\tau(z_2)$ has not yet hit the curve, Theorem~\ref{thm-simple-zipper} describes the law of $(\phi_\tau, \eta_\tau)$ in terms of the Liouville field and reverse SLE. \textbf{Right:} At the time $T$ that $g_T(z_2)$ hits $W_T$, if the quantum length measure near $W_T$ is finite, we can continue the conformal welding process to ``zip the marked point into the bulk'', e.g. until the depicted time $\tau'$. Theorem~\ref{thm-simple-zipper-rho} describes the law of $(\phi_{\tau'}, \eta_{\tau'})$.
	}
\end{figure}

\begin{theorem}\label{thm-simple-zipper}
	In the setting immediately above, assume $z_1, \dots, z_n \neq 0$ and let $\tau$ be a stopping time with respect to the filtration $\cF_t = \sigma(\eta_t)$ such that a.s.\ $g_\tau(z_j) \not \in  \eta_\tau$ for all $j$. Let $I = \{ i \: : \: z_i \in {\mathbb H}\}$ and $B = \{ b \: : \: z_b \in \R\}$. Then the law of $(\phi_\tau, \eta_\tau)$ is
	\[ \prod_{i \in I} |g_\tau'(z_i)|^{2 \Delta_{\alpha_i}} \prod_{b \in B} |g_\tau'(z_b)|^{\Delta_{2\alpha_b}}  \LF_{\mathbb H}^{(-\frac1{\sqrt\kappa}, W_\tau), (\alpha_j, g_\tau(z_j))_j, (\delta, \infty)} (d\phi) \rSLE^\tau_\kappa (d\eta),\]
	where $\Delta_\alpha := \frac\alpha2 (Q-\frac\alpha2)$ and $\rSLE_\kappa^\tau$ is the law of reverse $\SLE_\kappa$ run until the stopping time $\tau$. 
\end{theorem}

Informally, zipping up a Liouville field until a stopping time $\tau$ that depends only on the zipping interface, the curve $\eta_\tau$ is described by reverse $\SLE_\kappa$, and given $\eta_\tau$ the resulting field is described by a Liouville field with insertions at locations determined by $\eta_\tau$. 

In Theorem~\ref{thm-simple-zipper}  the condition that $g_\tau(z_j) \not \in \eta_\tau$ for all $j$ is necessary, since otherwise the law of the curve would be  singular with respect to reverse $\SLE_\kappa$. In Theorem~\ref{thm-simple-zipper-rho}, we allow boundary marked points to be zipped into the bulk, by using an $\SLE_\kappa$ variant called \emph{reverse $\SLE_{\kappa}$ with force points} (see Section~\ref{subsec-sle}). Regardless, the zipping procedure can only be run until the \emph{continuation threshold}, defined as the first time $t \leq \infty$ that any neighborhood of $W_t$ in $\R$ has infinite quantum length, i.e.\ $\cL_{\phi_t}((W_t-\eps, W_t+ \eps)) = \infty$ for all $\eps >0$. 
Once the continuation threshold is hit, there is no canonical way to continue the conformal welding. 

For finitely many $(a_j, p_j) \in \R \times \ol {\mathbb H}$ such that the points $p_j$ are distinct, we define 
\eqb\label{eq-Z0}
\cZ ((a_j, p_j)_j) = \prod_{i \in I} (2 \Im p_i)^{-a_i^2/2} \prod_{j < k} e^{a_j a_k G(p_j, p_k)} , \quad I = \{ i : p_i \in {\mathbb H}\}, 
\eqe
where $G(p,q) = -\log|p-q| - \log |p-\ol q|$. If the $p_j$ are not distinct, we combine all pairs $(a,p)$ with the same $p$ by summing their $a$'s to get a collection $(a'_j, p'_j)$ with $p'_j$ distinct, and define $\cZ ((a_j, p_j)_j) := \cZ ((a'_j, p'_j)_j) $. 

\begin{theorem}\label{thm-simple-zipper-rho}
	In the setting above Theorem~\ref{thm-simple-zipper}, let $\tau$ be a stopping time with respect to the filtration $\cF_t = \sigma(\eta_t)$ which is not beyond the continuation threshold. 
	Then the law of $(\phi_\tau, \eta_\tau)$ is 
	\eqb\label{eq-thm-simple-zipper-rho}
	\frac{\cZ((-\frac1{\sqrt\kappa}, 0), (\alpha_j, z_j)_j)}{\cZ( (-\frac1{\sqrt\kappa},W_\tau), (\alpha_j, g_\tau(z_j))_j)}
	\LF_{\mathbb H}^{(-\frac1{\sqrt\kappa}, W_\tau), (\alpha_j, g_\tau(z_j))_j, (\delta, \infty)} (d\phi) \rSLE^\tau_{\kappa, \rho} (d\eta)
	\eqe
	where $\rSLE_{\kappa, \rho}^\tau$ denotes the law of reverse $\SLE_{\kappa, \rho}$ with a force point at $z_j$ of weight $\rho_j = 2\sqrt\kappa \alpha_j$ for each $j$, run until the stopping time $\tau$. 
\end{theorem}
In fact,  Theorem~\ref{thm-simple-zipper} is an immediate consequence of Theorem~\ref{thm-simple-zipper-rho} via the formula for the Radon-Nikodym derivative of $\rSLE_{\kappa, \rho}^\tau$ with respect to $\rSLE_{\kappa}^\tau$ given in Proposition~\ref{prop-sle-weighted}.

We emphasize that while we study the same Liouville field as e.g.\ \cite{AHS-SLE-integrability, ARS-FZZ, AS-CLE, ASY-triangle}, for boundary insertions our notation differs by a factor of 2, see Remark~\ref{rem-lf-notation}. The present choice of notation simplifies the statement of Theorem~\ref{thm-simple-zipper-rho} since  boundary insertions zipped into the bulk maintain the same value of $\alpha$.

\begin{remark}
	Let $R_t := \cZ((-\frac1{\sqrt\kappa}, 0), (\alpha_j, z_j)_j)/\cZ( (-\frac1{\sqrt\kappa},W_t), (\alpha_j, g_t(z_j))_j)$ be the prefactor in~\eqref{eq-thm-simple-zipper-rho} at time $t$. If $\tau$ is a time such that $g_\tau(z_j) = W_\tau$ for some $j$, then $\lim_{t\uparrow \tau} R_t \in \{0,\infty\}$ whereas $R_\tau \in (0, \infty)$. This apparent discontinuity is  only cosmetic: the definition of $\LF_{\mathbb H}^{(-\frac1{\sqrt\kappa}, W_t), (\alpha_j, g_t(z_j))_j, (\delta, \infty)}$ includes a factor of $C_\kappa^{(-\frac1{\sqrt\kappa},W_t), (\alpha_j, g_t(z_j))_j, (\delta, \infty)}$ (defined in~\eqref{eq-C}), and $C_\kappa^{(-\frac1{\sqrt\kappa},W_t), (\alpha_j, g_t(z_j))_j, (\delta, \infty)} R_t$ is continuous in $t$. 
\end{remark}

\subsection{The $\kappa \in (4,8)$ LCFT  zipper}\label{sec-intro-48}
	In this section we work with \emph{beaded} quantum surfaces, whose interiors may not be connected. See Section~\ref{subsec-wedges-cells} for the definition. 
	
Let $\kappa \in (4,8)$ and $\gamma =\frac4{\sqrt\kappa}$. Using Proposition~\ref{prop-mate-forests}, 
we first construct a process of decorated beaded quantum surfaces $(\cS_s)_{s > 0}$ which corresponds to incrementally mating a pair of independent forested lines. See  Figure~\ref{fig-zipper-looptree} (left). 

Embed a weight $(2-\frac{\gamma^2}2)$ quantum wedge as $(D, \phi, 0, \infty)$ and let $\eta$ be a curve from $0$ to $\infty$ given by the concatenation of independent $\SLE_{\kappa}(\frac\kappa2-4; \frac\kappa2-4)$ curves in each connected component of $D$. Parametrize $\eta$ by quantum length. Let $\partial_\ell D$ (resp.\ $\partial_r D$) be the boundary arc of $D$ from $0$ to $\infty$ in the clockwise (resp.\ counterclockwise) direction.  For each $s>0$, let $p_\ell(s)$ be the point on $\partial_\ell D \cap \eta([0,s])$ furthest from $0$ on $\partial_\ell D$, and let $I_\ell^-(s)$ be the boundary arc from $0$ to $p_\ell(s)$ on $\partial_\ell D$. Likewise define $p_r(s), I_r^-(s)$.  Let $D(s)$ be the union of $I_\ell^-(s) \cup I_r^-(s) \cup \eta([0,s])$ and the bounded connected components of its complement in $\bbH$. Define 
\[\cS_s = (D(s), \phi, \eta|_{[0,s]}, 0, \eta(s), p_\ell(s), p_r(s))/{\sim_\gamma}\] and let $\partial_\ell^- \cS_s, \partial_r^- \cS_s$ be the boundary arcs parametrized by $I_\ell^-(s), I_r^-(s)$. Let $\mathrm{FL}_\kappa^2$ denote the law of the process $(\cS_s)_{s > 0}$. By Proposition~\ref{prop-mate-forests}, this process describes the incremental mating of a pair of independent forested lines, where $\cS_s$ is the resulting decorated beaded quantum surface at time $s$.

We now explain the $\kappa \in (4,8)$ LCFT zipper.  Let $n \geq 0$, let $(\alpha_j, z_j) \in \R\times \ol {\mathbb H}$ such that $z_1, \dots, z_n$ are distinct, and let $\delta \in \R$. 
Sample $(\phi_0, (\cS_s)_{s >0}) \sim \LF_{\mathbb H}^{(-\frac1{\sqrt\kappa}, 0), (\alpha_j, z_j)_j, (\delta, \infty)} \times \mathrm{FL}^2_\kappa$. Suppose $s>0$ is such that the quantum lengths of $\partial_\ell^-\cS_s$ and $\partial_r^- \cS_s$ are less than $\cL_{\phi_0}((0, \infty))$ and $\cL_{\phi_0}((-\infty, 0))$ respectively. Consider the conformal welding of $\cS_s$ and $(\bbH, \phi_0, 0, \infty)/{\sim_\gamma}$ according to quantum boundary length, where the first marked points of the two quantum surfaces are identified, and the whole boundary arc $\partial_\ell^- \cS_s \cup \partial_r^- \cS_s$ of $\cS_s$ is welded to part of the boundary of $(\bbH, \phi_0, 0, \infty)/{\sim_\gamma}$. This gives a quantum surface decorated with a curve which comes from $\cS_s$. Embed the curve-decorated quantum surface via the hydrodynamic normalization to get $({\mathbb H}, \hat \phi_s, \hat \eta_s)$. That is, if $\hat g_s$ is the conformal map from ${\mathbb H}$ to the unbounded connected component of ${\mathbb H}\backslash \hat \eta_s$ satisfying $\lim_{z \to \infty} \hat g_s(z) - z = 0$, then 
$\hat g_s^{-1} \bullet_\gamma \hat \phi_s = \phi_0$. 

Though we have constructed $(\hat \phi_s, \hat \eta_s, \hat g_s)$ without explicitly using forested lines, since $(\cS_s)_{s > 0}$ corresponds to the incremental mating of a pair of independent forested lines, this process can equivalently be defined by sampling $\phi_0 \sim \LF_{\mathbb H}^{(-\frac1{\sqrt\kappa}, 0), (\alpha_j, z_j)_j, (\delta, \infty)}$ and gluing an independent pair of independent forested lines to the boundary of $(\bbH, \phi_0)$ (as in \cite[Definition 1.4.6]{wedges}), then incrementally mating the pair of forested lines. See  Figure~\ref{fig-zipper-looptree} (middle, right).

We reparametrize the process $(\hat \phi_s, \hat \eta_s, \hat g_s)$ according to half-plane capacity to get $(\phi_t, \eta_t, g_t)$ such that $\mathrm{hcap}(\eta_t) = 2t$. 
Let $W_t$ be the endpoint of $\eta_t$ lying in $\R$. 
The \emph{continuation threshold} is the first time $t\leq \infty$ that any neighborhood of $W_t$ has infinite quantum length. 

\begin{figure}
	\begin{center}
		\includegraphics[scale=0.38]{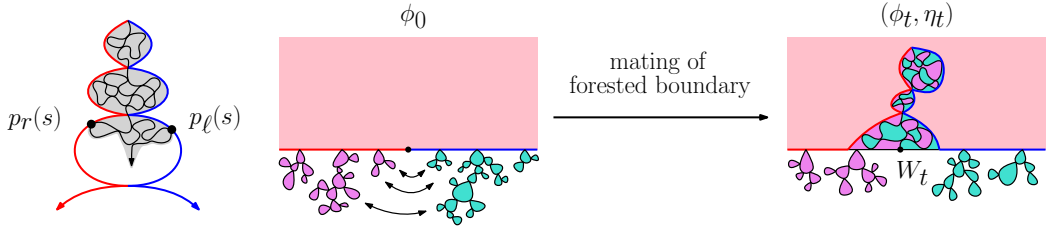}
	\end{center}
	\caption{\label{fig-zipper-looptree} Let $\kappa \in (4,8)$ and $\gamma = \frac4{\sqrt\kappa}$. \textbf{Left:} To define the process $(\cS_s)_{s > 0}$, sample a weight $2-\frac{\gamma^2}2$ quantum wedge and let $\eta$ be the concatenation of an independent $\SLE_\kappa(\frac\kappa2 - 4; \frac\kappa2-4)$ curve in each bead. Parametrize $\eta$ by quantum length. At time $s$, the curve-decorated quantum surface $\cS_s$ corresponds to the shaded region.  \textbf{Middle:} We sample a Liouville field $\phi_0$ and attach independent forested lines to the left and right boundary arcs. Marked points are not depicted. \textbf{Right:} Mating the forested lines corresponds to ``zipping up the quantum zipper''. The curve $\eta_t$ is the interface in ${\mathbb H}$ between the purple and green forests. 
		The conformal map $g_t$ sends  the pink region on the left to that on the right and satisfies $\lim_{z \to \infty}g_t(z) -z=0$. }
\end{figure}

\begin{theorem}\label{thm-forest-zipper} For $\kappa \in (4, 8)$ and $\gamma = \frac4{\sqrt\kappa}$, the conclusions of Theorems~\ref{thm-simple-zipper} and~\ref{thm-simple-zipper-rho} hold for the  process $(\phi_t, \eta_t)$ constructed immediately above.
\end{theorem}

\subsection{The $\kappa > 8$ LCFT  zipper} \label{sec-intro-sf}
Let $\kappa \geq 8$ and $\gamma = \frac4{\sqrt\kappa}$. Using  Proposition~\ref{prop-mot}, we first construct a process of decorated quantum surfaces $(\cC_s)_{s>0}$ which corresponds to incrementally mating a correlated pair of continuum random trees. 

In the setting of Proposition~\ref{prop-mot}, for $s > 0$ the intersection $\eta([0,s]) \cap \R$ is an interval which we call $[p_\ell(s), p_r(s)]$. Let $\cC_s := (\eta([0,s]), \phi, \eta|_{[0,s]}, 0, \eta(s), p_\ell(s), p_r(s))/{\sim_\gamma}$; this is a quantum surface decorated by a curve and four marked points. Let $\partial_{\ell}^- \cC_s$ and $\partial_{\ell}^+ \cC_s$ be the successive boundary arcs from $0$ to $p_\ell(s)$ to $\eta(s)$, and likewise define $\partial_r^- \cC_s$ and $\partial_r^+ \cC_s$. By definition the quantum lengths of $\partial_\ell^- \cC_s, \partial_\ell^+ \cC_s, \partial_r^- \cC_s, \partial_r^+ \cC_s$ are the quantities $Y_s^-, Y_s^+, X_s^-, X_s^+$ defined in Proposition~\ref{prop-mot}. Let $\mathrm{CRT}^2_\kappa$ denote the law of the process $(\cC_s)_{s \geq 0}$. This process described the incremental mating of the continuum random trees associated to $(X_\cdot, Y_\cdot) = (X_\cdot^+ - X_\cdot^-, Y_\cdot^+ - Y_\cdot^-)$, where $\cC_s$ is the resulting curve-decorated surface at time $s$.

We can now state the LCFT zipper for $\kappa >8$. Suppose $\kappa\geq 8$ and $\gamma = \frac4{\sqrt\kappa}$. 
	Sample $(\phi_0, (\cC_s)_{s \geq 0}) \sim \LF^{(-\frac1{\sqrt\kappa}, 0), (\alpha_j, z_j)_j, (\delta, \infty)} \times \mathrm{CRT}^2_\kappa$. For $s > 0$ such that  $\cL_{\phi_0}((-\infty, 0))$ and $\cL_{\phi_0}((0,\infty))$ are greater than the quantum lengths of $\partial_r^- \cC_s$ and $\partial_\ell^- \cC_s$	 respectively, consider the conformal welding of $\cC_s$ to $(\bbH, \phi_0, 0, \infty)/{\sim_\gamma}$ according to quantum length, where the first marked points of the quantum surfaces are identified, and the whole boundary arc $\partial_\ell^- \cC_s \cup \partial_r^- \cC_s$ is welded to part of the boundary of  $(\bbH,\phi_0,0,\infty)/{\sim_\gamma}$. This gives a quantum surface decorated by a curve from $\cC_s$. Embed the resulting curve-decorated quantum surface via the hydrodynamic normalization to get $({\mathbb H}, \hat \phi_s, \hat \eta_s)$. That is, if $\hat g_s$ is the conformal map from ${\mathbb H}$ to the unbounded connected component of ${\mathbb H}\backslash \hat \eta_s$ satisfying $\lim_{z \to \infty} \hat g_s(z) - z = 0$, then 
$\hat g_s^{-1} \bullet_\gamma \hat \phi_s = \phi_0$.

Though the process $(\hat \phi_s, \hat \eta_s, \hat g_s)$ was defined in a way that does not mention continuum random trees, since $(\cC_s)_{s > 0}$ corresponds to the incremental mating of a pair of correlated continuum random trees, this process can equivalently be defined by sampling $\phi_0 \sim \LF_{\mathbb H}^{(-\frac1{\sqrt\kappa}, 0), (\alpha_j, z_j)_j, (\delta, \infty)}$ and gluing an independent pair of correlated continuum random trees to the boundary of $(\bbH, \phi_0)$ according to quantum length measure, then incrementally mating the pair of trees. See  Figure~\ref{fig-zipper-sf}.

We reparametrize the process $(\hat \phi_s, \hat \eta_s, \hat g_s)$ according to half-plane capacity to get $(\phi_t, \eta_t, g_t)$ such that $\mathrm{hcap}(\eta_t) = 2t$. Let $W_t$ be the endpoint of $\eta_t$ lying in $\R$. 
The continuation threshold is the first time $t$ that any neighborhood of $W_t$ in $\R$ has infinite quantum length.

\begin{figure}
	\begin{center}
		\includegraphics[scale=0.38]{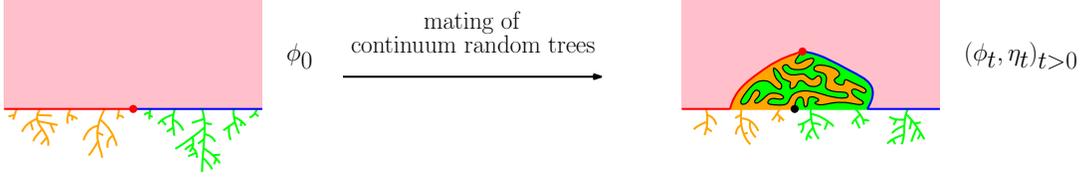}
	\end{center}
	\caption{\label{fig-zipper-sf} Let $\kappa \geq 8$ and $\gamma = \frac4{\sqrt\kappa}$. \textbf{Left:} We sample a Liouville field $\phi_0$ and attach correlated continuum random trees to the left and right boundaries. Marked points are not depicted. \textbf{Right:} Mating the trees corresponds to ``zipping up the quantum zipper''. The curve $\eta_t$ is the interface in ${\mathbb H}$ between the orange and green trees. 
		The conformal map $g_t$ sends the pink region on the left to that on the right and satisfies $\lim_{z \to \infty} g_t(z)-z = 0$. }
\end{figure}

\begin{theorem}\label{thm-sf-zipper} For $\kappa > 8$ and $\gamma = \frac4{\sqrt\kappa}$, the conclusions of Theorems~\ref{thm-simple-zipper} and~\ref{thm-simple-zipper-rho} hold for the  process $(\phi_t, \eta_t)$ constructed immediately above.
\end{theorem}

\begin{remark}\label{rem-not-proved}
	Theorem~\ref{thm-sf-zipper} should also hold for $\kappa = 8$ but our argument does not apply to this critical value, see Remark~\ref{rem-difficulty}. We leave the $\kappa = 8$ LCFT  zipper as an open problem. 
\end{remark}

\begin{remark}
	The mating-of-trees theorem Proposition~\ref{prop-mot} was also proved for the regime $\kappa \in (4,8)$ and $\gamma = \frac4{\sqrt\kappa}$, see Proposition~\ref{prop-mot-forr} for its statement. In this regime we can analogously  define $(\cC_s)$, but a.s.\ $\cC_s$ will not be simply connected for all $s$, 
	 and the conformal welding of $\cC_s$ to $(\bbH, \phi_0, 0, \infty)/{\sim_\gamma}$ will only be simply connected on the (fractal, measure zero) subset of times $\{ s \::\: X_s^+ = Y_s^+ = 0\}$. Restricting to these times and replacing the space-filling curves with non-space-filling curves  measurable with respect to them, the result is the process described by Theorem~\ref{thm-forest-zipper}. In this sense, the continuum random tree construction is equivalent to the forested line construction. See  Section~\ref{subsec-bpz-48} and Figure~\ref{fig-mot-forr} for more details.	
\end{remark}

\subsection{BPZ equation for Liouville conformal field theory}

We now give an application of the quantum zipper to LCFT, by proving a Belavin-Polyakov-Zamolodchikov (BPZ) equation for correlation functions with a degenerate boundary insertion. 

In this section we distinguish bulk and boundary insertions with different letters. Let $m, n \geq 0$. Let $(\alpha_j, z_j) \in \R \times {\mathbb H}$ for $j \leq m$ and assume the $z_j$ are distinct. Let $-\infty = x_0 < x_1 < \dots < x_n < x_{n+1} = +\infty$ be boundary points, let $\beta_1, \dots, \beta_n \in \R$, and let $\delta \in \R$. Let $\beta_* \in \{-\frac\gamma2, -\frac2\gamma\}$.  

For $k = 0, \dots, n$ let $I_k = (x_k, x_{k+1})$. Fix an index $k_* \in \{0, \dots, n\}$ and let $w \in I_{k_*}$.
Let $I_L = (x_{k_*}, w)$ and $I_R = (w, x_{k_* + 1})$. 
See Figure~\ref{fig-bpz}.  For $k \neq k_*$ let $\mu_k \in {\mathbb C}$ satisfy $\Re \mu_k \geq 0$, and assume  $\mu_L, \mu_R \in {\mathbb C}$ satisfy $\Re \mu_L, \Re \mu_R\geq0$ and are defined in terms of $\sigma_L, \sigma_R\in {\mathbb C}$ as follows: 
\eqb\label{eq-cosmocoupling}
\mu_L = g(\sigma_L), \,\, \mu_R = g(\sigma_R) \quad \text{ where } g(\sigma) = \frac{\cos(\pi \gamma (\sigma - \frac Q2))}{\sqrt{\sin (\pi \gamma^2/4)}} \text{ and } \sigma_L - \sigma_R = \pm \frac{\beta_*}2.  
\eqe

\begin{figure}
	\begin{center}
		\includegraphics[scale=0.38]{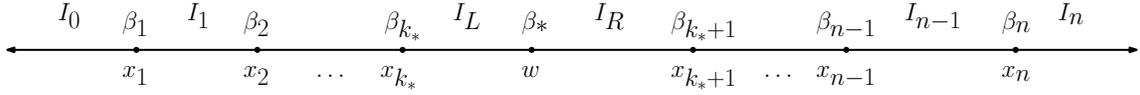}
	\end{center}
	\caption{\label{fig-bpz} Marked boundary points on $\partial {\mathbb H}$ and the intervals between them. 
		The bulk insertions are $(\alpha_j, z_j)_j$ (not depicted), the boundary insertions are $(\frac{\beta_*}2, w),(\frac\delta2, \infty)$ and  $(\frac{\beta_k}2, x_k)$ for $1 \leq k \leq n$.
		Each boundary interval $I_\bullet$ has an associated boundary cosmological constant $\mu_\bullet$. }
\end{figure}
Suppose the \emph{Seiberg bounds} hold:
\eqb\label{eq-seiberg}
\sum_{j} \alpha_j + \sum_k \frac{\beta_k}2 + \frac\delta2 + \frac{\beta_*}2 > Q, \qquad \alpha_j, \beta_k < Q \text{ for all }j,k, \quad  \delta <Q.
\eqe
We define the LCFT correlation function $F_{\beta_*}(w, (z_j)_j, (x_k)_k)$ to be 
\eqb\label{eq-corr}
\LF_{\mathbb H}^{(\frac{\beta_*}2, w), (\alpha_j, z_j)_j, (\frac{\beta_k}2, x_k), (\frac\delta2, \infty)}[\exp(-\cA_\phi({\mathbb H})- \sum_{\substack{0\leq k \leq n \\ k \neq k_*}} \mu_k \cL_\phi(I_k) - \mu_L \cL_\phi(I_L) - \mu_R \cL_\phi(I_R))].
\eqe
Because the Seiberg bounds hold, the integral~\eqref{eq-corr} converges absolutely \cite[Theorem 3.1]{hrv-disk}.

\begin{theorem}\label{thm-bpz}
	For $(\gamma, \beta_*) \neq (\sqrt2, -\frac\gamma2)$, the correlation functions $F_{\beta_*}$ are smooth and satisfy the BPZ equation
	\[\left(\frac1{\beta_*^2} \partial_{ww} + \sum_j (\frac1{w - z_j} \partial_{z_j} + \frac1{w-\ol z_j} \partial_{\ol z_j}) + \sum_k \frac1{w - x_k} \partial_{x_k} +\sum_j \Re \frac{2 \Delta_{\alpha_j}}{(w - z_j)^2} + \sum_k \frac{\Delta_{\beta_k}}{(w - x_k)^2} \right) F_{\beta_*} =0 .\]
\end{theorem}
If we assume the $\kappa = 8$ LCFT zipper not proved in this paper, the proof of Theorem~\ref{thm-bpz} carries through to the remaining case $(\gamma, \beta_*) = (\sqrt2, -\frac\gamma2)$. 

Here is a proof sketch. Consider the LCFT zipper with $\kappa = \frac4{\beta_*^2}$, and let $A_t, L_t, R_t$ be the quantum area and left and right quantum boundary lengths at time $t$. For $\beta_* = -\frac2\gamma$ and $\kappa \leq 4$, the coupling of $\sigma_L$ and $\sigma_R$ gives $\mu_L +\mu_R=0$, so $\mu_L L_t + \mu_R R_t = \mu_L (L_t - R_t)$. (Note that $\mu_L+ \mu_R = 0$ and $\mathrm{Re}(\mu_L), \mathrm{Re}(\mu_R) \geq 0$ force $\mu_L$ and $\mu_R$ to be purely imaginary.) This quantity is invariant under the conformal welding zipper process, giving rise to an $\SLE_\kappa$ martingale from which the BPZ equation is immediate.  For $\beta_* = -\frac\gamma2$ and $\kappa >4$, although $e^{-A_t - \mu_LL_t - \mu_R R_t}$ is not invariant, it evolves as a martingale  (Lemma~\ref{lem-bm-mtg}); the proof uses the mating-of-trees Brownian motion and requires the coupling~\eqref{eq-cosmocoupling}. This again gives an $\SLE_\kappa$ martingale and thus the BPZ equation. These arguments only yield that $F_{\beta_*}$ is a weak solution, but a hypoellipticity argument following \cite{dubedat-virasoro, pw-multiple-SLE} then implies that $F_{\beta_*}$ is smooth, completing the proof. 

In their pioneering work, Belavin, Polyakov and Zamolodchikov used representation theoretic methods to derive BPZ equations for the sphere \cite{bpz-conformal-symmetry}. These were recently mathematically proved for LCFT \cite{krv-local} using a subtle argument involving cancellations of not absolutely convergent integrals. See also \cite{remy-fb-formula, rz-gmc-interval, rz-boundary} for BPZ equations on the disk with bulk cosmological constant zero, that is, in~\eqref{eq-corr} the term $-\cA_\phi({\mathbb H})$ is removed. 

When the bulk cosmological constant is nonzero, the BPZ equations do not hold unless there is a coupling of the cosmological constants $\mu_L$ and $\mu_R$ via~\eqref{eq-cosmocoupling}. This was proposed by \cite{FZZ} after examining special cases, and by \cite{bb-hem-lcft} under the physical ansatz that in LCFT there exists only one primary field of a given conformal dimension.  Our SLE martingale argument gives an orthogonal conceptual explanation for~\eqref{eq-cosmocoupling}.

\begin{remark}
	The statement of Theorem~\ref{thm-bpz} was chosen for simplicity. Our argument is quite robust, and many of the conditions can be loosened. 
	\begin{itemize}
		\item We can choose  $\beta_k \geq Q$ if the boundary cosmological constants for the intervals adjacent to $x_k$ are zero. We can choose $\delta \geq  Q$ if $\mu_0 = \mu_n = 0$. 
		\item The condition that $\Re \mu_k, \Re \mu_L, \Re \mu_R \geq 0$ can be relaxed so long as~\eqref{eq-corr} converges absolutely.
		\item The bound $\sum_j \alpha_j + \sum_k \frac{\beta_k}2 + \frac\delta2 + \frac{\beta_*}2 >  Q$ is needed to ensure convergence of~\eqref{eq-corr}.  Sometimes the correlation function  can be defined even when~\eqref{eq-corr} is nonconvergent, by introducing \emph{truncations} where the exponential in~\eqref{eq-corr} is replaced by $e^z - 1$ or $e^z - 1- z$ for instance. See e.g.\ \cite[Proposition 3.4]{AHS-SLE-integrability}, 
		\cite[(1.7)]{ARS-FZZ} and \cite[Theorem B]{cercle-reflection-toda}. 
	\end{itemize}
\end{remark}

We note that the smoothness result of Theorem~\ref{thm-bpz} is itself rather nontrivial, and follows from the hypoellipticity argument of Dub\'edat \cite{dubedat-virasoro}.  \cite{krv-local} proved the bulk correlation functions are $C^2$ using mollifications of the GFF, and this was extended to $C^\infty$ by \cite{oikarinen-smoothness}. Using similar methods, \cite{cercle-hem} obtained expressions for weak first and second derivatives of boundary correlation functions. It is not clear to us whether these arguments can be extended to smoothness for boundary LCFT.

Finally, the \emph{higher equations of motion} for boundary LCFT developed in the physics literature by \cite{bb-hem-lcft} are closely related to the BPZ equation. Remarkably, shortly after the first version of this paper was uploaded, \cite{bw-hem-lcft} succeeded in proving a formulation of the level two higher equations of motion in terms of conformal blocks, and independently \cite{cercle-hem} obtained a formulation in terms of correlation functions and hence proved a generalization of Theorem~\ref{thm-bpz} without the assumption~\eqref{eq-cosmocoupling} and where the right hand side of the PDE is correspondingly nonzero \cite[Theorem 1.3]{cercle-hem}.

\subsection{Outlook}\label{sec-outlook}
We state here a few applications of the LCFT  zipper, and mention some future directions and open questions.

\subsubsection{Integrability of boundary LCFT}
The basic objects of conformal field theories are their correlation functions, and solving a conformal field theory means obtaining exact formulae for them. 
For the case of LCFT on surfaces without boundary, this was carried out in a series of landmark works that proved the  three-point structure constant equals the DOZZ formula proposed in physics \cite{krv-dozz}, then made rigorous the conformal bootstrap program of physicists to recursively solve for all correlation functions on all surfaces \cite{gkrv-bootstrap, gkrv-segal}. 

A similar program is currently being carried out for LCFT on surfaces with boundary. 
With Remy and Sun we computed the one-point bulk structure constant using mating-of-trees \cite{ARS-FZZ}. Taking the BPZ equation Theorem~\ref{thm-bpz} as input, with Remy, Sun and Zhu \cite{ARSZ-PT} we computed the boundary three-point structure constant and the bulk-boundary structure constant, making rigorous the formulas of Ponsot-Teschner~\cite{pt-formula} and Hosomichi~\cite{hosomichi} respectively. Our work provides all the initial data needed to compute all correlation functions on surfaces with boundary. The boundary conformal bootstrap  program has been initiated, see \cite{wu-bootstrap-annulus}.

\subsubsection{Reversibility of whole plane $\SLE_\kappa$ when $\kappa > 8$}
For $\kappa \in (0,8]$, chordal SLE is reversible in the following sense. 
Let $f: {\mathbb H}\to {\mathbb H}$ be a conformal automorphism with $f(0) = \infty$ and $f(\infty) = 0$. 
If $\eta$ is $\SLE_\kappa$ in ${\mathbb H}$ from $0$ to $\infty$, then the time-reversal of $f \circ \eta$ has the same law as $\eta$ up to time-reparametrization \cite{zhan-reversibility, ig3}. However, reversibility fails for  $\kappa > 8$ chordal $\SLE_\kappa$ \cite{ig4}.  

\emph{Whole-plane SLE} is a variant of SLE in ${\mathbb C}$ that starts at $0$ and targets $\infty$. Its reversibility was established by \cite{zhan-rev-whole-plane} for $\kappa \leq 4$ and \cite{ig4} for $\kappa \in (4,8]$. For $\kappa > 8$ the reversibility of whole-plane $\SLE_\kappa$  was conjectured in \cite{vw-loewner-kufarev} via reversibility of the $\kappa \to \infty$ large deviations rate function, which they obtained from a field-foliation coupling they interpret as describing a ``$\kappa \to \infty$ radial mating-of-trees''.
With Yu \cite{ay-reversibility}, we establish a radial mating-of-trees using the LCFT  zipper, then use mating-of-trees reversibility and LCFT reversibility to show whole-plane $\SLE_\kappa$ reversibility for $\kappa > 8$. 

\subsubsection{A general theory of conformal welding in LCFT}
It is known in many cases that conformally welding quantum surfaces described by LCFT produces a quantum surface also described by LCFT. This was first demonstrated for quantum wedges \cite{shef-zipper}, and many more cases followed \cite{wedges, hp-welding,  msw-cle-lqg, msw-non-simple,  ahs-disk-welding, AHS-SLE-integrability, ARS-FZZ, AS-CLE, ASY-triangle}. In all cases the welding interfaces are either chords or loops, and the surfaces have genus 0. In forthcoming work with Pu Yu, we will use the LCFT  zipper to obtain radial conformal weldings and subsequently develop a systematic framework that produces conformal weldings of quantum surfaces of arbitrary genus.

\subsubsection{Other BPZ equations in random conformal geometry}
Our argument uses the mating-of-trees framework to prove the boundary BPZ equation for LCFT on the disk. The BPZ equation for LCFT on the sphere has already been shown \cite{krv-local}, but can something similar to our proof of Theorem~\ref{thm-bpz} give an alternative proof?

The \emph{conformal loop ensemble (CLE)} \cite{shef-cle,shef-werner-cle} is a canonical conformally invariant collection of loops that locally look like SLE \cite{shef-cle}, and arises as the scaling limit of the collection of interfaces of statistical physics models. It is expected that CLE is described by a CFT -- indeed a suitably-defined CLE three-point function agrees with the \emph{generalized minimal model} CFT structure constant \cite{AS-CLE} -- so CLE multipoint functions should satisfy the BPZ equations. 
In the present work we obtained the boundary LCFT BPZ equation from a BPZ equation for SLE (in the sense that using It\^o calculus on the martingale of Lemma~\ref{lem-mtg} gives a second-order differential equation resembling BPZ), using mating-of-trees. One might hope to turn this argument around: could a BPZ equation for CLE be obtained from the BPZ equation for LCFT \cite{krv-local} using mating-of-trees? 

\medskip

\noindent\textbf{Organization of the paper.}
In Section~\ref{sec-prelim} we recall some preliminaries about LQG, LCFT, SLE and the mating-of-trees framework. In Section~\ref{sec-gff-partition} we explain that reverse $\SLE_{\kappa, \rho}$ is reverse $\SLE_\kappa$ weighted by a GFF partition function. In Section~\ref{sec-sheffield-zipper} we adapt a GFF quantum zipper of \cite{wedges} to obtain the $\kappa \leq 4$ LCFT zipper. In Section~\ref{sec-sf-zip} we prove the $\kappa > 8$ LCFT zipper, and in Section~\ref{sec-4-8} the $\kappa \in (4,8)$ LCFT zipper. Finally, we establish the LCFT boundary BPZ equations in Section~\ref{sec-bpz}.
\medskip 

\noindent\textbf{Acknowledgements.} 
We thank Guillaume Remy, Xin Sun and Tunan Zhu for earlier discussions on alternative approaches towards proving the boundary LCFT BPZ equations. 
We thank Guillaume Baverez, Yoshiki Fukusumi, Xin Sun, Yilin Wang,  Baojun Wu, Pu Yu and Dapeng Zhan for helpful discussions, and thank two anonymous referees for their valuable feedback. The author was supported by the Simons Foundation as a Junior Fellow at the Simons Society of Fellows.

\section{Preliminaries}\label{sec-prelim}
Our arguments fundamentally depend on non-probability measures. To facilitate the presentation, we will extend the terminology of probability theory to the non-probability setting. For a finite or $\sigma$-finite measure space $(\Omega, \cF, M)$, we say $X$ is a \emph{random variable} if it is a $\cF$-measurable function. The \emph{law} of $X$ is the push-forward measure $M_X := X_* M$, and we say that $X$ is \emph{sampled} from $M_X$. If $X$ is real-valued, we write $M[X] := \int X(\omega) M(d\omega)$. If $X$ is nonnegative, \emph{weighting by $X$} corresponds to considering the measure space $(\Omega, \cF, \wt M)$ where $\wt M \ll M$ is defined by $\frac{d\wt M}{dM}(\omega) = X(\omega)$. \emph{Conditioning} on an event $E \in \cF$ (with $0 < M(E) < \infty$) corresponds to considering the probability space $(E, \cF_E, M_E)$ where $M_E = M[E \cap \cdot] / M[E]$ and  $\cF_E= \{ A \cap E \: : \: A \in \cF\}$. Finally,  if $M$ is finite (i.e., $M(\Omega)<\infty$), define the probability measure $M^\# := M/M(\Omega)$. 

We note that most measures considered in this paper are non-probability measures, the exceptions being the Gaussian free field law $P_\bbH$ introduced in Section~\ref{subsec-gff}, and the laws of SLE and its variants. 

\subsection{The Gaussian free field and Liouville quantum gravity}\label{subsec-gff}
Let $m$ be the uniform probability measure on the unit half-circle $\{ z\: : \: z \in {\mathbb H}, |z| = 1\}$. The Dirichlet inner product is defined by $\langle f, g\rangle_\nabla = (2\pi)^{-1} \int_{\mathbb H} \nabla f \cdot \nabla g$. Consider the collection of smooth functions $f$ on $\ol \bbH$ with $\langle f, f\rangle_\nabla < \infty$ and $\int f(z) m(dz) = 0$, and let $H$ be its Hilbert space closure with respect to $\langle \cdot, \cdot \rangle_\nabla$. Let $(f_n)$ be an orthonormal basis of $H$ and let $(\alpha_n)$ be a collection of independent standard Gaussians. The summation 
\[h = \sum_n \alpha_n f_n \]
a.s.\ converges in the space of distributions (i.e., the strong dual space of smooth compactly supported functions on $\bbH$). Then $h$ is the \emph{free boundary Gaussian free field on ${\mathbb H}$ normalized so $\int h(z) m(dz) = 0$} \cite[Section 4.1.4]{wedges}.   We denote its law by $P_\bbH$; this is a probability measure on the space of distributions. (The GFF does have some stronger regularity but for our purposes the space of distributions suffices; see e.g.\ \cite[Remark 6.7]{berestycki-lqg-notes}.)
	One can also define free boundary Gaussian free fields with different normalizations.

Write $|z|_+ := \max ( |z|, 1)$. For $z,w \in \ol {\mathbb H}$ we define
\begin{equation}\label{eq-G}
	\begin{gathered}
		G_{\mathbb H}(z,w) =  -\log|z-w| - \log|z-\ol w| + 2\log |z|_+ + 2\log |w|_+, \\
		G_{\mathbb H}(z,\infty) = \lim_{w\to \infty} G_{\mathbb H}(z,w) =  2\log|z|_+.
	\end{gathered}
\end{equation}
The GFF $h$ is the centered Gaussian field with covariance structure formally given by ${\mathbb E}[h(z)h(w)] = G_{\mathbb H}(z,w)$. This is formal because $h$ is a distribution and so does not admit pointwise values, but for smooth compactly supported        functions $f,g$ on ${\mathbb H}$ we have ${\mathbb E}[(h, f) (h, g)] = \iint G_{\mathbb H}(z,w) f(z)g(w) \, dz\, dw$.

Suppose $\phi = h + g$ where $g$ is a (random) function on ${\mathbb H}$ which is continuous at all but finitely many points. Let $\phi_\eps(z)$ denote the average of $\phi$ on $\partial B_\eps(z) \cap {\mathbb H}$. 
For $\gamma \in (0,2)$, the $\gamma$-LQG area measure $\cA_\phi$ on ${\mathbb H}$ can be defined by the almost sure weak limit $\cA_\phi (dz) = \lim_{\eps \to 0} \eps^{\gamma^2/2} e^{\gamma \phi_\eps(z)} dz$ \cite{shef-kpz}. Similarly, the $\gamma$-LQG boundary length measure $\cL_\phi$ on $\R$ can be defined by $\cL_\phi(dx) := \lim_{\eps \to 0} \eps^{\gamma^2/4} e^{\frac\gamma2 \phi_\eps(x)}dx$. For the critical parameter $\gamma = 2$ a correction is needed to make the measure nonzero; we set $\cA_\phi (dz) = \lim_{\eps \to 0} (\log(1/\eps) - \phi_\eps(z)) \eps^2 e^{2 \phi_\eps(z)}dz$ and $\cL_\phi(dz) = \lim_{\eps \to 0} (\log(1/\eps) - \frac12 \phi_\eps(x)) \eps e^{\phi_\eps(x)}dx$ \cite{shef-deriv-mart}. See for instance \cite{lacoin-crit, powell-crit} for more details. 

A \emph{distribution modulo additive constant} on $\bbH$ is a continuous linear functional on the space of smooth compactly supported functions $f$ with $\int_\bbH f(z) \, dz = 0$. Alternatively, it is 
	an equivalence class of distributions where $\phi_1 \sim \phi_2$ whenever $\phi_1 - \phi_2 \equiv C$ for some constant $C$. The \emph{free boundary Gaussian free field  modulo additive constant} on $\bbH$ is the projection of $h \sim P_\bbH$ to the space of distributions modulo additive constant. See, e.g.,  \cite[Section 6.1]{berestycki-lqg-notes}.

\subsection{The Liouville field}\label{sec-LF} 
Let $\gamma \in (0,2]$ be the LQG parameter, and $Q = \frac\gamma2 + \frac2\gamma$. Write $|z|_+ := \max (|z|, 1)$. In this section we define some infinite measures on the space of distributions on $\bbH$. Recall that we extend the terminology of probability theory to the setting of non-probability measures (see start of Section~\ref{sec-prelim}), and in particular will use words such as ``sample'' and ``law'' for the infinite measures in the following definitions. 
\begin{definition}\label{def-LF-bare}
	Let $(h, \mathbf c)$ be sampled from $P_{\mathbb H} \times [e^{-Qc}\,dc]$ and let $\phi(z) = h(z) - 2Q \log |z|_+ + \mathbf c$. We call $\phi$ the \emph{Liouville field on ${\mathbb H}$} and denote its law by $\LF_{\mathbb H}$. 
\end{definition}

Define
\[ C_\gamma^{(\alpha, z)} =
\left\{
\begin{array}{ll}
	|z|_+^{-2\alpha(Q-\alpha)}(2 \Im z)^{-\alpha^2/2}  & \mbox{if } z \in {\mathbb H} \\
	|z|_+^{-2\alpha(Q-\alpha)}  & \mbox{if } z \in \R
\end{array}
\right. .\]
For $\delta \in \R$, $n \geq 0$ and $(\alpha_j, z_j) \in \R \times \ol {\mathbb H}$ for $j \leq n$ such that the $z_j$ are distinct, let 
\eqb \label{eq-C}
C_\gamma^{(\alpha_j, z_j)_j, (\delta, \infty)} = \prod_j C_\gamma^{(\alpha_j, z_j)} e^{\alpha_j \delta G_{\mathbb H}(z_j, \infty)} \times \prod_{1 \leq j < k \leq n} e^{\alpha_j \alpha_k G_{\mathbb H}(z_j, z_k)}.
\eqe
More generally, if the $(\alpha_j, z_j)$ are such that the $z_j$ are not distinct, we combine all pairs $(\alpha, z)$ with the same $z$ by summing their $\alpha$'s to get a collection $(\alpha'_j, z_j')$ where the $z_j'$ are distinct, and define $C_\gamma^{(\alpha_j, z_j)_j, (\delta, \infty)} = C_\gamma^{(\alpha_j', z_j')_j, (\delta, \infty)}$.
\begin{definition}\label{def-LF}
	For $\delta \in \R$, $n \geq 0$ and $(\alpha_j, z_j) \in \R \times \ol {\mathbb H}$, let $(h, \mathbf c)$ be sampled from $C^{(\alpha_j, z_j)_j, (\delta, \infty)}_\gamma P_{\mathbb H}\times [e^{(\sum_j \alpha_j + \delta - Q)c}\, dc]$, and set
	\[\phi(z) = h(z)  + \sum_j \alpha_j G_{\mathbb H}(z, z_j) +(\delta - Q) G_{\mathbb H}(z, \infty) + \mathbf c  .\]
	We call $\phi$ the  \emph{Liouville field with insertions $(\alpha_j, z_j)_j$, $(\delta, \infty)$} and we write $\LF_{\mathbb H}^{(\alpha_j, z_j)_j, (\delta, \infty)}$ for its law. 
\end{definition}
Note that $\LF_{\mathbb H}^{(\alpha_j, z_j)_j, (\delta, \infty)}$ depends implicitly on the parameter $\gamma$. 
When $\delta = 0$ there is no insertion at $\infty$, so we write $C_\gamma^{(\alpha_j, z_j)_j}$ and $\LF_{\mathbb H}^{(\alpha_j, z_j)_j}$ rather than $C_\gamma^{(\alpha_j, z_j)_j, (0, \infty)}$ and $\LF_{\mathbb H}^{(\alpha_j, z_j)_j, (0, \infty)}$. 

\begin{remark}\label{rem-neutrality}
	When $\sum_j \alpha_j + \delta = Q$, Definition~\ref{def-LF} has the following simplification. Sample $(h, \mathbf c) \sim 	C^{(\alpha_j, z_j)_j, (\delta, \infty)}_\gamma P_{\mathbb H} \times dc$, and let
	\eqb
	\phi(z) = h(z) + \sum_j \alpha_j G(z, z_j) + \mathbf c	
	\eqe
	where $G(z,w) = -\log|z-w| - \log |z- \ol w|$.  Then the law of $\phi$ is $\LF_{\mathbb H}^{(\alpha_j, z_j)_j, (\delta, \infty)}$. Moreover, we have the simplification $C^{(\alpha_j, z_j)_j, (\delta, \infty)}_\gamma = \cZ((\alpha_j, z_j)_j)$ with $\cZ$ defined in and below~\eqref{eq-Z0}. 
\end{remark}

\begin{remark}\label{rem-lf-notation}
	Let $m, n \geq 0$, let $(\alpha_j, z_j) \in \R \times {\mathbb H}$ for $j \leq m$, let $(\beta_k, x_k) \in \R \times \R$ for $k \leq n$, and let $\beta \in \R$.
	The measure that we call $\LF_{\mathbb H}^{(\alpha_j, z_j)_j, (\frac12\beta_k, x_k)_k, (\frac12\beta, \infty)}$ is instead called ``$\mathrm{LF}_{\mathbb H}^{(\alpha_j, z_j)_j, (\beta_k, x_k)_k, (\beta, \infty)}$'' in 
	\cite{AHS-SLE-integrability, ARS-FZZ}; there, the boundary insertions are described by their log-singularities (near $x_k$ the field blows up as $-\beta_k \log |\cdot - x_k|$), whereas our notation instead uses the Green function coefficient (near $x_k$ the field blows up like $\frac{\beta_k}{2} G_{\mathbb H}(\cdot, x_k)$). Our notation is more convenient for this paper since the Green function coefficient is invariant when a boundary  point is zipped into the bulk. 
\end{remark}

Finally, we will need the following rooted measure statement: sampling a point from the quantum area measure of a Liouville field is equivalent to adding a $\gamma$-insertion to the field.
\begin{proposition}\label{prop-pointed}
	Suppose $\gamma \in (0,2)$. 
	Let $n \geq 0$ and $(\alpha_j, z_j) \in \R \times \ol {\mathbb H}$ for $j \leq n$, and let $\delta \in \R$. Then
	\[ \cA_\phi(dz) \, \LF_{\mathbb H}^{(\alpha_j, z_j)_j, (\delta, \infty)}(d\phi) = \LF_{\mathbb H}^{(\alpha_j, z_j)_j, (\gamma, z),  (\delta, \infty)}(d\phi) \, dz.\]
\end{proposition}
Formally, Proposition~\ref{prop-pointed} holds because $\cA_\phi(dz) = \lim_{\eps \to 0} \eps^{\gamma^2/2} e^{\gamma \phi_\eps(z)} dz$ and \\$\LF_{\mathbb H}^{(\alpha_j, z_j)_j, (\gamma, z),  (\delta, \infty)}(d\phi)=\lim_{\eps \to 0} \eps^{\gamma^2/2} e^{\gamma \phi_\eps(z)} \LF_{\mathbb H}^{(\alpha_j, z_j)_j, (\delta, \infty)}(d\phi)$, where $\phi_\eps(z)$ is the average of $\phi$ on $\partial B_\eps(z) \cap {\mathbb H}$. A rigorous proof of the analogous statement for ${\mathbb C}$ is given in  \cite[Lemma 2.31]{AHS-SLE-integrability}; the proof of Proposition~\ref{prop-pointed} is identical so we omit it.

\subsection{Reverse Schramm-Loewner evolution}\label{subsec-sle}
In this section we recall some basic properties of reverse SLE.
Let $\kappa > 0$. 
Let $\rho_1, \dots, \rho_n \in \R$ and let $z_1, \dots, z_n \in \ol {\mathbb H}$ be distinct points. 
With $(B_t)_{t \geq 0}$ a standard Brownian motion, the driving function $(W_t)_{t \geq 0}$ for reverse $\SLE_{\kappa,\rho}$ is defined by the stochastic differential equations
\eqb\label{eq:SDE}
\begin{gathered}
	W_0 = 0, \quad dW_t = \sum_{j=1}^n \Re \left(\frac{-\rho_j}{Z^j_t- W_t}\right)\, dt + \sqrt\kappa\, dB_t,	\\
	Z^j_0 = z_j, \quad d Z^j_t = -\frac{2}{Z^j_t - W_t}\,dt  \quad \text{for }j = 1, \dots, n.
\end{gathered}
\eqe
For each $z \in {\mathbb H}$ and $s \geq 0$, let $(g_{s,t}(z))_{t \geq s}$ be the solution to $g_{s,s}(z) = z$, $\frac{d}{dt} g_{s,t}(z) = -\frac2{g_{s,t}(z) - W_t}$. This defines a family of conformal maps $g_{s,t}$; we also write $g_t := g_{0,t}$. 
For each $t >0$ we can define a curve $\eta_t:[0,t] \to \ol {\mathbb H}$ by $\eta_t(u) := \lim_{z \to W_{t-u}} g_{t-u, t}(z)$. We call the family of curves $(\eta_t)_{t \geq 0}$ \emph{reverse $\SLE_{\kappa}$ with force points at $z_j$ of weight $\rho_j$}.  Note that $g_t$ is the unique conformal map from ${\mathbb H}$ to the unbounded connected component of ${\mathbb H} \backslash \eta_t$ such that $\lim_{z \to \infty} g_t(z) - z = 0$, and that reverse SLE is parametrized by half-plane capacity in the sense that $\mathrm{hcap}(\eta_t) := \lim_{z \to \infty} z (z-g_t(z))$ equals $2t$. The curve $\eta_t$ is simple when $\kappa \leq 4$, self-intersecting (but not self-crossing) when $\kappa \in (4,8)$, and space-filling when $\kappa \geq8$.  

\begin{figure}
	\begin{center}
		\includegraphics[scale=0.8]{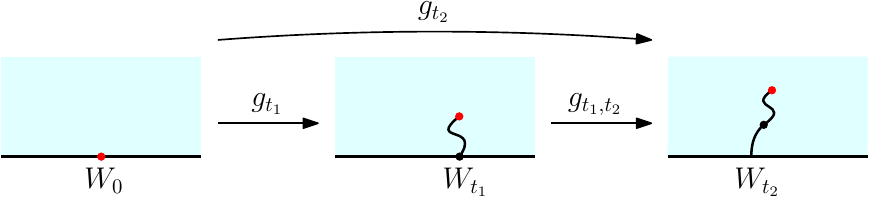}
		\caption{\label{fig-reverse-sle} Reverse $\SLE_\kappa$ for $\kappa \leq 4$; similar diagrams with different topologies can be drawn for $\kappa \in (4.8)$ and $\kappa \geq 8$. The curve grows from its base. As depicted, $\eta_t(0) = W_t$ and $\eta_t(t) = \lim_{z \to W_0} g_t(z)$.}
	\end{center}
\end{figure}

In the case of no force points ($n=0$), this process is well-defined for all time. For $n \geq 0$, by absolute continuity with respect to the $n=0$ case, the SDE~\eqref{eq:SDE} can be run until the time $\tau$ that $g_t(z_j) = W_t$ for some $j$.	Let $J_\tau = \{ j \: : \: g_\tau(z_j) = W_\tau\}$. If $\sum_{j \in J_\tau} \rho_j < \frac\kappa2 + 4$, then there is a unique way of continuing the process such that $g_t(z_j) \in {\mathbb H}$ for all $t > \tau$ \cite[Proposition 3.3.1]{wedges}. Applying this continuation procedure to each time a force point hits $W_t$, we see that reverse $\SLE_{\kappa,\rho}$ is well-defined until the \emph{continuation threshold}: the first time $\tau \leq \infty$ that the sum of the weights of force points hitting $W_\tau$ is at least $\frac\kappa2 + 4$.

\subsection{Quantum wedges and quantum cells}\label{subsec-wedges-cells}
In this section we define a number of special quantum surfaces, including the \emph{quantum wedges} introduced in \cite{shef-zipper, wedges}, and the \emph{quantum cell} which arises from partially mating correlated continuum random trees.

A quantum surface decorated by one or more curves and marked points is an equivalence class of tuples $(D, h, (\eta_i)_i, (z_j)_j)$ where $D \subset \C$ is a simply connected domain, $h$ is a distribution on $D$, each $\eta_i$ is a curve in $\ol D$, and each $z_j$ is a point in $\ol D$, under the equivalence relation $\sim_\gamma$ where $(D, h, (\eta_i)_i, (z_j)_j) \sim_\gamma (\wt D, \wt h, (\wt \eta_i)_i, (\wt z_j)_j)$ if there exists a conformal map $g: D \to \wt D$ such that $\wt h = g \bullet_\gamma h$, $\wt \eta_i = g \circ \eta_i$ and $\wt z_i = g(z_i)$ for all $i,j$. Here, $\bullet_\gamma$ is as defined in~\eqref{eq-coordinate-change}.

We now define \emph{beaded quantum surfaces}, which extend the notion of quantum surface to the setting of nontrivial topology. Consider pairs $(D,h)$ where $D \subset \C$ is a closed set such that each connected component of the interior of $D$ with its prime end boundary is homeomorphic to the closed disk, and $h$ is a distribution defined on the interior of $D$. We say $(D, h) \sim_\gamma(\wt D, \wt h)$ if there is a homeomorphism $g: D \to \wt D$ which is conformal on each component of the interior of $D$ and which satisfies $\wt h = 	g \bullet_\gamma h$. A beaded quantum surface is an equivalence class of pairs $(D, h)$ under $\sim_\gamma$. As before, we can define beaded quantum surfaces decorated by one or more curves and marked points, by specifying that $g$ should also identify these curves and points. 

We first define quantum wedges. For symmetry reasons it is easier to work in the strip $\cS := \R \times (0,\pi)$ than in ${\mathbb H}$. Let $m$ be the uniform probability measure on the segment $\{ 0\} \times (0,\pi)$. 
As in Section~\ref{subsec-gff}, consider the space of smooth functions $f$ on $\overline \cS$ with  $\langle f, f \rangle_{\nabla}< \infty$ and $\int f (z) m (dz) = 0$, and let $H$ be the Hilbert space closure with respect to $\langle \cdot, \cdot \rangle_{\nabla}$. As before, let $h = \sum_j \alpha_j f_j$ where the $\alpha_j$ are i.i.d.\ standard Gaussians and $\{f_j\}$ is an orthonormal basis for $H$. We call $h$ the \emph{GFF on $\cS$ normalized so $\int h (z) m(dz) = 0$}, and denote its law by $P_\cS$. 

We can decompose $H = H_\mathrm{av} \oplus H_\mathrm{lat}$ where $H_\mathrm{av}$ (resp.\ $H_\mathrm{lat}$) is the subspace of functions which are constant (resp.\ have mean zero) on $\{ t \} \times (0,\pi)$ for all $t \in \R$; this gives a decomposition $h = h_\mathrm{av} + h_\mathrm{lat}$ of $h$ into independent components. 

We now define thick quantum wedges. These are $\gamma$-LQG surfaces with the half-plane topology, and have a \emph{weight} parameter $W \geq \frac{\gamma^2}2$. 
\begin{definition}[Thick quantum wedge]\label{def-wedge}
	For $W \geq \frac{\gamma^2}2$, let $\alpha = \frac12(Q + \frac\gamma2 - \frac W\gamma)$. 
	Let 
	\[Y_t = \left\{
	\begin{array}{ll}
		B_{2t} + (Q - 2\alpha)t  & \mbox{if } t \geq 0 \\
		\wt B_{-2t} + (Q - 2\alpha)t & \mbox{if } t < 0
	\end{array}
	\right. \]
	where $(B_s)_{s \geq 0}$ is standard Brownian motion, and independently $(\wt B_s)_{s \geq 0}$ is standard Brownian motion conditioned on $\wt B_{2s} - (Q - 2\alpha)s < 0$ for all $s>0$. Let $\phi_\mathrm{av}(z) = Y_{\Re z}$, let $\phi_\mathrm{lat}$ be the projection of an independent free boundary GFF onto $H_\mathrm{lat}$, and let $\phi_0 = \phi_\mathrm{av}+ \phi_\mathrm{lat}$. We call the quantum surface $(\cS, \phi_0, -\infty, +\infty)/{\sim_\gamma}$ a \emph{weight $W$ quantum wedge} and write $\cM^\mathrm{wed}(W)$ for its law. 
\end{definition}

Thin quantum wedges are an ordered collection of two-pointed quantum surfaces.
\begin{definition}[Thin quantum wedge]\label{def-thin-wedge}
	For $W \in (0, \frac{\gamma^2}2)$, sample a Bessel process of dimension $\delta = 1 + \frac{2W}{\gamma^2}$. It decomposes as a countable ordered collection of Bessel excursions; for each excursion $e$ we construct a two-pointed quantum surface $\cB_e = (\cS, \phi^e, -\infty, \infty)/{\sim_\gamma}$ as follows. Let $(Y_{t}^e)_{t \in \R}$ be any time-reparametrization of $\frac2\gamma \log e$ which has quadratic variation $2dt$. Let $\phi_{\mathrm{av}}^e(z) = Y_{\Re z}^e	$, let $\phi^e_{ \mathrm{lat}}$ be the projection of a free boundary GFF (independent of everything else) onto $H_\mathrm{lat}$, and set $\phi^e = \phi_{\mathrm{av}}^e+ \phi_{\mathrm{lat}}^e$. Then the \emph{weight $W$ quantum wedge} is the ordered collection $(\cB_e)$. 
\end{definition}

Assume $\gamma \in (0,\sqrt2]$ and $\kappa = \frac{16}{\gamma^2}$. 
Let $({\mathbb H}, h, 0, \infty)$ be an embedding of a sample from $\cM^\mathrm{wed}(2 - \frac{\gamma^2}2)$. Let $\eta$ be an independent $\SLE_{\kappa}$ from $0$ to $\infty$; parametrize $\eta$ by quantum area so $\cA_h (\eta([0,a])) = a$ for all $a>0$. Let $D_a = \eta([0,a])$ and let $p = \eta(0)$ and $q = \eta(a)$. Let $x_L$ (resp.\ $x_R$) be the last point on the left (resp.\ right) boundary arc of ${\mathbb H}$ hit by $\eta$ before time $a$. We note that since $\gamma \leq \sqrt2$ the domain $D_a$ is simply connected \cite[Remark 5.2]{ghs-mating-survey}.

\begin{definition}\label{def-quantum-cell}
	For $\gamma \in (0,\sqrt2]$, we call the  curve-decorated quantum surface 
	$\cC_a = (D_a, h, \eta|_{[0,a]}, p, q, x_L, x_R)/{\sim_\gamma}$ an \emph{area $a$ quantum cell}. We denote its law by $P_a$. 
\end{definition}
An analogous definition can be given when $\gamma \in (\sqrt2,2)$; in this regime  $\cM^\mathrm{wed}(2-\frac{\gamma^2}2)$ is a thin quantum wedge, and the corresponding $D_a$ is not simply connected \cite[Remark 5.2]{ghs-mating-survey} so the definition is slightly more complicated. We thus omit it.

We remark that the area $a$ quantum cell is the decorated quantum surface arising from mating a pair of correlated continuum random trees corresponding to the pair of Brownian motions $(X_t, Y_t)_{t \geq 0}$ as in Proposition~\ref{prop-mot} until the time when the quantum area is $a$.
Quantum cells are used in the construction of $(\phi_t, \eta_t)$ in Theorem~\ref{thm-sf-zipper}.

\begin{remark}
	When $(\gamma, \kappa) \neq (\sqrt2,8)$, the quantum cell $\cC_a$ is measurable with respect to $(D_a, h, \eta|_{[0,a]})/{\sim_\gamma}$, that is, the four marked points can be recovered from $(D_a, h, \eta|_{[0,a]})/{\sim_\gamma}$. First, we can recover two of the marked points $p = \eta(0)$ and $q = \eta(a)$. Let $\psi: D_a \to {\mathbb H}$ be a conformal map with $\psi(p) = 0$ and $\psi(q) = \infty$.
	From the imaginary geometry construction of space-filling $\SLE_\kappa$ \cite{ig4},  modulo time-parametrization the law of $\psi \circ \eta$ given $\psi(x_L)$ and $\psi(x_R)$ is $\SLE_{\kappa, \rho}$ with force points of size $\kappa/2 - 4$ at  $\psi(x_L)$ and $\psi(x_R)$. Since $\kappa/2-4\neq 0$ the points $\psi(x_L)$ and $\psi(x_R)$ are measurable with respect to $\psi \circ \eta$, so we can recover $x_L$ and $x_R$ also. In this paper, we will not consider the $(\gamma, \kappa) = (\sqrt2, 8)$ case, so for notational simplicity we will frequently omit the marked points. 
\end{remark}

Since the duration $a$ Brownian motion $(X_t, Y_t)_{[0,a]}$ is reversible, it is unsurprising that the area $a$ quantum cell is reversible:

\begin{lemma}\label{lem-rev-cell}
	$\cC_a$ is symmetric in law in the following sense. 
	If $\wt \eta$ is the time-reversal of $\eta|_{[0,a]}$, then $(D_a, h, \eta|_{[0,a]}, p, q, x_L, x_R)/{\sim_\gamma}$ and $(D_a, h, \wt\eta, q, p, x_R, x_L)/{\sim_\gamma}$ have the same law. (Here, the second quantum surface is obtained from the first by reversing the direction of the curve $\eta|_{[0,a]}$ to get  $\wt \eta$ and switching the marked points $p \leftrightarrow q$ and $x_L \leftrightarrow x_R$.)
\end{lemma}
\begin{proof}
	\cite[Theorem 1.4.1]{wedges} states the following. Sample a  \emph{$\gamma$-quantum cone embedded as $({\mathbb C}, h, 0, \infty)$}  and an independent \emph{SLE$_\kappa$ curve $\eta$ in $\C$ from $\infty$ to $\infty$} (see \cite[Page 27, Footnote 4]{wedges} for the definition of this variant of $\SLE_\kappa$).  Choose the monotone reparametrization of $\eta$ such that $\eta(0) = 0$ and the quantum area of $\eta([a,b])$ is $b-a$ for all $a < b$.
Then the quantum surface $(\eta((0,\infty)), h, 0, \infty)/{\sim_\gamma}$ has law $\cM^\mathrm{wed}(2-\frac{\gamma^2}2)$. Thus $(\eta([0,a]), h, \eta|_{[0,a]})/{\sim_\gamma}$ has law $P_a$. 

Let $z = \eta(a)$, then  \cite[Theorem 1.4.1]{wedges} states that $({\mathbb C}, h(\cdot + z), \eta(\cdot + a) - z)/{\sim_\gamma}$ also has the law of a $\gamma$-quantum cone decorated by an independent $\SLE_\kappa$ from $\infty$ to $\infty$, and by the reversibility of this $\SLE_\kappa$ curve, the decorated quantum surface $({\mathbb C}, h(\cdot + z), \eta(a - \cdot) - z)/{\sim_\gamma}$ also has this law. Thus, $(\eta([0,a]), h, \tilde \eta)$ also has law $P_a$, where $\tilde \eta: [0,a] \to {\mathbb C}$ is given by $\tilde \eta(t) = \eta(a-t)$. 
\end{proof}

\subsection{Conformal welding} \label{sec-conf-weld}

In this section we explain what a conformal welding is, and why the conformal weldings of quantum surfaces in this paper exist and are unique. See \cite[Section 3.5]{wedges} for further details. 

Given a pair of topological disks  with boundary $D_1, D_2$ and a homeomophism identifying their boundaries, one can obtain a topological surface $S$ by taking the union of the disks quotiented by the homeomorphism, and this surface has a curve $\eta$ given by the image of the boundaries of the disks. 
Further suppose that each disk $D_i$ comes with a conformal structure, thus endowing $S \backslash \eta$ with a conformal structure. A \emph{conformal welding} is any way of extending this conformal structure to the whole surface $S$. 
The notion of conformal welding can similarly be defined for the configurations we consider, e.g., welding a pair of domains along distinguished boundary arcs, or welding two boundary arcs of the same domain. 
It is possible that no conformal welding exists, or that multiple conformal weldings exist. 

In this paper we consider conformal weldings of two boundary arcs with the homeomorphism specified by their respective LQG boundary measures, where the fields are GFFs plus log singularities at finitely many marked points.
We explain why, in this setting, the conformal welding exists and is unique. 

If there are no marked points, then for $\gamma < 2$ the existence and uniqueness are due to \cite{shef-zipper}, and for $\gamma = 2$ the existence was shown by \cite{hp-welding} and uniqueness was demonstrated by \cite{kms-sle4}. 
Next, consider the case where we have marked points. Let $S$ be the topological surface obtained from quotienting by the homeomorphism, and $\eta$ the image of the boundary arcs being welded. We have a conformal structure on $S \backslash \eta$ and want to show it extends uniquely to a conformal structure on $S$. By local absolute continuity with respect to the first case, if $\eta' \subset \eta$ is the image of boundary intervals containing no marked points, then the conformal structure of $S \backslash \eta$ uniquely extends to a conformal structure on $S \backslash (\eta \backslash \eta')$. Taking limits, we conclude that if $A$ is the image of the set of marked points, then the conformal structure on $S \backslash \eta$ uniquely extends to a conformal structure on $S \backslash A$. Since $A$ is finite, the conformal structure uniquely extends to $S$ as desired. 

\section{Reverse $\SLE_{\kappa, \rho}$ and the GFF partition function}\label{sec-gff-partition}
We make no claim of originality for the material in this section, and present it here for completeness. See the end of this section for references.

Let $n \geq 0$ and let $(\alpha_j, z_j) \in \R \times \ol {\mathbb H}$ for $j \leq n$, where the $z_j$ are distinct. Let $w \in \R$. Let $\kappa > 0$, $\gamma = \min (\sqrt \kappa, \frac 4{\sqrt \kappa})$. The GFF partition function $\cZ$ defined in~\eqref{eq-Z0}, times a factor arising from uniformizing in ${\mathbb H}$,  is a martingale observable for reverse SLE:
\begin{lemma}\label{lem-mtg}
	Let $\kappa > 0$. Let $n \geq 0$, and let $(\alpha_j, z_j) \in \R \times \ol{\mathbb H}$ for $j\leq n$.  Let $I = \{j : z_j \in {\mathbb H}\}$ and $B = \{j: z_j \in \R\} $. Sample reverse $\SLE_\kappa$ with no force points (as in~\eqref{eq:SDE}) and define
	\eqb\nonumber
	M_t = \prod_{i \in I} |g_t'(z_i)|^{2 \Delta_{\alpha_i}} \prod_{b \in B} |g_t'(z_b)|^{\Delta_{2\alpha_b}} \cZ\big((-\frac1{\sqrt\kappa}, W_t), (\alpha_j, g_t(z_j))_j \big) 
	\eqe
	where $\Delta_\alpha := \frac\alpha2 (Q - \frac\alpha2)$ and $Q = \frac\gamma2 + \frac2\gamma$.
	Let $\tau$ be a stopping time such that almost surely $g_t(z_j)\neq W_t$ for all $t \leq \tau$. Then $M_{t\wedge \tau}$ is a martingale. 
\end{lemma}

\begin{proof}
	To lighten notation, we assume there are only bulk insertions (i.e. $z_j \in {\mathbb H}$ for all $j$);  the general case follows from the same arguments. We write $Z_{j,t} = g_t(z_j)$. Our goal is to compute $d \log M_t$. Recalling $G(p,q) := -\log|p-q| - \log |p - \ol q|$, we have
	\begin{align}\label{eq-logMt}
		\log M_t = &\sum_j 2 \Delta_{\alpha_j} \log |g_t'(z_j)| - \frac{\alpha_j^2}2 \log (2 \Im Z_{j,t}) 
		\\&+ \frac{2\alpha_j}{\sqrt\kappa} \log |Z_{j,t} - W_t|  + \sum_{1 \leq j < k \leq m} \alpha_j \alpha_k G(Z_{j,t}, Z_{k,t}). \nonumber
	\end{align}
	By the definition of the reverse Loewner flow and $Z_{j,t} = g_t(z_j)$, we have
	\[
	\begin{gathered}
		dW_t = \sqrt \kappa dB_t, \enspace d\langle W_t \rangle = \kappa dt,  \\
		\enspace dZ_{j,t} = -\frac{2dt}{Z_{j,t} - W_t} , \enspace d \ol Z_{j,t} = -\frac{2dt}{\ol Z_{j,t} - W_t}, \enspace d \log |g_t'(z_j)| = \Re \frac{2dt}{(Z_{j,t} - W_t)^2}. 
	\end{gathered}\]
	The last identity holds since
	$d (g_t(z)) = -\frac{2}{g_t(z) - W_t}dt$ implies  $d(g_t'(z)) = -\frac2{(g_t(z)-W_t)^2} g_t'(z)dt$. 
	Using these we have 
	\[d \log (2 \Im Z_{j,t}) = d \log (Z_{j,t} - \ol Z_{j,t}) = \frac1{Z_{j,t} - \ol Z_{j,t}} dZ_{j,t} + \frac1{\ol Z_{j,t} - Z_{j,t}} d \ol Z_{j,t} = \frac1{|Z_{j,t} - W_t|^2}dt\] and, since $\log |Z_{j,t} - W_t| = \Re  \log (Z_{j,t} - W_t) $, we have  \[
	d\log |Z_{j,t} - W_t| 
	= \Re \frac{\sqrt\kappa}{W_t - Z_{j,t}} dB_t  - \Re \frac{\sqrt\kappa Q}{(W_t-Z_{j,t})^2}dt.\]
	Thus
	\begin{align} \label{eq-appendix-comp}
		&d \left(2 \Delta_{\alpha_j} \log |g_t'(z_j)|  - \frac{\alpha_j^2}2 \log (2 \Im Z_{j,t}) + \frac{2\alpha_j}{\sqrt\kappa}  \log |Z_{j,t} - W_t| \right) \\ &=
		\Re \frac{2\alpha_j}{W_t - Z_{j,t}} dB_t -  (\Re \frac{1}{(Z_{j,t} - W_t)^2} + \frac{1}{|Z_{j,t} - W_t|^2})\alpha_j^2dt \nonumber\\
		&= \Re \frac{2\alpha_j}{W_t - Z_{j,t}} dB_t - 2\alpha_j^2 \left(\Re \frac1{Z_{j,t} - W_t}\right)^2 dt.  \nonumber
	\end{align}
	In the last equality, we used the identity $\Re (z^2) + |z|^2 = 2 (\Re z)^2$. Next, since $G(Z_{j,t}, z) = - \Re (\log (Z_{j,t} - z) + \log (Z_{j,t} - \ol z))$,  we have
	$d G(Z_{j,t}, z) = \Re\left( (\frac{1}{Z_{j,t} - z} + \frac1{Z_{j,t} - \ol z})\frac{2}{Z_{j,t} - W_t}\right)dt$, so 
	\alb
	d& G(Z_{j,t} , Z_{k,t}) = \Re \left((\frac{1}{Z_{j,t} - Z_{k,t}} + \frac1{Z_{j,t} - \ol Z_{k,t}})\frac{2}{Z_{j,t} - W_t}
	+
	(\frac{1}{Z_{k,t} - Z_{j,t}} + \frac1{Z_{k,t} - \ol Z_{j,t}})\frac{2}{Z_{k,t} - W_t}
	\right) dt \\
	&=  \Re \left(
	-\frac2{(Z_{j,t} - W_t)(Z_{k,t} - W_t)} - \frac2{(Z_{j,t} - W_t)(\ol Z_{k,t} - W_t)} \right) dt = -4 \Re \left(\frac1{Z_{k,t} - W_t} \right) \Re \left(\frac1{Z_{j,t}-W_t} 
	\right) dt.
	\ale
	Combining this with~\eqref{eq-appendix-comp}, we can compute $d \log M_t$ from~\eqref{eq-logMt}:
	\[d \log M_t = ( \sum_j \Re \frac{2\alpha_j}{W_t-Z_{j,t}} )d B_t - \frac12( \sum_j \Re \frac{2\alpha_j}{W_t-Z_{j,t}} )^2 dt. \]
	Thus, $M_{t\wedge \tau}$ is a martingale, completing the proof for the case of only bulk insertions. The general case is essentially the same: we get the same expression for $d \log M_t$ so $M_{t \wedge \tau}$ is a martingale. 
\end{proof}

\begin{proposition}\label{prop-sle-weighted}
	In the setting of Lemma~\ref{lem-mtg}, the law of reverse $\SLE_\kappa$ run until $\tau$ and weighted by $\frac{M_\tau}{M_0}$ is precisely that of  $\SLE_{\kappa, \rho}$ run until the stopping time $\tau$, where at each $z_j$ there is a force point of weight $\rho_j = 2\sqrt\kappa \alpha_j$. In other words, writing $\rSLE_{\kappa}^\tau$ (resp.\ $\rSLE_{\kappa, \rho}^\tau$) for the law of the reverse SLE$_\kappa$ (resp.\ reverse SLE$_{\kappa, \rho}$) curve at time $\tau$, we have the Radon-Nikodym derivative
	\[\frac{d\SLE_{\kappa, \rho}^\tau}{d\SLE_\kappa^\tau} =  \frac{M_\tau}{M_0}.\]
\end{proposition}
\begin{proof}
	We use the notation $\pZ(w, (z_j)_j) := \cZ((-\frac1{\sqrt\kappa}, w), (\alpha_j, z_j)_j )$. 
	By Girsanov's theorem, under the reweighted law, the driving function satisfies the SDE
	\[
	dW_t =  \sqrt\kappa\, dB_t + \kappa \partial_w \log \pZ(W_t, (g_t(z_j))_j) dt.
	\]
	We have $\partial_w \log \pZ(w, (z_j)_j) =  \sum_j -\frac2{\sqrt\kappa}\alpha_j \partial_w \Re \log (z_j - w) = \sum_j \Re \left(\frac{-2\alpha_j/\sqrt\kappa}{z_j - w} \right)$, so the SDE agrees with~\eqref{eq:SDE} as desired. 
\end{proof}

The observation that reverse SLE with force points is intimately related to the Coulomb gas partition function is not original. See for instance \cite{lawler-reverse-multifractal, fukusumi17, alekseev19, kk-multiple-sle,km-bdy-cft,  koshida-multiple-backward}. A similar perspective arose earlier in the setting of forward SLE. See e.g.\ \cite{bb-cft-sle,bauer, kytola-cft-sle}.

\section{Sheffield's coupling and the $\kappa \leq 4$ LCFT zipper}\label{sec-sheffield-zipper}

In this section we prove Theorems~\ref{thm-simple-zipper} and~\ref{thm-simple-zipper-rho}. These results are fairly straightforward consequences of Sheffield's coupling of LQG and SLE.

A version of the following was first proved in \cite{shef-zipper}, and the full version where force points may be zipped into the bulk was proved in \cite{wedges}. Recall $G(z,w) = -\log|z-w| - \log |z-\ol w|$. 
\begin{proposition}[{\cite[Theorem 5.1.1]{wedges}}]\label{prop-dms-chordal}
	Let $\kappa > 0$, $\gamma = \min(\sqrt\kappa, \frac4{\sqrt\kappa})$, let $n \geq 0$ and $(\rho_j, z_j) \in \R \times \ol {\mathbb H}$ for $j \leq n$. Suppose that $(W_t)_{t \geq 0}$ is the driving function for reverse $\SLE_{\kappa,\rho}$ with force points at $z_j$ of size $\rho_j$, and let $g_t$ be the Loewner map. For each $t \geq0$ let 
	\[ {\mathfrak h}_t(z) = -\frac1{\sqrt\kappa} G(z,W_t) + \frac1{2\sqrt\kappa} \sum_{j=1}^n  \rho_j G(g_t(z), g_t(z_j)) + Q \log |g_t'(z)|. \]
	Let $h$ be an independent free boundary GFF modulo additive constant on ${\mathbb H}$. Suppose $\tau$ is an a.s.\ finite stopping time 
	such that $\tau$ occurs before or at the continuation threshold for $W_t$. Then, as distributions modulo additive constant, 
	\eqb\label{eq:gff-zip}
	{\mathfrak h}_0 +h \stackrel d=  {\mathfrak h}_\tau + h \circ g_\tau. 
	\eqe
\end{proposition}
Note that if there is a force point at $0$ with weight $\rho \geq \frac\kappa2 + 4$ then the reverse $\SLE_\kappa$ immediately hits the continuation threshold, i.e.\ $\tau = 0$.

As we see next, adding a constant chosen from Lebesgue measure to the field essentially gives Theorem~\ref{thm-simple-zipper-rho} when $\delta = Q - \sum_j \alpha_j + \frac1{\sqrt\kappa}$. 
\begin{proposition}\label{prop-main}
	Let $\kappa > 0$ and $\gamma = \min(\sqrt\kappa, \frac4{\sqrt\kappa})$. Let $n\geq 0$ and let $(\alpha_j, z_j) \in \R \times \ol{\mathbb H}$ for $j \leq n$. Let $\delta = Q - \sum_j \alpha_j + \frac1{\sqrt\kappa}$.
	Let $(\eta_t)$ be reverse $\SLE_{\kappa, \rho}$ with force points at $z_j$ of size $\rho_j = 2\sqrt\kappa \alpha_j$, run until a stopping time $\tau$ which a.s.\ occurs before or at the continuation threshold. Let $g_\tau$ be the conformal map from ${\mathbb H}$ to the unbounded connected component of ${\mathbb H}\backslash \eta_\tau$ such that $\lim_{z \to \infty} g_\tau(z)-z = 0$.
	Then for any nonnegative measurable function $F$ on the space of distributions in $\bbH$, 
	\eqb\label{eq-lf-shef}
	\frac{\mathrm{LF}_{\mathbb H}^{(-\frac1{\sqrt\kappa},0), (\alpha_j, z_j)_j, (\delta,\infty)} [F(\phi) ]}
	{
		\cZ((-\frac1{\sqrt\kappa}, 0), (\alpha_j, z_j)_j)
	} = {\mathbb E} \left[ \frac{ \mathrm{LF}_{\mathbb H}^{(-\frac1{\sqrt\kappa}, W_\tau), (\alpha_j, g_\tau(z_j))_j, (\delta, \infty)} [F(g_\tau^{-1} \bullet_\gamma \phi)]}
	{
		\cZ((-\frac1{\sqrt\kappa}, W_\tau), (\alpha_j, g_\tau(z_j))_j)
	}
	\right].
	\eqe
	Here, the expectation is taken with respect to $(\eta_t)$.
\end{proposition}
Here~\eqref{eq-lf-shef} is understood in the sense of the extended real numbers, so the left and right hand sides may both equal $+\infty$. In particular we need no integrability assumption on $F$.
\begin{proof}
	We will use the Liouville field description from Remark~\ref{rem-neutrality}. 
	Write $\mathfrak f_\tau(z) = -\frac{1}{\sqrt\kappa} G(z,W_\tau) + \sum_j \alpha_j G(z, g_\tau(z_j))$. Writing ${\mathbb E}$ to denote expecation with respect to $\eta_\tau$  and independently $h \sim P_{\mathbb H}$, the right hand side of~\eqref{eq-lf-shef} equals
	\[{\mathbb E}[\int_\R F(g_\tau^{-1} \bullet_\gamma (h + \mathfrak f_\tau + c) ) \, dc ] = {\mathbb E}[\int_\R F(h \circ g_\tau + \mathfrak h_\tau + c) 	 \, dc ].\]
	By Proposition~\ref{prop-dms-chordal} this equals ${\mathbb E}[\int_\R F(h + \mathfrak h_0 + c)\, dc]$, which is the left hand side of~\eqref{eq-lf-shef}.
\end{proof}

To remove the constraint that $\delta = Q - \sum_j \alpha_j + \frac1{\sqrt\kappa}$, in Proposition~\ref{prop-reweight-infty} we will weight by the average of the field on very large semicircles to change the value of $\delta$, see Lemma~\ref{lem-reweight-infty}. As an intermediate step we need the following collection of identities. Recall $G_{\mathbb H}$ from~\eqref{eq-G} and $G(z,w) = -\log|z-w| - \log|z-\ol w|$. 
\begin{lemma}\label{lem-green-conformal-infty}
	Suppose $K \subset \ol {\mathbb H}$ is compact,  ${\mathbb H}\backslash K$ is simply connected and $g: {\mathbb H}\to {\mathbb H} \backslash K$ is the conformal map such that $\lim_{|z| \to\infty} g(z)-z = 0$. Let $u \in g(B_{1/\eps}(0) \cap {\mathbb H})$, let $\eps > 0$ and let $\theta_{\eps, \infty}$ be the uniform probability measure on $\{ z \in {\mathbb H} : |z| = 1/\eps\}$.  Suppose $g(-1/\eps), g(1/\eps) \in \R$ and $|g(z)| \geq 1$ whenever $|z| = 1/\eps$, then 
	\[
	\begin{gathered}
		\int G(u,v) (g_* \theta_{\eps, \infty})(dv) = 2\log \eps; \qquad \int \log g'(v) \theta_{\eps, \infty}(dv) = 0;
		\\
		\int G_{\mathbb H}(u, v)(g_*\theta_{\eps, \infty})(dv) = G_{\mathbb H}(u, \infty); \qquad \int G_{\mathbb H}(\infty, v) (g_*\theta_{\eps, \infty})(dv) = -2 \log \eps.
	\end{gathered}
\]
\end{lemma}
\begin{proof}
	Extend $g$ by Schwarz reflection to a map with image ${\mathbb C} \backslash (K \cup \ol K)$. Consider $a \in B_\eps(0) \backslash \{0\}$, then $a^{-1} \in \C \backslash B_{1/\eps}(0)$ so $g(a^{-1}) \in \C \backslash g(B_{1/\eps}(0))$. In particular, $g(a^{-1}) \neq u$ for all $a$. Thus the map $f:B_{\eps}(0) \backslash \{0\} \to {\mathbb C}$ defined by $f(a) = (u - g(a^{-1}))^{-2} a^{-2}$ is holomorphic. One can check that $f$ extends continuously to $f(0) = 1$, hence the extended map $f:B_{\eps}(0)\to {\mathbb C}$ is holomorphic. Let $\theta_{\eps, 0}$ be the uniform probabilitiy measure on $\partial B_{\eps}(0)$. Since $\log |f|$ is harmonic we have $(\log |f|, \theta_{\eps, 0}) = \log |f(0)| = 0$, and rephrasing gives $\int -2 \log|u - g(a^{-1})| \theta_{\eps, 0}(da) = 2 \log \eps$. Now the change of variables $v = a^{-1}$ gives the first assertion. 
	The second assertion  is proved similarly: the function $f:B_\eps(0) \backslash \{0\} \to \C$ given by  $f(a) := g'(a^{-1})$ is holomorphic, and since $\lim_{|z| \to\infty} g(z)-z = 0$, the function $f$ extends to a holomorphic function $f: B_\eps(0) \to \C$ with $f(0) = 1$. As before $0 = (\log|f|, \theta_{\eps, 0}) = \int \log g'(a^{-1}) \theta_{\eps, 0}(da)$, and the change of variables $v = a^{-1}$ gives the claim. The third and fourth assertions follows from $G_{\mathbb H}(u,v) = G(u,v) - G(0, v) + G_{\mathbb H}(u, \infty)$ and $G_{\mathbb H}(\infty, v) = -G(0, v)$ for $|v|\geq1$.
\end{proof}

As we now explain, weighting the Liouville field by the field average near $\infty$ changes the value of $\delta$: 
\begin{lemma}\label{lem-reweight-infty}
	Let $\gamma \in (0,2]$. 
	In the setting of Lemma~\ref{lem-green-conformal-infty}, suppose $(\alpha_j, w_j) \in \R \times \ol {\mathbb H}$ and $\delta, \delta' \in \R$. Let ${\mathbb H}_\eps = B_{1/\eps}(0)  \cap {\mathbb H}$. Then for any nonnegative measurable function $F$ on the space of distributions on $\bbH$ such that $F(\phi)$ depends only on $\phi|_{g({\mathbb H}_\eps)}$, 
	\[ \LF_{\mathbb H}^{(\alpha_j, w_j)_j, (\delta, \infty)} [\eps^{(\delta'-\delta)(\delta'+\delta - 2Q)} e^{(\delta'-\delta)(g^{-1}\bullet_\gamma \phi, \theta_{\eps, \infty})}F(\phi)] = \LF_{\mathbb H}^{(\alpha_j, w_j)_j, (\delta', \infty)}[F(\phi)].\]
\end{lemma}
\begin{proof}
	Write $\theta_\eps = (\delta'-\delta) \theta_{\eps, \infty}$. 
	Let $Z_\eps = {\mathbb E}[e^{(h, g_*\theta_\eps)}]$ where $h \sim P_{\mathbb H}$, then Lemma~\ref{lem-green-conformal-infty} gives $Z_\eps = \eps^{-(\delta' -\delta)^2}$. With $\mathfrak f(z) := \sum_j \alpha_j G_{\mathbb H}(z, w_j) + (\delta - Q) G_{\mathbb H}(z, \infty)$,  Lemma~\ref{lem-green-conformal-infty} also gives
	\begin{align*}
		(\mathfrak f, g_*\theta_\eps) &= (\delta'-\delta)\sum_j 2\alpha_j \log |w_j|_+ - 2(\delta' - \delta)(\delta - Q) \log \eps \\ 
		&= \log \frac{C_\gamma^{(\alpha_j, w_j)_j, (\delta', \infty)}}{C_\gamma^{(\alpha_j, w_j)_j, (\delta, \infty)}} - 2(\delta' - \delta)(\delta - Q) \log \eps.
	\end{align*}
	Writing $\phi = h + \mathfrak f + c$, Lemma~\ref{lem-green-conformal-infty} gives $(g^{-1} \bullet_\gamma \phi, \theta_\eps) = (\phi, g_*\theta_\eps) + Q(\log|g'|, \theta_\eps)= (\phi, g_*\theta_\eps)$, so
	\[e^{(g^{-1} \bullet_\gamma \phi, \theta_\eps)} = \eps^{-2(\delta'-\delta)(\delta-Q)} e^{(\delta'-\delta)c} \frac{C_\gamma^{(\alpha_j, w_j)_j, (\delta', \infty)}}{C_\gamma^{ (\alpha_j, w_j)_j, (\delta, \infty)}} e^{(h, g_*\theta_\eps)} .\]
	Thus, writing $\wt F (\phi) = \eps^{(\delta'-\delta)(\delta'+\delta - 2Q)} e^{(\delta'-\delta)(g^{-1}\bullet_\gamma \phi, \theta_{\eps, \infty})}F(\phi)$, we have 
	\alb
	\LF&_{\mathbb H}^{(\alpha_j, w_j)_j, (\delta, \infty)} [\wt F(\phi) ] =  C_\gamma^{(\alpha_j, w_j)_j, (\delta, \infty)}\int_{-\infty}^\infty e^{(\sum_j \alpha_j + \delta - Q)c} {\mathbb E}[ \wt F(h + \mathfrak f + c) ] \, dc \\
	&= C_\gamma^{(\alpha_j, w_j)_j, (\delta', \infty)} \int_{-\infty}^\infty e^{(\sum_j \alpha_j + \delta' - Q)c}{\mathbb E}[ \frac{e^{(h, g_* \theta_\eps)}}{Z_\eps} F(h + \mathfrak f + c) ] \, dc \\
	&= C_\gamma^{(\alpha_j, w_j)_j, (\delta', \infty)} \int_{-\infty}^\infty e^{(\sum_j \alpha_j + \delta' - Q)c}{\mathbb E}[ F(h + \mathfrak f + \int G_{\mathbb H}(\cdot, v) g_*\theta_{\eps}(dv)+ c) ] \, dc.
	\ale 
	where in the last line we use Girsanov's theorem.
	Lemma~\ref{lem-green-conformal-infty} gives for any $u \in g({\mathbb H}_\eps)$ that $\int G(u, v)g_*\theta_{\eps}(dv) = (\delta'-\delta)G_{\mathbb H}(u, \infty)$, so since $F$ only depends on the field on $g({\mathbb H}_\eps)$, the last line equals the right hand side of the desired identity.
\end{proof}

The following is a consequence of Lemma~\ref{lem-reweight-infty}. 

\begin{proposition}\label{prop-reweight-infty}
	Suppose any of Theorems~\ref{thm-simple-zipper},~\ref{thm-simple-zipper-rho},~\ref{thm-forest-zipper} or~\ref{thm-sf-zipper} holds for insertions $(\alpha_j, z_j)_j$ and $(\delta, \infty)$. Then the same theorem holds for insertions $(\alpha_j, z_j)_j$ and $(\delta', \infty)$ where $\delta'$ is arbitrary.
\end{proposition}
\begin{proof}
	We discuss the cases of Theorems~\ref{thm-simple-zipper} and~\ref{thm-simple-zipper-rho} ($\kappa \leq 4$) in detail; the $\kappa > 4$ results are obtained identically. 
	Proposition~\ref{prop-sle-weighted} implies that  Theorem~\ref{thm-simple-zipper} is equivalent to  Theorem~\ref{thm-simple-zipper-rho} when $\tau$ a.s.\ satisfies $g_\tau(z_j) \not \in \eta_\tau$ for all $j$. Thus it suffices to discuss only Theorem~\ref{thm-simple-zipper-rho}.

	We introduce some notation to simplify the discussion. For $\eps > 0$ let ${\mathbb H}_\eps = B_{1/\eps}(0)\cap {\mathbb H}$. 
	In the setting of Theorem~\ref{thm-simple-zipper-rho}, let $\mathrm{Zip}^{(\alpha_j, z_j)_j, (\delta, \infty)}_\tau$ denote the law of $(\phi_\tau, \eta_\tau)$, and let $\mathrm{Zip}^{(\alpha_j, z_j)_j, (\delta, \infty)}_{\tau, \eps}$ denote the law of $(\phi_\tau|_{g_{\tau}(\bbH_\eps)}, \eta_\tau)$. Let $\nu^{(\alpha_j, z_j)_j, (\delta, \infty)}_\tau$ be the law of $(\phi, \eta)$ described in~\eqref{thm-simple-zipper-rho}, and correspondingly let $\nu^{(\alpha_j, z_j)_j, (\delta, \infty)}_{\tau, \eps}$ be the law of $(\phi|_{g_\tau(\bbH_\eps)}, \eta)$. With this notation,
	  Theorem~\ref{thm-simple-zipper-rho} can be rephrased as  $\mathrm{Zip}^{(\alpha_j, z_j)_j, (\delta, \infty)}_\tau = \nu^{(\alpha_j, z_j)_j, (\delta, \infty)}_\tau$.

	Consider the setup of Theorem~\ref{thm-simple-zipper-rho} with insertions $(\alpha_j, z_j)_j$, $(\delta, \infty)$.
	For $\eps$ of the form $\eps = 2^{-n}$, let 
		$\tau_\eps \leq \tau$ be the earliest time $t$ such that at least one of the following holds: (A) $t = \tau$, (B) $g_t(-1/\eps) = W_t$, (C) $g_t (1/\eps) = W_t$, (D) $|g_t(z)| = 1$ for some $z$ with $|z| = 1/\eps$. Let $E_\eps = \{ \tau_\eps = \tau\}$. 
Since $\tau$ is a.s.\ finite, the event $\bigcup_{\eps} E_\eps$ occurs almost surely.

Weight by $\eps^{(\delta' - \delta)(\delta'+\delta-2Q)} e^{(\delta' - \delta) (\phi_0, \theta_{\eps, \infty})}$. By Lemma~\ref{lem-reweight-infty} with $K = \emptyset$, the weighted law of $\phi_0|_{{\mathbb H}_\eps}$ agrees with the law of $\phi|_{{\mathbb H}_\eps}$ where $\phi \sim \LF_{\mathbb H}^{(-\frac1{\sqrt\kappa}, 0), (\alpha_j, z_j)_j, (\delta', \infty)}$. By the definition of $\tau_\eps$, the field  $\phi_{\tau_\eps}|_{g_{\tau_\eps}(\bbH_\eps)}$ can be obtained from $\phi_0|_{\bbH_\eps}$ by conformal welding. 
Thus, the weighted law of $(\phi_{\tau_\eps}|_{g_{\tau_\eps}(\bbH_\eps)}, \eta_{\tau_\eps})$ is $\mathrm{Zip}^{(\alpha_j, z_j)_j, (\delta, \infty)}_{\tau_\eps, \eps}$, and so the weighted law of $(\phi_{\tau_\eps}|_{g_{\tau_\eps}(\bbH_\eps)}, \eta_{\tau_\eps})$ restricted to $E_{\eps}$ is $1_{E_\eps}\mathrm{Zip}^{(\alpha_j, z_j)_j, (\delta, \infty)}_{\tau, \eps}$. On the other hand, by Theorem~\ref{thm-simple-zipper-rho} with the stopping time $\tau_\eps$, the (unweighted) law of $(\phi_{\tau_\eps}, \eta_{\tau_\eps})$ is $\nu_{\tau_\eps}^{(\alpha_j, z_j)_j, (\delta, \infty)}$, so by Lemma~\ref{lem-reweight-infty} with $K$ the hull of $\eta_{\tau_\eps}$, the weighted law of $(\phi_{\tau_\eps}|_{g_{\tau_\eps}({\mathbb H}_\eps)}, \eta_{\tau_\eps})$ is $\nu^{(\alpha_j, z_j)_j, (\delta, \infty)}_{\tau_\eps, \eps}$. Since $\tau_\eps = \tau$ on $E_\eps$, the weighted law of $(\phi_{\tau_\eps}|_{g_{\tau_\eps}({\mathbb H}_\eps)}, \eta_{\tau_\eps})$  restricted to $E_\eps$ is $1_{E_\eps}\nu^{(\alpha_j, z_j)_j, (\delta, \infty)}_{\tau, \eps}$. 
We conclude that $1_{E_\eps} \mathrm{Zip}^{(\alpha_j, z_j)_j, (\delta', \infty)}_{\tau, \eps} = 1_{E_\eps} \nu^{(\alpha_j, z_j)_j, (\delta', \infty)}_{\tau, \eps}$. 

We will now take the $\eps \to 0$ limit to obtain the desired $\mathrm{Zip}^{(\alpha_j, z_j)_j, (\delta', \infty)}_\tau = \nu^{(\alpha_j, z_j)_j, (\delta', \infty)}_\tau$. First fix $\eps>0$, let $\wt \eps \in (0,\eps)$ and define $\wt \tau = \tau_\eps$. If we define $\wt \tau_{\wt \eps}$ from $\wt \tau$ in the same way that $\tau_\eps$ is defined from $\tau$, we get $\wt \tau_{\wt \eps} = \tau_\eps$. Therefore, the conclusion of the previous paragraph applied to $(\wt \tau, \wt \eps)$ implies $\mathrm{Zip}^{(\alpha_j, z_j)_j, (\delta', \infty)}_{\wt \tau, \wt \eps} = \nu^{(\alpha_j, z_j)_j, (\delta', \infty)}_{\wt \tau, \wt \eps}$. Since $\wt \tau = \tau$ on $E_\eps$, this gives $1_{E_\eps} \mathrm{Zip}^{(\alpha_j, z_j)_j, (\delta', \infty)}_{\tau, \wt \eps} = 1_{E_\eps} \nu^{(\alpha_j, z_j)_j, (\delta', \infty)}_{\tau, \wt \eps}$.
Taking the $\wt \eps \to 0$ limit gives $1_{E_\eps} \mathrm{Zip}^{(\alpha_j, z_j)_j, (\delta', \infty)}_{\tau, 0} = 1_{E_\eps} \nu^{(\alpha_j, z_j)_j, (\delta', \infty)}_{\tau, 0}$, 
so $1_{E_\eps} \mathrm{Zip}^{(\alpha_j, z_j)_j, (\delta', \infty)}_\tau = 1_{E_\eps} \nu^{(\alpha_j, z_j)_j, (\delta', \infty)}_\tau$. As $\bigcup_{\eps > 0} E_\eps$ occurs a.s., we conclude $\mathrm{Zip}^{(\alpha_j, z_j)_j, (\delta', \infty)}_\tau = \nu^{(\alpha_j, z_j)_j, (\delta', \infty)}_\tau$, i.e.,  Theorem~\ref{thm-simple-zipper-rho} holds for insertions $(\alpha_j, z_j)_j, (\delta', \infty)$. 
\end{proof}

Combining the above, we now prove the $\kappa \leq 4$ theorems. 

\begin{proof}[Proof of Theorem~\ref{thm-simple-zipper-rho}]
	Suppose $-\frac1{\sqrt\kappa} + \sum_j \alpha_j  + \delta = Q$.  Proposition~\ref{prop-main} implies that if we sample $(\phi, \eta)$ from~\eqref{eq-thm-simple-zipper-rho} and let $g: {\mathbb H} \to  {\mathbb H}\backslash \eta$ be the conformal map such that  $\lim_{z\to\infty} g(z)-z = 0$, then the law of $\phi_0 = g^{-1} \bullet_\gamma \phi$ is  $\LF_{\mathbb H}^{(-\frac1{\sqrt\kappa}, 0), (\alpha_j, z_j)_j, (\delta, \infty)}$. 
	Now, we justify that $(\phi, \eta)$ is obtained from $\phi_0$ by conformal welding  according to quantum length. For $\kappa < 4$,  \cite[Theorem 1.3]{shef-zipper} shows that when an $\SLE_\kappa$ curve is drawn on top of an independent weight 4 quantum wedge, the quantum boundary length measures on the left and right of the curve a.s.\ agree. For any segment of $\eta_\tau$ disjoint from $W_\tau$ and the points $g_\tau(z_j)$, by local absolute continuity the same statement holds. Since the quantum length measure is nonatomic, $\eta_\tau$ is the conformal welding according to quantum length. The same argument applies for $\kappa = 4$ using \cite[Theorem 1.2]{hp-welding}.
	 Thus Theorem~\ref{thm-simple-zipper-rho} holds when $-\frac1{\sqrt\kappa} + \sum_j \alpha_j  + \delta = Q$. 
	Proposition~\ref{prop-reweight-infty} then removes this constraint. 
\end{proof}

Finally, for future use we note an easy generalization of Proposition~\ref{prop-main}.
	\begin{lemma}\label{lem-main-generalized}
		Proposition~\ref{prop-main} holds for general $\delta \in \R$. That is, the condition $\delta = Q -\sum_j \alpha_j + \frac1{\sqrt\kappa}$ can be removed. 
	\end{lemma}
	\begin{proof}
		This follows from  Proposition~\ref{prop-main} and Lemma~\ref{lem-reweight-infty} by an argument similar to that of Proposition~\ref{prop-reweight-infty}. 
	\end{proof}

\begin{proof}[Proof of Theorem~\ref{thm-simple-zipper}]
	This is immediate from Theorem~\ref{thm-simple-zipper-rho} and  Proposition~\ref{prop-sle-weighted}.
\end{proof}

\section{The $\kappa > 8$ LCFT zipper}\label{sec-sf-zip}

In this section we prove Theorem~\ref{thm-sf-zipper}. There are two key steps. First, we need to produce the quantum cell in the setting of LCFT. In Section~\ref{subsec-unif-embed} we prove that the uniform embedding of the quantum wedge is a Liouville field, and in Section~\ref{sec-8-cell} we use this to transfer a statement about cutting a quantum cell from a quantum wedge to one about cutting a quantum cell from a Liouville field. Second, we must show the quantum cell arises in the time-evolution described by Sheffield's coupling (see Proposition~\ref{prop-main}). This is accomplished for a special case in Section~\ref{sec-8-special} by a technical limiting argument. In Section~\ref{sec-8-general} we bootstrap the special case to the full result.

\subsection{Uniform embedding of quantum wedge}\label{subsec-unif-embed}
We show that when a quantum wedge is embedded in the upper half-plane uniformly at random, the resulting field is a Liouville field. 
\begin{proposition}\label{prop-unif-embed}
	Let $W > \frac{\gamma^2}2$ and $\alpha = \frac12(Q+\frac\gamma2-\frac W\gamma)$. Sample $(\cW, T) \sim \cM^\mathrm{wed}(W) \times dT$ and let $({\mathbb H}, \phi_0, 0, \infty)$ be an embedding of $\cW$ chosen in a way not depending on $T$. Let $f_T(z) = e^{T} z$. Then the law of $\phi = f_T \bullet_\gamma \phi_0$ is $\frac1{Q-2\alpha} \LF_{\mathbb H}^{(\alpha, 0), (Q - \alpha, \infty)}$. 
\end{proposition}
Proposition~\ref{prop-unif-embed}  is analogous to \cite[Theorem 2.13]{AHS-SLE-integrability} which proves a similar statement for quantum disks. 
The proof hinges on the following Brownian motion identity.
\begin{lemma}\label{lem-unif-wed}
	Let $\alpha < \frac Q2$. The random processes $X_1, X_2$ defined below have the same law: 
	\begin{itemize}
		\item Let $P$ be the law of $Y$ from Definition~\ref{def-wedge}. Sample $(Y, T)\sim P \times dT$ and let $X_1(t) = Y_{t - T}$ for all $t\in \R$. 
		
		\item Let $P'$ be the law of standard two-sided Brownian motion $B$ with $B_0= 0$. Sample $(B, \mathbf c)$ from $(Q - 2\alpha)^{-1}P' \times dc$, and let $X_2(t) = B_{2t} + (Q-2\alpha)t + \mathbf c$ for all $t \in \R$. 
	\end{itemize}
	\begin{proof}
		We claim that for any $y \in \R$ we have $X_1 + y \stackrel d= X_1$. Indeed, the process $Y_t$ is a variance 2 Brownian motion with drift $(Q-2\alpha)$ started at $-\infty$ and run until it reaches $+\infty$, so setting $\wt Y_t = Y_t + y$ and $\wt \tau = \inf\{ t : \wt Y_t \geq 0\}$, we have $(\wt Y_{t + \wt \tau})_{t \in \R} \stackrel d= (Y_t)_{t \in \R}$.  Now the translation invariance of Lebesgue measure $dT$ yields $X_1 + y \stackrel d= X_1$. We will refer to  this fact as  \emph{translation symmetry}. 
		
		We first show that the law of $X_1(0)$ is $(Q-2\alpha)^{-1}dx$. 
		By Fubini's theorem, the measure of $\{X_1(0) \in (0,b)\}$ is ${\mathbb E}[\int_\R 1_{Y_t \in (0,b)} \, dt] = \int_\R \P[Y_t \in (0,b)]\, dt$. This expectation is $((Q-2\alpha)^{-1}+o(1))b$ as $b \to \infty$ since the drift term of $Y_t = B_{2t} + (Q-2\alpha)t$ dominates the Gaussian fluctuations of $B_{2t}$ for large $t$. Explicitly, fix $\eps \in (0,Q-2\alpha)$ that we will later send to $0$. For the lower bound, 
 we have $\{ |B_{2t}| < \eps t\} \subset \{ Y_t \in (0, (Q - 2\alpha + \eps)t) \}$, hence
			\[\int_\R \P[Y_t \in (0,b)]\, dt > \int_{\eps b}^{b/(Q - 2\alpha +\eps)}  \P[|B_{2t}| < \eps t]\, dt = (\frac b{Q-2\alpha+\eps} - \eps b)(1 - o_b(1)).\]
			For the upper bound, when $t > (Q-2\alpha-\eps)^{-1}b$ we have $\{Y_t < b\}\subset 
			\{ B_{2t} < -\eps t\}$ so 
			\[\int_{(Q-2\alpha-\eps)^{-1}b}^\infty \P[Y_t < b] \, dt <  \int_{(Q-2\alpha-\eps)^{-1}b}^\infty \P[B_{2t} < -\eps t] \, dt < \int_{(Q-2\alpha-\eps)^{-1}b}^\infty e^{-(\eps t)^2/4t} \, dt,\]
			giving $ \int_0^\infty \P[Y_t < b]\,dt \leq (Q-2\alpha - \eps)^{-1} b + o_b(1)$. We conclude $\int_\R \P[Y_t \in (0,b)]\, d= ((Q-2\alpha)^{-1}+o(1))b$. Translation symmetry implies the law of $X_1(0)$ is $c \, dx$ for some constant $c$, and we see that $c = (Q-2\alpha)^{-1}$. 
		
		Now, let $I$ be a finite interval. We show that conditioned on
		$X_1|_{(-\infty, 0]}$ and $X_1(0) \in I$, the process $X_1|_{[0,\infty)}$ agrees in law with $B_{2t} + (Q-2\alpha)t + X_1(0)$ where $(B_t)_{t \geq 0}$ is standard Brownian motion. By translation symmetry we may assume $I \subset (0, \infty)$,  then we may restrict to the event $\{T > 0\}$ since $Y_t < 0$ for $t<0$. By the Markov property of Brownian motion, given $(Y_t)_{(-\infty, T]}$ and $Y_T \in I$, the conditional law of $(Y_t)_{[T, \infty)}$ is $B_{2t} + (Q-2\alpha)t + Y_T$; rephrasing in terms of $X_t$ gives the desired statement. 
		
		Finally, a similar argument  gives that conditioned on $X_1|_{[0,\infty)}$ and $X_1(0) \in I$, the process  $X_1|_{(-\infty, 0]}$ agrees in law with $B_{-2t} + (Q - 2\alpha)t + X_1(0)$. Since $X_1(0) \stackrel d= X_2(0)$, and the conditional law of $X_1$ given $X_1(0)=x$ agrees with the conditional law of $X_2$ given $X_2(0) = x$ for all $x$, we conclude $X_1 \stackrel d= X_2$. 
	\end{proof}
\end{lemma}

\begin{proof}[Proof of Proposition~\ref{prop-unif-embed}]
	Let $\log: {\mathbb H}\to \cS$ be the confomal map sending $(0, 1, \infty)$ to $(-\infty, 0, +\infty)$, and let $\phi^\cS_0 = \log \bullet_\gamma \phi_0$. By definition there is a random $x \in \R$ such that $(\phi^{\cS}_0, T) \stackrel d= (\phi_\mathrm{av}(\cdot +x) + \phi_\mathrm{lat}(\cdot + x), T')$ where the law of $T'$ is Lebesgure measure on $\R$ and $(\phi_\mathrm{av}, \phi_\mathrm{lat})$ are independently sampled as in Definition~\ref{def-wedge}. By the translation invariance of Lebesgue measure on $\R$, and the translation invariance in law of $\phi_\mathrm{lat}$, Lemma~\ref{lem-unif-wed} implies $\phi_\mathrm{av}(\cdot + x - T') + \phi_\mathrm{lat}(\cdot + x - T')\stackrel d= X_2(\Re \cdot) + h_\mathrm{lat}$ where $X_2$ is as in Lemma~\ref{lem-unif-wed} and $h_\mathrm{lat}$ is the projection of an independent GFF on $\cS$ to $H_\mathrm{lat}(\cS)$. Since the projection of a GFF to $H_\mathrm{av}(\cS)$ has the law of $(B_{2t})_{t \in \R}$ where $B_t$ is standard two-sided Brownian motion, we conclude that $\log \bullet_\gamma( f_T \bullet_\gamma \phi_0)$ agrees in law with $h + (Q - 2\alpha) \Re (\cdot) + \mathbf c$ where $(h, \mathbf c) \sim \frac1{Q - 2\alpha}P_\cS\times dc$.  But by definition $\log^{-1} \bullet_\gamma (h + (Q - 2\alpha) \Re (\cdot) + \mathbf c)$ has law $\frac1{Q - 2\alpha}\LF_{\mathbb H}^{(\alpha, 0), (Q - \alpha, \infty)}$, concluding the proof. 
\end{proof}

\subsection{Cutting a quantum cell from the Liouville field}\label{sec-8-cell}

The goal of this section is to prove the following. We write $\SLE_\kappa$ for the law of forward $\SLE_\kappa$ in $({\mathbb H}, 0, \infty)$. Recall that $P_a$ is the law of the area $a$ quantum cell. 
\begin{proposition}\label{prop:MOT-stationary-qs}
	Suppose $\kappa > 8$ and $\gamma = \frac4{\sqrt\kappa}$. 
	Sample a triple $( \psi, z, \eta)$ from the measure
	\eqb\label{eq-unif-pt-wed}
	\cA_\psi(dz) \, \LF_{\mathbb H}^{(\frac{3\gamma}4, 0),(Q - \frac{3\gamma}4, \infty)} (d\psi)  \SLE_{\kappa}(d\eta).
	\eqe
	Parametrize $\eta$ by quantum area and let $A$ be the time $\eta$ first hits $z$. Let $g: {\mathbb H} \to  {\mathbb H}\backslash \eta([0,A])$ be the unique conformal map with $g(0) = z$ and $g(w) = w + O(1)$ as $w \to \infty$, and let $\phi = \psi \circ g + Q \log |g'|$. Let $\cC = (\eta([0,A]), h, \eta|_{[0,A]})/{\sim_\gamma}$. Then the law of $(\phi, \cC, A)$ is 
	\eqb \label{eq-LF-cell-decomp}
	\LF_{\mathbb H}^{(\frac{3\gamma}4, 0),(Q - \frac{3\gamma}4, \infty)} (d\phi) P_A(d\cC) 1_{A>0} dA.
	\eqe
\end{proposition}
To that end, we prove a similar statement for the quantum wedge (Lemma~\ref{lem-wedge-cell}), then transfer to the Liouville field using uniform embedding (Proposition~\ref{prop-unif-embed}).

For the next lemma, let $P(d\psi_0)$ be any law on fields such that the law of $({\mathbb H}, \psi_0, 0, \infty)/{\sim_\gamma}$ is $\cM^\mathrm{wed}(2-\frac{\gamma^2}2)$. For instance $P$ could be the law of the field of Definition~\ref{def-wedge} after parametrizing in ${\mathbb H}$ via $z \mapsto e^{z}$.

\begin{lemma}\label{lem-wedge-cell}
	Suppose $\kappa > 8$ and $\gamma = \frac4{\sqrt\kappa}$. 
	Sample $(z_0, \psi_0, \eta_0) \sim \cA_{\psi_0}(dz_0) \, P(d\psi_0) \SLE_\kappa(d\eta_0)$.  Parametrize $\eta_0$ by quantum area and let $A$ be the time it hits $z_0$. Then the marginal law of $A$ is the Lebesgue measure $\mathrm{Leb}_{(0,\infty)}$, and the conditional  law of the pair of quantum surfaces $({\mathbb H} \backslash \eta_0([0,A]), \psi_0, z_0, \infty)/{\sim_\gamma}$ and $(\eta_0([0,A]), \psi_0, \eta_0|_{[0,A]})/{\sim_\gamma}$ given $A$ is the product measure $\cM^\mathrm{wed}(2-\frac{\gamma^2}2) P_A$. 
\end{lemma}
\begin{proof}
	We construct the same random objects $(z_0, \psi_0, \eta_0)$ of the lemma statement from a different approach. Sample $(A, \psi_0, \eta_0) \sim \mathrm{Leb}_{(0,\infty)}(dA) \, P(d\psi_0) \SLE_\kappa(d\eta_0)$. Let $z_0 = \eta(A)$. Since  $\cA_{\psi_0} = \eta_* \mathrm{Leb}_{(0,\infty)}$, the law of $(z_0, \psi_0, \eta_0)$ is $\cA_{\psi_0}(dz_0) P(d\psi_0)\SLE_\kappa(d\eta_0)$. By our construction the marginal law of $A$ is $\mathrm{Leb}_{(0,\infty)}$.
	
	Condition on $A$. The conditional law of 
	$(\psi_0, \eta_0)$ is $P(d\psi_0) \SLE_\kappa (d\eta_0)$. 	
	By definition, the  law of $\cC_A := (\eta([0,A]), \psi_0, \eta_0|_{[0,A]})/{\sim_\gamma}$ is $P_A$. Finally, by the last two claims of \cite[Theorem 1.4.1]{wedges},   $(\eta_0([A,\infty)), \psi_0, \eta(A), \infty)/{\sim_\gamma}$ has law $\cM^\mathrm{wed}(2-\frac{\gamma^2}2)$ and is independent of $\cC_A$.

\end{proof}

We are now ready to prove  Proposition~\ref{prop:MOT-stationary-qs}.

\begin{proof}[Proof of Proposition~\ref{prop:MOT-stationary-qs}]	
	Sample $(z_0, \psi_0, \eta_0, T) \sim \cA_{\psi_0}(dz_0) \, P(d\psi_0) \SLE_\kappa(d\eta_0) \, dT$. Let $f_T(z) = e^T z$, and let $(z, \psi, \eta) = (f_T(z_0), f_T \bullet_\gamma \psi_0, f_T \circ \eta_0)$. By Proposition~\ref{prop-unif-embed} and the scale-invariance of $\SLE_\kappa$ the law of $(z, \psi, \eta)$ is~\eqref{eq-unif-pt-wed} times $(Q-\frac{3\gamma}2)^{-1}$. 
	
	Let $A$ be the time that $\eta_0$ hits $z_0$. 
	Let $g_0: {\mathbb H} \to {\mathbb H} \backslash \eta_0([0,A])$ be the conformal map with $g_0(0) = z_0$ and $g_0(w) = w + O(1)$ as $w \to \infty$,  and let $\phi_0 = \psi_0 \circ g_0 + Q \log |g_0'|$.
	Let $g: {\mathbb H}\to {\mathbb H}\backslash \eta([0,A])$ be the conformal map with $g(0) = z$ and $\lim_{w \to \infty} g(w) - w = 0$, and let $\phi = \psi \circ g + Q \log |g_T'|$. 
	Let $\cC = (\eta([0,A]), \psi, \eta|_{[0,A]})/{\sim_\gamma} = (\eta_0([0,A]), \psi_0, \eta_0|_{[0,A]})/{\sim_\gamma}$.
	
	By Lemma~\ref{lem-wedge-cell} the law of $(({\mathbb H}, \phi_0, 0, \infty)/{\sim_\gamma}, \cC, A, T)$ is $\cM^\mathrm{wed}(2-\frac{\gamma^2}2) \, P_A(d\cC) \, 1_{A > 0} dA \, dT$. It is easy to check that $f_T \circ g_0 = g \circ f_T$, so $f_T \bullet_\gamma \phi_0 = \phi$. Thus by Proposition~\ref{prop-unif-embed} the law of $(\phi, \cC, A)$ is~\eqref{eq-LF-cell-decomp} times $(Q-\frac{3\gamma}2)^{-1}$.
\end{proof}

\subsection{The $n=0$ case of   Theorem~\ref{thm-sf-zipper}}\label{sec-8-special}
In this section, we prove the following. 
\begin{proposition}\label{prop-thm-8}
	Theorem~\ref{thm-sf-zipper} holds for the special case where $n=0$.
\end{proposition}

In Lemma~\ref{lem-shef-coupling} 
we construct a process $(\phi_t, \eta_t)_{t \geq 0}$ on field-curve pairs where the evolution is given by Sheffield's coupling. 
In Proposition~\ref{prop:MOT-zip-abstract} we run this process until the time $\tau$ that the quantum area has increased by a Lebesgue-typical amount, and identify the law of $(\phi_\tau, \eta_\tau)$. In Proposition~\ref{prop:MOT-zip-concrete} we use Propositions~\ref{prop:MOT-stationary-qs} and~\ref{prop:MOT-zip-abstract} to show that $(\phi_\tau, \eta_\tau)$ comes from conformally welding $\phi_0$ with a quantum cell. Proposition~\ref{prop-thm-8} then follows quickly.

We write $\SLE_{\kappa}^t$ to denote the law of forward $\SLE_{\kappa}$ in ${\mathbb H}$ from $0$ to $\infty$ run for time $t$ (so its half-plane capacity is $2t$), and write $\rSLE_{\kappa}^t$ for reverse SLE in $({\mathbb H}, 0, \infty)$ run for time $t$. 
Let $\mathfrak F$ be the space of distributions on ${\mathbb H}$ (the test functions being smooth functions  compactly supported in $\bbH$) and let $\mathfrak C$ be the space of bounded parametrized curves in $\ol {\mathbb H}$ equipped with the metric $d(\eta_1, \eta_2) =|T_1 - T_2| +  \sup_{t \leq T_1 \wedge T_2} |\eta_1(t) - \eta_2(t)|$ where $T_i$ is the duration of $\eta_i$. 
\begin{lemma}\label{lem-shef-coupling}
	Let $\kappa > 4$ and $\gamma = \frac4{\sqrt\kappa}$. There is an infinite measure $M$ on $C([0,\infty), \mathfrak F \times \mathfrak C)$ such that, for $(\phi_t, \eta_t)_{t\geq0}$ sampled from $M$, 
	\begin{enumerate}[i)]
		\item $\eta_t$ satisfies $\mathrm{hcap}(\eta_t) = 2t$ and is parametrized by half-plane capacity (i.e.\ $\mathrm{hcap}(\eta_t([0,s])) = 2s$);
		\item for $\tau$ an a.s.\  finite stopping time  for the filtration $\cF_t = \sigma(\eta_t)$, the law of $(\phi_\tau, \eta_\tau)$ is \\ $\LF_{\mathbb H}^{(-\frac\gamma4, \eta_\tau(0)), (Q - \frac{3\gamma}4, \infty)} \rSLE_{\kappa}^\tau$;
		
		\item for $t_1 < t_2$, let $g_{t_1, t_2}: {\mathbb H} \to {\mathbb H}\backslash \eta_{t_2}([0,t_2-t_1])$ be the conformal map with $g_{t_1, t_2}(\eta_{t_1}(0)) = \eta_{t_2}(t_2-t_1)$ and $\lim_{z \to \infty} g_{t_1, t_2}(z) - z = 0$. Then $\phi_{t_1} = g_{t_1,t_2}^{-1} \bullet_\gamma \phi_{t_2}$. 
	\end{enumerate}
\end{lemma}
\begin{proof}
	Fix $T$. Sample $(\phi_T, \eta_T) \sim \LF_{\mathbb H}^{(-\frac\gamma4, \eta_\tau(0)), (Q - \frac{3\gamma}4, \infty)} \rSLE_{\kappa}^T$, for $0 \leq t < T$ let $\eta_t$ be the reverse SLE curve at time $t$, for $t_1 < t_2 < T$ define $g_{t_1,t_2}:  {\mathbb H} \to {\mathbb H}\backslash \eta_{t_2}([0,t_2-t_1])$  as above, and let $\phi_t = g_{t, T}^{-1} \bullet_\gamma \phi_T$. Let $M_T$ denote the law of $(\phi_t, \eta_t)_{[0,T]} \in C([0,T], \mathfrak F \times \mathfrak C)$.
	We claim that $M_T$ satisfies the desired conditions up to time $T$.  i) is immediate by definition. 
	For ii), we want to determine the law of the field and curve when $(\phi_T, \eta_T)$ is ``unzipped'' for $(T-\tau)$ units of time; this follows from the strong Markov property of reverse $\SLE_\kappa$, and  Lemma~\ref{lem-main-generalized} with $(n, \delta)=(0,Q - \frac{3\gamma}4)$ and with its stopping time equal to $T - \tau$. Next, by ii) the law of $(\phi_{t_2}, \eta_{t_2})$ is 
	$\LF_{\mathbb H}^{(-\frac\gamma4, \eta_{t_2}(0)), (Q - \frac{3\gamma}4, \infty)} \rSLE_{\kappa}^{t_2}$, so applying Lemma~\ref{lem-main-generalized} with $n=0$, $\delta = Q - \frac{3\gamma}4$, and stopping time $t_2 - t_1$ gives iii). Finally, continuity of $(\phi_t)_{t \geq 0}$ is immediate from iii), and continuity of reverse $\SLE$ is well known. 
	
	The measure $M_T$ satisfies the desired properties up to time $T$. Moreover for $T' > T$, if we sample $(\phi_t, \eta_t)_{t \leq T'} \sim M_{T'}$, then by ii) the law of $(\phi_t, \eta_t)_{t \leq T}$ is $M_T$. Thus  the Kolmogorov extension theorem gives the existence of $M$. 
\end{proof}

For $(\phi_t, \eta_t)_{t \geq 0} \sim M$, for each $t$ let $W_t := \eta_t(0) \in \R$. We point out that the field-curve pair $(\phi_t(\cdot + W_t), \eta_t - W_t)$ is the translation of $(\phi_t, \eta_t)$ sending $\eta_t(0)$ to $0$.

\begin{proposition} \label{prop:MOT-zip-abstract}
	Let $\kappa > 8$ and $\gamma = \frac4{\sqrt\kappa}$.
	Sample $(\{(\phi_t, \eta_t)\}_{t \geq 0}, A)$ from the product measure $M 1_{A>0} \, dA$. Let $\tau > 0$ be the time $t$ that $\cA_{\phi_t}(\eta_t([0,t])) = A$.
	Then the law of $(\phi_\tau(\cdot + W_\tau), \eta_\tau - W_\tau, \tau)$ is  $ \mathcal K \cdot  \LF_{\mathbb H}^{(\frac{3\gamma}4, 0), (Q - \frac{3\gamma}4,\infty)}  \SLE_{\kappa}^t  1_{t >0} dt$ for some constant $\mathcal K \in (0,\infty)$.
\end{proposition}
At high level, the proof of Proposition~\ref{prop:MOT-zip-abstract} goes as follows, see Figure~\ref{fig-sf-until-A}. First, if we instead sample $(\{(\phi_t, \eta_t)\}_{t \geq 0}, A, T)\sim \delta^{-1} 1_{T \in [\tau, \tau + \delta]} M 1_{A>0} \, dA  \, dT$ then the marginal law of $(\phi_\tau, \eta_\tau, \tau)$ is the same as in Proposition~\ref{prop:MOT-zip-abstract}. Secondly, by the definition of $\tau$ the field $\phi_T$ has a singularity $\gamma G(\cdot, z)$ at
$z = \eta_T(T - \tau)$ since $z$ is ``quantum typical'' (Proposition~\ref{prop-pointed}). As $\delta \to 0$ we have $T - \tau \to 0$ so $z \to W_T$, so in the limit the field $\phi_T$ has the singularity $\gamma G(\cdot, W_T) - \frac\gamma4 G(\cdot, W_T)$ at $W_T$. Finally, since $T \to \tau$ as $\delta \to 0$, we have $(\phi_T(\cdot  +W_T), \eta_T - W_T, T) \to (\phi_\tau(\cdot + W_\tau), \eta_\tau - W_\tau, \tau)$. The primary complication is in taking limits of infinite measures; below we truncate on events to work with finite measures. We state the finiteness of one such event as Lemma~\ref{lem-EN-finite}. 
\begin{proof}
	To simplify notation, let $(\phi^1, \eta^1) := (\phi_\tau(\cdot + W_\tau), \eta_\tau - W_\tau)$.
	Let $\rho$ be an arbitrary smooth compactly-supported function in $\bbH$. 
	Let $N>0$ and define 
	\eqb\label{eq-EN}
	E_N := \{\tau , |(\phi^1, \rho)| , |(f^{-1} \bullet_\gamma \phi^1, \rho)| < N \}
	\eqe
	where $f: {\mathbb H}\to {\mathbb H}\backslash \eta^1([0,\tau])$ is the conformal map with $f(0) = \eta^1(\tau)$ and $f(z) = z + O(1)$ as $z \to \infty$. 
	Let $P_N$ denote the conditional law of $(\phi^1, \eta^1, \tau)$ given $E_N$; this is well defined because the measure of $E_N$ is finite (Lemma~\ref{lem-EN-finite}). Likewise let $\wt P_N$ denote the conditional law of $(\phi, \eta, t) \sim \LF_{\mathbb H}^{(\frac{3\gamma}4, 0), (Q - \frac{3\gamma}4,\infty)}  \SLE_{\kappa}^t  1_{t >0} dt$ given
	$\{t, |(\phi, \rho)|, |(g^{-1} \bullet_\gamma\phi)| < N\}$
	with $g: {\mathbb H}\to {\mathbb H}\backslash \eta([0,t])$ the conformal map with $g(0) = \eta(t)$ and $g(z) = z + O(1)$ as $z \to \infty$. We will show that $P_N = \wt P_N$; sending $N \to \infty$ concludes the proof. 
	
	Sample $(\{(\phi_t, \eta_t)\}_{t \geq 0}, A, T)$ from $1_{T > \tau} M  \, 1_{A>0} \, dA \, 1_{T > 0} dT$ where as before $\tau$ is the time $t$ that $\cA_{\phi_t}(\eta_t([0,t])) = A$. As before let $(\phi^1, \eta^1) := (\phi_\tau(\cdot + W_\tau), \eta_\tau - W_\tau)$. Let $F_{\delta} = \{ T - \tau < \delta\}$. 
	Let $(\phi^2, \eta^2) = (\phi_T(\cdot + W_T), \eta_T - W_T)$ and let $G_{\eps} = \{ \eta^2([0, T - \tau]) \subset B_{\eps}(0)\}$. See Figure~\ref{fig-sf-until-A}. The two claims below imply $P_N = \wt P_N$ as needed.
	\medskip 
	
	\noindent \textbf{Claim 1: Law of $(\phi^1, \eta^1, \tau)$ given $E_N \cap F_\delta \cap G_\eps$ converges to $P_N$ as $\delta \to 0$ then $\eps \to 0$. }
	\\Let $X = T - \tau$, then the conditional law of $(\phi^1, \eta^1, \tau, X)$ given $E_N \cap F_{\delta}$ is $P_N  \delta^{-1} 1_{X \in [0,\delta]} \, dX$. As $\delta \to 0$, we have $X \to 0$ in probability, so by continuity of the Loewner chain the diameter of $\eta^2([0, T - \tau])$ shrinks to 0 in probability. Hence the conditional probability of $G_{\eps}$ given $E_N \cap F_\delta$ is $1-o_{\delta}(1)$. Thus, as $\delta \to 0$ then $\eps \to 0$, the conditional law of $(\phi^1, \eta^1)$ given $E_N \cap F_{\delta} \cap G_{\eps}$ converges to $P_N$ in total variation. 
	\medskip
	
	\noindent \textbf{Claim 2: Law of $(\phi^1, \eta^1, \tau)$ given $E_N \cap F_\delta \cap G_\eps$ converges to $\wt P_N$ as $\delta \to 0$ then $\eps \to 0$. }
	\\By the second property of $M$ in Lemma~\ref{lem-shef-coupling} applied to time $T$, and the fact that for fixed $t$ the curves $\SLE_\kappa^t$ and centered $\rSLE_\kappa^t$ have the same law, the unconditioned law of $(\phi^2, \eta^2, A, T)$ is 
	\[1_{A \in (0, \cA_{\phi^2}(\eta^2([0,T]))}dA\, \LF_{\mathbb H}^{(-\frac\gamma4, 0), (Q - \frac{3\gamma}4, \infty)}(d\phi^2) \SLE_\kappa^{T}(d\eta^2)1_{T >0} dT.\]
	Let $z = \eta^2(T - \tau)$. 
	The conditional law of $z$ given $(\phi^2, \eta^2, T)$ is the probability measure proportional to $\cA_{\phi^2}|_{\eta^2([0,T])}$. Thus, using Proposition~\ref{prop-pointed},  the unconditioned law of $(\phi^2, \eta^2, z, T)$ is 
	\begin{align}
		&1_{z \in \eta^2([0,T])}\cA_{\phi^2}(dz) \LF_{\mathbb H}^{(-\frac\gamma4, 0), (Q - \frac{3\gamma}4, \infty)}(d\phi^2) \SLE_\kappa^{T}(d\eta^2)1_{T >0} dT \nonumber\\
		&=\LF_{\mathbb H}^{(-\frac\gamma4, 0), (\gamma, z), (Q - \frac{3\gamma}4, \infty)}(d\phi^2) 1_{z \in \eta^2([0,T])}dz \SLE_\kappa^{T}(d\eta^2)1_{T >0} dT. \label{eq-phi2-phrasing}
	\end{align}
	Now we express the event $E_{N} \cap F_{\delta} \cap G_{\eps}$ in terms of $(\phi^2, \eta^2, z, T)$. Let $\tau_z$ be the time $\eta^2$ hits $z$, let $f_1: {\mathbb H}\to {\mathbb H}\backslash \eta^2([0,\tau_z])$ (resp.\ $f_0: {\mathbb H} \to {\mathbb H} \backslash \eta^2([0,T])$) be the conformal map sending $0$ to $\eta^2(\tau_z)$ (resp.\ $\eta^2(T)$) and with asymptotic behavior $w \mapsto w + O(1)$ as $w \to \infty$, and let $f = f_0 \circ f_1^{-1}$. Then
	\[E_N = \{T - \tau_z,|(\phi^1, \rho)|, |(f^{-1} \bullet_\gamma \phi^1, \rho)| < N \}, \qquad \phi^1 := f_1^{-1} \bullet_\gamma \phi^2,	\] $F_{\delta} = \{\tau_z < \delta \}$ and $G_{\eps} = \{\eta^2([0,\tau_z]) \subset B_{\eps}(0) \}$. 
	
	We now have the description~\eqref{eq-phi2-phrasing} of the unconditioned law of $(\phi^2, \eta^2, z,T)$, and the previous paragraph's description of $E_N, F_\delta, G_\eps$. 
	Since $f_1$ converges to the identity map and $z \to 0$ by the definition of $G_\eps$, and since $\lim_{z \to 0} G_{\mathbb H} (\cdot, z) = G_{\mathbb H}(\cdot, 0)$ in $\mathfrak F$, 
	the law of $(\phi^1, \eta^1, T)$  conditioned on $E_N \cap F_\delta \cap G_\eps$ converges to $\wt P_N$ as $\delta \to 0$ then $\eps \to 0$. See Lemma~\ref{lem-couple-conv} (with $I = [-N, N]$ and $\rho', I'$ chosen such that $\{(\phi^1, \rho') \in I'\} = \{|f^{-1} \bullet_\gamma \phi^1, \rho)|<N \}$) for details. Thus the second claim holds. 
\end{proof}

\begin{figure}
	\begin{center}
		\includegraphics[scale=0.38]{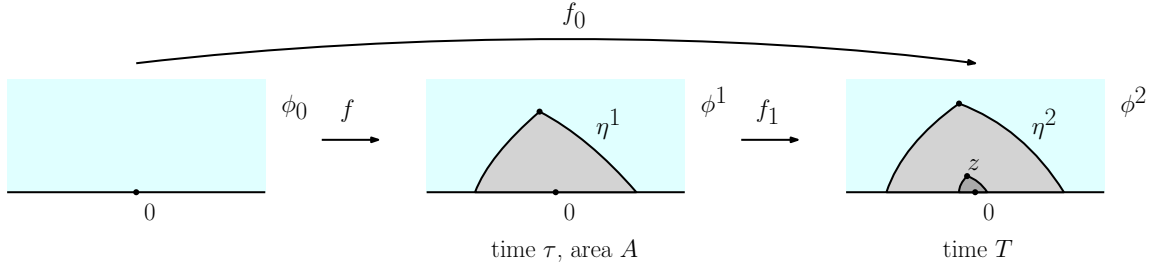}
	\end{center}
	\caption{\label{fig-sf-until-A} Figure for Proposition~\ref{prop:MOT-zip-abstract}. 
		Sample $(\{(\phi_t, \eta_t)\}_{t \geq 0}, A, T)$ from $1_{T > \tau} M \, 1_{A>0} \, dA \, 1_{T > 0} dT$ and let $\tau$ be the time the curve (grey) has quantum area $A$.
		Roughly speaking, $E_N$ is the event that $\tau$ and certain field observables of $\phi_0$ and $\phi^1$ are not too large, 
		$F_\delta=\{T - \tau < \delta \}$, and  $G_\eps$ is the event that the dark grey region is small in a Euclidean sense.  As $\delta \to 0$ then $\eps \to 0$,  $(\phi^2, \eta^2)$ approximates $(\phi^1, \eta^1)$. We use this to show $P_N = \wt P_N$. }
\end{figure}

\begin{lemma}\label{lem-EN-finite}
	In the setting of Proposition~\ref{prop:MOT-zip-abstract}, with $(\phi^1, \eta^1) := (\phi_\tau(\cdot + W_\tau), \eta_\tau - W_\tau)$ and $E_N$ defined in~\eqref{eq-EN}, the $(M \, 1_{A>0}\,dA)$-measure of $E_N$ is finite. 
\end{lemma}
\begin{proof}
	The event $\{ \tau < N\}$ agrees with $\{ \cA_{\phi_N}(\eta_N([0,N])) < A \}$. Thus, the $(M \, 1_{A>0}\,dA)$-measure of $E_N$ is $M[\cA_{\phi_N} (\eta_N([0,N]))1_{\wt E_N}]$,  where $\wt E_N = \{ |(\phi_\tau, \rho)|, |(\phi_0, \rho)| < N \}$. 
	By the second property of $M$ from Lemma~\ref{lem-shef-coupling}, and the fact that for fixed $t$ the curves $\SLE_\kappa^t$ and centered $\mathrm{rSLE}_\kappa^t$ have the same law, the law of $(\phi_N(\cdot + W_N), \eta_N - W_N)$ is $\LF_{\mathbb H}^{(-\frac\gamma4, 0), (Q - \frac{3\gamma}4, \infty)} \SLE_{\kappa}^N$. Thus it suffices to show that 
	\eqb\label{eq-EN-fin-rephrased}
	(\LF_{{\mathbb H}}^{(-\frac\gamma4, 0), (Q - \frac{3\gamma}4,\infty)} \, \SLE_\kappa^N) [\cA_\phi(\eta([0,N])) 1_{|(f \bullet_\gamma \phi, \rho)| < N} ] < \infty,
	\eqe
	where $f:{\mathbb H}\backslash \eta([0,N]) \to {\mathbb H}$ is the conformal map sending $f(\eta(N)) = 0$ and with $f(z) = z + O(1)$ as $z \to \infty$. 
	
	We first claim that for some $C>0$,
	\eqb\label{eq-SLE-size}
	\P \left[\eta([0,N]) \not \subset R_x \right] < 2 e^{-x^2/C^2} \quad \text{ for }\eta \sim \SLE_\kappa^N \text{ and }x>0, \text{ where } R_x := [-x, x] \times [0,2\sqrt N].
	\eqe
	Indeed, the driving function $W_t$ of $\SLE_\kappa$ is a multiple of Brownian motion so $\sup_{0<t<N} |W_t|$ is sub-Gaussian, and  for compact $K \subset \ol {\mathbb H}$ such that ${\mathbb H} \backslash K$ is simply connected, $\mathrm{hcap}(K) = 2N$ implies $K \subset \R \times [0, 2\sqrt N]$. 
	
	Next, fix $K$ (we will later set $K = \eta$) and let $f: {\mathbb H}\backslash K \to {\mathbb H}$ be a conformal map fixing $\infty$. Then
	\begin{align*}
		&\LF_{{\mathbb H}}^{(-\frac\gamma4, 0), (Q - \frac{3\gamma}4,\infty)}[\cA_\phi(R_x) 1_{|(f \bullet_\gamma \phi, \rho)|< N}] = \int_\R e^{-\gamma c} {\mathbb E} [\cA_{\wt h + c} (R_x) 1_{|(f \bullet_\gamma (\wt h + c), \rho)| < N)}] \, dc \\
		&= \int_\R {\mathbb E} [\cA_{\wt h} (R_x) 1_{|(f \bullet_\gamma \wt h, \rho) + c | < N)}] \, dc =2N {\mathbb E}[\cA_{\wt h}(R_x)],
	\end{align*}
	where the expectation is taken over $h \sim P_{\mathbb H}$ and we write $\wt h = h - \frac\gamma4 G_{\mathbb H}(\cdot, 0) + (Q - \frac{3\gamma}4) G_{\mathbb H}(\cdot, \infty)$, we use $-\frac\gamma4 + (Q - \frac{3\gamma}4) - Q = -\gamma$ in the first equality and $\cA_{\wt h + c} = e^{\gamma c} \cA_{\wt h}$ in the second equality, and we interchange ${\mathbb E}$ and $\int$ in the third equality. Since for some $C'$ we have $(\wt h - h)\leq C'\log x$ on $R_x$, 
	\[{\mathbb E}[\cA_{\wt h}(R_x)] \leq x^{C'} {\mathbb E}[\cA_h(R_x)] = x^{C'} \int_{R_x} (2\Im z)^{-\frac{\gamma^2}2} |z|_+^{2\gamma^2} dz \lesssim x^{C' + 2\gamma^2+1}.\]
	where we use here the formula for $G_{\mathbb H}$ from~\eqref{eq-G}. Thus, $\LF_{{\mathbb H}}^{(-\frac\gamma4, 0), (Q - \frac{3\gamma}4,\infty)}[\cA_\phi(R_x) 1_{|(f \bullet_\gamma \phi, \rho)|< N}]$ is bounded above by a polynomial in $x$; combining with~\eqref{eq-SLE-size} gives the desired~\eqref{eq-EN-fin-rephrased}.

\end{proof}

\begin{lemma}\label{lem-couple-conv}
	Let $\rho, \rho'$ be  smooth compactly-supported functions in $\bbH$, and let $I, I' \subset \R$ be intervals of finite nonzero length. For $\eps > 0$, let $K \subset B_\eps(0) \cap \ol\bbH$ be a compact set such that  $\bbH\backslash K$ is simply connected, $z \in \partial K \cap \bbH$ and let $f_1: \bbH\to \bbH\backslash K$ be the conformal map with $f_1(0) = z$ with asymptotic behavior $w \mapsto w + O(1)$ as $w \to \infty$. For $\phi^2 \sim \LF_{\bbH}^{(-\frac\gamma4, 0), (\gamma, z), (Q - \frac{3\gamma}4, \infty)}$ and $\phi^1 := f_1^{-1} \bullet_\gamma \phi^2$, the law of $\phi^1$  conditioned on $\{ (\phi^1, \rho) \in I\}$ converges as $\eps \to 0$ to the law of $\phi \sim \LF_\bbH^{(\frac{3\gamma}4, 0), (Q - \frac{3\gamma}4, \infty)}$ conditioned on $\{ (\phi, \rho) \in I\}$.  Moreover, the conditional probability of $\{ (\phi^1, \rho') \in I'\}$ converges as $\eps \to 0$ to the conditional probability of $\{ (\phi, \rho') \in I'\}$, and the law of $\phi^1$ conditioned on $\{ (\phi^1, \rho)\in I , (\phi^1, \rho') \in I' \}$ converges as $\eps \to 0$ to the law of $\phi$ conditioned on $\{ (\phi, \rho) \in I, (\phi, \rho') \in I'\}$.
\end{lemma}
\begin{proof}
	Recall from Definition~\ref{def-LF} that $\phi = h + \frac{3\gamma}4 G_\bbH(\cdot, 0)  -\frac{3\gamma}4 G_\bbH(\cdot, \infty)$ where $(h, \mathbf c) \sim P_\bbH \times dc$. Notice that if we define $\wt \phi^2$ on the same probability space via $\wt \phi^2 := h - \frac\gamma4 G_\bbH(\cdot, 0)+ \gamma G_\bbH(\cdot, z)  -\frac{3\gamma}4 G_\bbH(\cdot, \infty)$, then the law of $\wt \phi^2$ is a multiple of $\LF_{\bbH}^{(-\frac\gamma4, 0), (\gamma, z), (Q - \frac{3\gamma}4, \infty)}$, so the conditional law of $\wt \phi^1 := f_1^{-1} \bullet_\gamma \wt \phi^2$ given $\wt E_{\rho, I} := \{ (\wt \phi^1, \rho)\in I\}$ agrees with the conditional law of $\phi^1$ given $\{ (\phi^1, \rho)\in I\}$. The field $\wt \phi^1$ a.s.\ converges to $\phi$ as $\eps \to 0$, since $f_1$ converges to the identity map and $G_\bbH(\cdot, z)$ converges to $G_\bbH(\cdot, 0)$ in the space of distributions $\mathfrak F$. 
	It is easy to check that the symmetric difference of the events $E_{\rho, I} := \{ (\phi, \rho) \in I\}$ and $\wt E_{\rho, I}$ has measure $o_\eps(1)$, yielding the first claim. Similarly one can check that the symmetric difference of $E_{\rho, I} \cap E_{\rho',I'}$ and $\wt E_{\rho,I} \cap \wt E_{\rho',I'}$ has measure $o_\eps(1)$; this implies the second claim. 
\end{proof}

\begin{proposition}\label{prop:MOT-zip-concrete}
	In the setting of Proposition~\ref{prop:MOT-zip-abstract}, let $\cC = (\eta_\tau([0,\tau]), \phi_\tau, \eta_\tau)/{\sim_\gamma}$. Then the law of $(\phi_0, \cC, A)$ is \[\LF_{\mathbb H}^{(-\frac\gamma4, 0), (Q-\frac{3\gamma}4,\infty)}\, P_a\,  1_{a>0} da.\]
\end{proposition}
\begin{proof}
	Sample $(\{(\phi_t, \eta_t)\}_{t \geq 0}, A_1, A_2)$ from $M \, 1_{A_1, A_2 > 0} dA_1 dA_2$. Let $\tau^1$ (resp.\ $\tau^2$) be the time $t$ that  $\cA_{\phi_t}( \eta_t([0,t]) )$ equals $A_1$ (resp.\ $A_1+A_2$). Let $\cC_2 = (\eta_{\tau^2}([0,\tau^2-\tau^1]), \phi_{\tau^2}, \eta_{\tau^2}|_{[0,\tau^2-\tau^1]})/{\sim_\gamma}$.
	We will show that the law of $(\phi_0, \cC_2, A_1, A_2)$ is 
	\eqb\label{eq:MOT-zip-weighted}
	\LF_{\mathbb H}^{(-\frac\gamma4, 0), (Q-\frac{3\gamma}4,\infty)}(d\phi^0) P_{A_2}(d\cC_2) 1_{A_1,A_2>0}dA_1 dA_2.
	\eqe
	To simplify notation, let $(\phi^j, \eta^j) = (\phi_j(\cdot + W_{\tau_j}), \eta_j - W_{\tau_j})$ for $j=1,2$
	and let $\eta^{12} = \eta^2|_{[0,\tau^2-\tau^1]}$. Let $z = \eta^2(\tau^2 - \tau^1)$.  See Figure~\ref{fig-sf-two-step}. 
	
	\begin{figure}
		\begin{center}
			\includegraphics[scale=0.38]{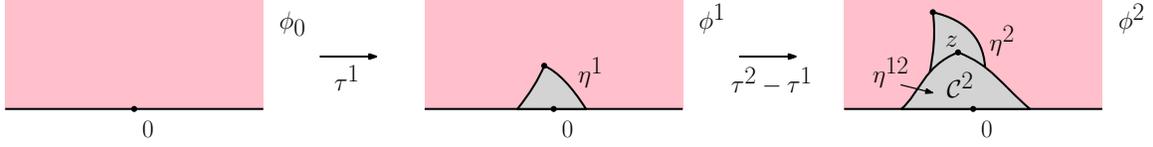}
		\end{center}
		\caption{\label{fig-sf-two-step} Figure for Proposition~\ref{prop:MOT-zip-concrete}. The pair $(\phi^j, \eta^j)$ corresponds to $(\phi_{\tau_j}, \eta_{\tau_j})$ translated so that the curve starts from $0$ rather than $W_{\tau_j}$. The curve $\eta^{12}$ is the initial segment of $\eta^2$.}
	\end{figure}
	
	Let $S = A_1 + A_2$, then by a change of variables the law of $(\{(\phi_t, \eta_t)\}_{t \geq 0}, A_1, S)$ is $M \, 1_{A_1 \in [0,S]} dA_1 \, 1_{S > 0} dS$. By Proposition~\ref{prop:MOT-zip-abstract} the law of $(A_1, \phi^2, \eta^2, \tau^2)$ is 
	\[\mathcal K \cdot 1_{A_1 \in [0,S]}dA_1\, \LF_{\mathbb H}^{(\frac{3\gamma}4, 0), (Q - \frac{3\gamma}4, \infty)}(d\phi^2) \SLE_{\kappa}^{\tau^2}(d\eta^2) 1_{\tau^2>0} d\tau^2, \quad S:= \cA_{\phi^2}(\eta^2([0,\tau^2])).\] Since $z$ is the point such that $\eta^2$ covers $S - A_1$ units of quantum area before hitting $z$, we conclude the law of $(z,\phi^2, \eta^2, \tau^2)$ is
	\[ \mathcal K\cdot 1_{z \in \eta^2([0,\tau^2])} \cA_{\phi^2}(dz) \LF_{\mathbb H}^{(\frac{3\gamma}4, 0), (Q - \frac{3\gamma}4, \infty)}(d\phi^2) \SLE_{\kappa}^{\tau^2}(d\eta^2) 1_{\tau^2>0} d\tau^2.\]
	By Lemma~\ref{lem-curve-leb-z} stated below, the law of $(\phi^2, \eta^1, \eta^{12},z,\tau^1)$ is 
	\[\mathcal K \cdot \SLE_{\kappa}^z(d\eta^{12})\, \cA_{\phi^2}(dz)\, \LF_{\mathbb H}^{(\frac{3\gamma}4, 0),  (Q - \frac{3\gamma}4, \infty)}(d\phi^2)   \SLE_{\kappa}^{\tau^1}(d\eta^1)  \, 1_{\tau^1>0} d\tau^1\]
	where $\SLE_{\kappa}^z$ denotes forward $\SLE_{\kappa}$ run until the time it hits $z$. 
	By Proposition~\ref{prop:MOT-stationary-qs}, the law of $(\phi^1, \eta^1, \tau^1, \cC_2, A_2)$ is 
	\[\mathcal K \cdot \LF_{\mathbb H}^{(\frac{3\gamma}4, 0), (Q - \frac{3\gamma}4, \infty)}(d\phi^1) \SLE_{\kappa}^{\tau^1}(d\eta^1) 1_{\tau^1>0} d\tau^1 P_{A_2}(d \cC_2) 1_{A_2>0} dA_2.  \]
	By Proposition~\ref{prop:MOT-zip-abstract} applied to the measure $\mathcal K \cdot \LF_{\mathbb H}^{(\frac{3\gamma}4, 0), (Q - \frac{3\gamma}4, \infty)}(d\phi^1) \SLE_{\kappa}^{\tau^1}(d\eta^1) 1_{\tau^1>0} d\tau^1$, we conclude that the law of $(\phi_0, \cC_2, A_1, A_2)$ is~\eqref{eq:MOT-zip-weighted}.
	
	As a consequence of~\eqref{eq:MOT-zip-weighted}, if we sample  $(\{(\phi_t, \eta_t)\}_{t \geq 0}, A_1, A_2)$ from $M \, \mathrm{Unif}_{[0,\eps]}(dA_1) 1_{A_2 > 0} dA_2$ where $\mathrm{Unif}_{[0,\eps]}$ is the uniform probability measure on $[0,\eps]$, then the law of $(\phi_0, \cC_2, A_1, A_2)$ is 
	$\LF_{\mathbb H}^{(-\frac\gamma4, 0), (Q-\frac{3\gamma}4,\infty)}(d\phi_0) P_{A_2}(d\cC_2) \mathrm{Unif}_{[0,\eps]}(dA_1) 1_{A_2>0}dA_2$. Sending $\eps \to 0$ yields the desired result. 
\end{proof}

\begin{lemma}\label{lem-curve-leb-z}
	Let $z \in {\mathbb H}$ and consider a pair $(\eta, T) \sim 1_E \SLE_\kappa^t 1_{t>0} dt$ where $E$ is the event that $z$ lies in the range of the curve. Let $\tau_z$ be the time $\eta$ hits $z$, let $\eta^{12} = \eta|_{[0,\tau_z]}$, let $T^1 = T - \tau_z$ and let $\eta^1 = f^{-1} \circ \eta(\cdot + \tau_z)|_{[0, T^1]}$ where $f: {\mathbb H}\to {\mathbb H} \backslash \eta([0,\tau_z])$ is the conformal map with $f(0) = z$ and $f(w) = w + O(1)$ as $w \to \infty$. Then the law of $(\eta^{12}, \eta^1, T^1)$ is $\SLE_\kappa^z \SLE_\kappa^{t^1} 1_{t^1>0}dt^1$, where $\SLE_\kappa^z$ is the law of $\SLE_\kappa$ run until it hits $z$. 
\end{lemma}
\begin{proof}
	From the change of variables $T^1 = T - \tau_z$, the law of $T^1$ is $1_{t^1>0} dt^1$. Conditioned on $T^1$ the conditional law of $\eta^{12}$ is $\SLE_\kappa^z$, and by the domain Markov property, conditioned on $T^1$ and $\eta^{12}$ the conditional law of $\eta^1$ is $\SLE_\kappa^{T^1}$.
\end{proof}

\begin{proof}[Proof of Proposition~\ref{prop-thm-8}]
	By Proposition~\ref{prop-mot} and Definition~\ref{def-quantum-cell}, the area $a$ quantum cell is the quantum surface obtained by mating $(X_t, Y_t)_{t \geq 0} \sim \mathrm{CRT}^2_\kappa$ for quantum area $a$. 
	
	Recall $M$ from Lemma~\ref{lem-shef-coupling}.
	Proposition~\ref{prop:MOT-zip-concrete} implies that the time-evolution of $M$ for quantum area $a$ arises from conformally welding an area $a$ quantum cell to $\phi_0$, and so corresponds to mating trees sampled from $\mathrm{CRT}^2_\kappa$ for quantum area $a$. Sending $a \to \infty$, we conclude that $M$ is the process of  Theorem~\ref{thm-sf-zipper}  when $n=0$ and $\delta = Q - \frac{3\gamma}4$. 
	The first property of $M$ in Lemma~\ref{lem-shef-coupling} is the desired description of $(\phi_\tau, \eta_\tau)$, so we obtain Theorem~\ref{thm-sf-zipper} for $n=0, \delta = Q-\frac{3\gamma}4$. Finally, we extend to arbitrary $\delta\in \R$ by Proposition~\ref{prop-reweight-infty}. 
\end{proof}

\begin{remark}\label{rem-difficulty}
	The main difficulty in adapting our proof of Proposition~\ref{prop-thm-8} to the case $\gamma = \sqrt2, \kappa = 8$ is in proving the analog of  Proposition~\ref{prop:MOT-zip-abstract}, where 
	$\LF_{\mathbb H}^{(\frac{3\gamma}4, 0), (Q - \frac{3\gamma}4, \infty)} = \LF_{\mathbb H}^{(\frac Q2, 0), (\frac Q2, \infty)}$ has to be replaced by $\lim_{\beta \uparrow Q} \frac1{Q - \beta} \LF_{\mathbb H}^{(\frac\beta2, 0), (\frac Q2, \infty)}$. This limit is constructed in \cite[Section 2.5]{ASY-triangle}. Our present argument does not apply since the $(M \, 1_{A>0}\,dA)$-measure of $E_N$ is infinite, that is,  Lemma~\ref{lem-EN-finite} fails for $\kappa = 8$. One could try using a finite event such as $E_N \cap \{ A < N\}$, or instead directly tackle the $\kappa = 8$ variant of Proposition~\ref{prop-thm-8} by taking a $\kappa \to 8$ limit. 
\end{remark}

\subsection{Proof of Theorem~\ref{thm-sf-zipper}}\label{sec-8-general}

In this section, we prove Theorem~\ref{thm-sf-zipper}.
The first step is to extend Proposition~\ref{prop-thm-8} to  the setting where insertions are allowed but the process stops before any insertions hit the curve:

\begin{proposition}\label{prop-thm-8-nonbulk}
	Theorem~\ref{thm-sf-zipper} holds for any stopping time $\tau$ such that a.s.\ $g_\tau(z_j) \not \in \eta_\tau$ for all $j$. 
\end{proposition}

Given this, we can  complete the proof of Theorem~\ref{thm-sf-zipper}.

\begin{proof}[Proof of Theorem~\ref{thm-sf-zipper}]
	Proposition~\ref{prop-main} gives us a  process $(\wt \phi_t, \wt \eta_t)_{0 \leq t \leq \tau}$ such that $\wt \phi_0$ has law $\LF_{\mathbb H}^{(-\frac1{\sqrt\kappa}, 0), (\alpha_j, z_j)_j, (\delta, \infty)}$ and $(\wt \phi_\tau, \wt \eta_\tau)$ has law~\eqref{eq-thm-simple-zipper-rho}, where $\tau$ is a stopping time for $\cF_t = \sigma(\wt \eta_t)$.	Let $S = \{ t : g_t(z_j) = W_t \text{ for some }j\}$.
	By Proposition~\ref{prop-thm-8-nonbulk}, if $\sigma, \sigma'$ are stopping times for $\cF_t$ such that $[\sigma, \sigma'] \cap S = \emptyset$, then on $[\sigma, \sigma']$,  this process can alternatively be described by conformal welding with a process $(\cC_s)_{s>0}$ with law $\mathrm{CRT}^2_\kappa$ (equivalently,  mating continuum random trees)  independent of $\cF_{\sigma}$ until the stopping time $\sigma'$. Since $S$ is  finite, by continuity we conclude that the \emph{whole} process $(\wt \phi_t, \wt \eta_t)_{0 \leq t \leq \tau}$ arises from the procedure described in Theorem~\ref{thm-sf-zipper}. This completes the proof. 
\end{proof}

The proof of Proposition~\ref{prop-thm-8-nonbulk} is identical to that of Proposition~\ref{prop-reweight-infty}, but with different calculations which we detail below. We will show how to  weight  the Liouville field  to introduce insertions in Lemma~\ref{lem-reweight-finite}, then apply this to the special case of Theorem~\ref{thm-sf-zipper} to obtain the more general statement desired. 

For Lemma~\ref{lem-reweight-finite} we need  Lemma~\ref{lem-green-conformal-infty} and the following. Recall  $G(z,w) = - \log|z-w| - \log |z - \ol w|$. 

\begin{lemma}\label{lem-green-conformal}
	Suppose $K \subset \ol {\mathbb H}$ is compact,  ${\mathbb H}\backslash K$ is simply connected and there is a conformal map $g: {\mathbb H}\to {\mathbb H} \backslash K$ such that $\lim_{|z| \to\infty} g(z)-z = 0$. Let $\eps > 0$. Then for $z_0 \in {\mathbb H}$ such that $\Im z_0 > \eps$ and $u \not \in g(B_\eps(z_0))$ we have $\int G(u,v) (g_* \theta_{\eps, z_0})(dv) = G(u, g(z_0))$, where $\theta_{\eps, z_0}$ is the uniform probability measure on $\partial B_\eps(z_0)$. For $x_0 \in \partial {\mathbb H}$ such that $g([x_0 - \eps, x_0 +\eps]) \subset \R$ and $u \not \in g(B_\eps(x_0) \cap {\mathbb H})$ we have $\int G(u, v)(g_* \theta_{\eps, x_0})(dv) = G(u, g(x_0))$, where $\theta_{\eps, x_0}$ is the uniform probability measure on $\partial B_\eps(x_0) \cap {\mathbb H}$.
\end{lemma}
\begin{proof}
	For the first assertion, the function $G(u, g(\cdot))$ is harmonic on $\ol {B_\eps(z_0)}$ because $G(u, \cdot)$ is harmonic and $g$ is conformal. Thus, by the mean value property of harmonic functions, $\int G(u,v) (g_*\theta_{\eps, z_0})(dv) = \int G(u, g(v)) \theta_{\eps, z_0}(dv) = G(u, g(z_0))$. The second assertion is proved similarly. 
\end{proof}

Now we add insertions to the Liouville field by weighting. We place constraints on the insertions so that they sum to $Q$; this allows us to work with $G(\cdot, \cdot)$ rather than the more complicated $G_{\mathbb H}(\cdot, \cdot)$. Recall that $\theta_{\eps, \infty}$ is the uniform probability measure on $\partial B_{1/\eps}(0) \cap {\mathbb H}$.  
\begin{lemma}\label{lem-reweight-finite}
	In the setting of Lemma~\ref{lem-green-conformal}, let $(\alpha_j, z_j) \in \R \times \ol{\mathbb H}$ and $\beta = -\sum_j \alpha_j$. Let $I = \{ i : z_i \in {\mathbb H}\}$ and $B = \{ b : z_b \in \R\}$. 
	Suppose $\eps >0$ satisfies $|z_j - z_k| > 2\eps$ for $j \neq k$,  $\Im z_i > \eps$ for $i \in I$,  $g((z_b - \eps, z_b + \eps))\subset \R$ for $b \in B$, and $g(-1/\eps), g(1/\eps)\in \R$. Let ${\mathbb H}_\eps^g = {\mathbb H} \backslash (g({\mathbb H} \backslash B_{1/\eps}(0)) \cup \bigcup_j g(B_\eps(z_j)))$ and $\theta_\eps := \sum_j \alpha_j \theta_{\eps, z_j} + \beta \theta_{\eps, \infty}$. For any $(\alpha, w) \in \R \times {\mathbb H}_\eps^g$ and nonnegative function $F(\phi)$ depending only on $\phi|_{{\mathbb H}_\eps^g}$, 
	\alb
	&\LF_{\mathbb H}^{(\alpha, w), (Q-\alpha, \infty)} [ \eps^{- \beta^2 +2 \alpha \beta } \prod_{i \in I} \eps^{\alpha_i^2/2}\prod_{b \in B} \eps^{\alpha_b^2}  \times e^{(g^{-1}\bullet_\gamma \phi, \theta_\eps)}F(\phi)]
	\\
	&=\prod_{i \in I} |g_\tau'(z_i)|^{2 \Delta_{\alpha_i}} \prod_{b \in B} |g_\tau'(z_b)|^{\Delta_{2\alpha_b}}
	\LF_{\mathbb H}^{(\alpha, w), (\alpha_j, g(z_j))_j, (\beta + Q - \alpha, \infty)}[F(\phi)].
	\ale
\end{lemma}
\begin{proof}
	The proof of this is identical to that of Lemma~\ref{lem-reweight-infty}; the only difference is in the computations of $Z_\eps:= {\mathbb E}[e^{(h, g_*\theta_\eps)}]$ and $e^{(g_t^{-1} \bullet_\gamma \phi, \theta_\eps)}$. 
	
	To compute $Z_\eps$ we first need $\Var((h, g_*\theta_\eps))$. Since $\int 1 \, g_*\theta_\eps(dv) = 0$, the law of $(h, g_*\theta_\eps)$ when $h \sim P_{\mathbb H}$ agrees with that when $h$ is instead considered as a distribution modulo additive constant. Thus, 
	$\Var ((h, g_*\theta_\eps)) =  \iint G(u,v) g_* \theta_\eps (du) g_*\theta_\eps(dv)$ where $G(u,v) = -\log|u-v| - \log |u - \ol v|$. 
	
	For $i \in I$ we have
	\begin{align*}
		\iint G(u,v) g_*\theta_{\eps, z_i}(du) g_*\theta_{\eps, z_i}(dv) &= \int G(g(z_i), g(v)) \theta_{\eps, z_i}(dv)
		\\ &= - \log |g'(z_i)| + \log (2 \Im (g(z_i))) + \log \eps,
	\end{align*}
	where in the first equality we use Lemma~\ref{lem-green-conformal} and
	in the second equality we use the fact that the function $f: B_\eps(z_i) \to \R$ defined by $f(w) = -\log|g_t(w) - g_t(z_i)| + \log |w - z_i|$  for $w \neq z_i$ and $f(z_i) = -\log |g_t'(z_i)|$ is harmonic. Similarly, $\iint G(u,v)  g_*\theta_{\eps, z_b}(du) g_*\theta_{\eps, z_b}(dv) = -2 \log |g_\tau'(z_b)| + 2 \log \eps$, and proceeding similarly we can compute all cross-terms. Summing gives the value of $\Var ((h, g_*\theta_\eps))$, and finally 
	\[Z_{\eps} = e^{\frac12{\Var} ((h, g_*\theta_\eps))}=  \eps^{-\beta^2 + \sum_i \alpha_i^2/2 + \sum_b \alpha_b^2} \prod_{i \in I} |g'(z_i)|^{-\alpha_i^2/2} \prod_{b \in B} |g'(z_b)|^{-\alpha_b^2} C_{\gamma}^{(\alpha_j, g(z_j))_j, (\beta, \infty)}.\]
	Next,  by Lemmas~\ref{lem-green-conformal-infty} and~\ref{lem-green-conformal} we have $(G(\cdot, w), g_*\theta_\eps) = -2\beta\log\eps + \sum_j \alpha_j G(z_j, w)$, so with $\phi = h +\alpha G(\cdot, w)  + c$, we have 
	\[e^{(g^{-1} \bullet_\gamma \phi, \theta_\eps)} = \eps^{2\alpha\beta} \frac{C_\gamma^{(\alpha, w), (\alpha_j, g(z_j))_j, (\delta, \infty)}}{C_\gamma^{(\alpha_j, g(z_j))_j, (\beta, \infty)}} e^{(h, g_*\theta_\eps)} .\]
	The rest of the argument is identical to that of Lemma~\ref{lem-reweight-infty}.
\end{proof}

\begin{proof}[Proof of Proposition~\ref{prop-thm-8-nonbulk}]
	By 
	Proposition~\ref{prop-thm-8},  Theorem~\ref{thm-sf-zipper} holds for $n=0, \delta = Q + \frac\gamma4$.  Applying the argument of Proposition~\ref{prop-reweight-infty} with  Lemma~\ref{lem-reweight-finite} as input, we obtain the $n \geq 0$ case when the insertions $(\alpha_j, z_j)_j$ and $(\delta, \infty)$ satisfy $-\frac1{\sqrt\kappa} + \sum_j \alpha_j +\delta = Q$. We were able to use  Lemma~\ref{lem-reweight-finite} because of the condition $g_\tau(z_j) \not \in \eta_\tau$. A final application of Proposition~\ref{prop-reweight-infty} removes this constraint. 	
\end{proof}

\section{The $\kappa \in (4,8)$ LCFT zipper}\label{sec-4-8}

In this section we prove Theorem~\ref{thm-forest-zipper}. 
\cite[Theorem 6.4.1]{wedges} says that when a particular GFF variant is cut by a forward $\SLE_{\kappa}$ with $\kappa \in (4,8)$, the connected components in the complement of the curve are a pair of independent forested lines. As we will explain, their argument proves the stronger result Proposition~\ref{prop-nonsimple-dms}, from which Theorem~\ref{thm-forest-zipper} easily follows. 

The following is part of the setup for the proof of \cite[Theorem 6.4.1]{wedges}, as described in \cite[Section 6.4.3]{wedges}. 
\begin{lemma}\label{lem-M-nonsimple}
	Let $\kappa \in (4, 8)$ and $\gamma = \frac4{\sqrt\kappa}$. 
	There exists a random process $(\wt h^t, \wt \eta^t)$ 
	with the following properties.
	\begin{enumerate}[(i)]
		\item  the marginal law of the process $\wt \eta^t$ is reverse $\SLE_\kappa$, and each $\wt \eta^t$ is parametrized by half-plane capacity, i.e.,  $\mathrm{hcap}(\wt \eta^t([0,s])) = 2s$ for $s \leq t$;  
		\item we have $\wt h^0 \stackrel d= h + \frac\gamma2 \log|\cdot|$ where $h\sim P_\bbH$;
		\item for any stopping time $\tau$ for the filtration $\cF_t = \sigma(\wt \eta^t)$, conditioned on $\wt \eta^\tau$, we have $\wt h^\tau \stackrel d= \wt h + \frac\gamma2 \log |\cdot - \wt \eta^\tau(0)|$ where $\wt h$ is a free boundary GFF with additive constant fixed in a way that depends on $\wt \eta^\tau$;
		\item for $t_1 < t_2$, let $g_{t_1, t_2}$ be the conformal map from $\bbH$ to the unbounded connected component of $\bbH\backslash \wt \eta^{t_2}([0, t_2 - t_1])$ such that $g_{t_1, t_2}(\wt \eta^{t_1}(0)) = \wt \eta^{t_2}(t_2 - t_1)$ and $\lim_{z \to \infty} g_{t_1, t_2}(z) - z = 0$. Then $\wt h^{t_1} = g_{t_1, t_2}^{-1} \bullet_\gamma \wt h^{t_2}$. 
	\end{enumerate}
\end{lemma}
\begin{proof}
	The proof is identical to that of Lemma~\ref{lem-shef-coupling}: we use  Proposition~\ref{prop-dms-chordal} to construct the process for $t \in [0,T]$ for any finite $T$, then apply Kolmogorov's extension theorem to extend this construction to all $t \geq 0$.
\end{proof}

\begin{proposition}	\label{prop-nonsimple-dms}
	Let $\kappa \in (4, 8)$ and $\gamma = \frac4{\sqrt\kappa}$. 
	Let $h_0 = h + \frac\gamma2 \log|\cdot|$ where $h \sim P_\bbH$, and independently sample   $(\wt \cS_s)_{s > 0} \sim \mathrm{FL}^2_\kappa$. Construct the process $(h_t, \eta_t)$ from $(h_0, (\wt \cS_s)_{s > 0})$ in the manner described above Theorem~\ref{thm-forest-zipper} with $\phi_0$ replaced by $h_0$. Then $(h_t,  \eta_t)$  agrees in law with the process described in Lemma~\ref{lem-M-nonsimple}. 
\end{proposition}
\begin{proof}
	The desired result is implicitly obtained in the proof of \cite[Theorem 6.4.1]{wedges}. We give a high-level overview of their argument, and explain how our desired result follows. 
	
	Consider the process $(\wt h^t, \wt \eta^t)$ from Lemma~\ref{lem-M-nonsimple}. At each time $t$, the bounded connected components of $\bbH \backslash \wt\eta^t$, with the LQG field $\wt h^t$ restricted to each component, give countably many quantum surfaces. Let $\cP_t$ be the collection of these quantum surfaces. By Property (iv) of Lemma~\ref{lem-M-nonsimple}, for $t < t'$ we have $\cP_t\subset \cP_{t'}$, so we can define $\cP_\infty = \lim_{ t\to\infty} \cP_t$ and  label each quantum surface $\cB \in \cP_\infty$ with the time $\tau(\cB) := \inf\{ t\: : \: \cB \in \cP_t\}$ at which it is ``zipped in''. 
	Let $\wh \cP_\infty$ be $\cP_\infty$ where the times $\tau(\cB)$ are forgotten but the induced ordering on the $\cB$ is retained. 
	
	As a step of the proof of \cite[Theorem 6.4.1]{wedges}, it is shown that $\wh \cP_\infty$ has the law of a Poisson point process of LQG disks (this is the  Poisson point process in the theorem statement of \cite[Theorem 6.4.1]{wedges}). By the definition of forested line in \cite[Section 1.4.2]{wedges}, the quantum disks in $\wh \cP_\infty$ to the left and right of the curve give a pair of independent forested lines $(F_L, F_R)$.

	Let $\wt D_t$ be the complement in $\bbH$ of the unbounded connected component of $\bbH \backslash \wt \eta^t$, let $\wt p_\ell(t), \wt p_r(t) \in \R$ satisfy $\wt D_t \cap \R = [\wt p_r(t), \wt p_\ell(t)]$, let $\wt \eta^t_\mathrm{rev}$ be the time-reversal of $\wt \eta^t$ parametrized by quantum length, and let $\wt \cS_t' := (\wt D_t, \wt h^t, \wt \eta^t_\mathrm{rev}, \wt \eta^t_\mathrm{rev}(0), \wt \eta^t(0), \wt p_\ell(t), \wt p_r(t))/{\sim_\gamma}$. Let $(\wt \cS_s)_{s>0}$ be the time-reparametrization of  $(\wt \cS_t')_{t > 0}$ according to the quantum length of the curve. We claim that $(\wt \cS_s)_{s > 0}$ is the incremental mating of $(F_L, F_R)$ from Proposition~\ref{prop-mate-forests}, i.e., the process with law $\mathrm{FL}^2_\kappa$. 
	This follows from three facts. First, for each $t$, the curve $\wt \eta^t$ and field $\wt h^t$ (viewed modulo additive constant) are, in neighborhoods bounded away from $\wt \eta^t(t)$ and $\R$, locally absolutely continuous with respect to the curve and field of  $\SLE_\kappa$ on an independent weight $\frac{3\gamma^2}2 - 2$ quantum wedge. Second, the latter setting arises from (a pair of quantum wedges and) a pair of independent forested lines mated in the manner specified by  Proposition~\ref{prop-mate-forests} \cite[Theorem 1.4.8]{wedges}. Third, the mating procedure of Proposition~\ref{prop-mate-forests} is local (Lemma~\ref{lem-local-mating-nonsimple}). 
	
	Thus, $\wt h^0 \stackrel d= h_0$ (Property (ii) of Lemma~\ref{lem-M-nonsimple}), and  the law of $(\wt \cS_s)_{s > 0}$ is $\mathrm{FL}^2_\kappa$. To conclude, we need to check that $\wt h^0$ is independent of $(\wt \cS_s)_{s > 0}$, i.e., $\wt h^0$ is  independent of $\wh \cP_\infty$. We deduce this by looking into the proof from \cite{wedges} that $\wh \cP_\infty$ has the law of a Poisson point process. For $c > 0$, let $\cP_\infty^c \subset \cP_\infty$ be the set of quantum surfaces having quantum area and length both at least $c$, and let $(\cB_j^c)_{j \geq 1}$ be the sequence of elements of $\cP_\infty^c$. 
	
	The proof of \cite[Theorem 6.4.1]{wedges} shows that for each $c$ and $n$, conditioned on $\cB_1^c, \dots, \cB_{n-1}^c$, the conditional law of $\cB_{n}^c$ is precisely the law $\cL_c$ of a quantum disk conditioned to have quantum area and boundary length at least $c$. We now summarize their argument. They consider a sequence of stopping times for $\cF_t = \sigma(\eta_t)$ such that a.s.\ some subsequence $t_1, t_2, \dots$ converges to $\tau(\cB_n^c)$ from above. For each $t_i$, they apply Property (iii) of Lemma~\ref{lem-M-nonsimple} and use the GFF Markov property to deduce that, conditioned on $(\tau(\cB^c_n), \wt h^{\tau(\cB_n^c)}, \wt \eta^{\tau(\cB_n^c)})$ and some further information, the conditional law of $\cB_n$ is close to $\cL_c$. Taking a limit as $i \to \infty$, 
		the conclusion is that conditioned on  $(\tau(\cB^c_n), \wt h^{\tau(\cB_n^c)}, \wt \eta^{\tau(\cB_n^c)})$, the conditional law of $\cB_n^c$ is $\cL_c$. They then use the fact that $(\cB_1^c, \dots, \cB_{n-1}^c)$ is measurable with respect to $(\wt h^{\tau(\cB_n^c)}, \wt \eta^{\tau(\cB_n^c)})$ to obtain the claim at the start of this paragraph. In our setting, we instead note that $(\wt h^0, \cP_{\tau(\cB^c_n)})$ is measurable with respect to $(\tau(\cB^c_n), \wt h^{\tau(\cB_n^c)}, \wt \eta^{\tau(\cB_n^c)})$, and hence $(\wt h^0, \cP_{\tau(\cB^c_n)})$ is independent of $\cB^c_n$. Varying $c$ and $n$, we conclude that $\wt h^0$ is independent of $\wh \cP_\infty$, as needed. 
\end{proof}

\begin{proof}[{Proof of Theorem~\ref{thm-forest-zipper}}]
	Sample $\mathbf c$ from Lebesgue measure on $\R$ and independently sample $(h_0, (\wt \cS_s)_{s > 0})$ as in Proposition~\ref{prop-nonsimple-dms}. Define $\phi_t = h_t + \mathbf c$, so the law of $\phi_0$ is $\LF_\bbH^{(-\frac\gamma2, 0), (\delta, \infty)}$ where $\delta =  Q + \frac\gamma4$. For each $s>0$ let $\cS_s$ be obtained from $\wt \cS_{e^{-\frac2\gamma \mathbf c} s}$ by adding $\mathbf c$ to its field; since for each constant $c>0$ the weight $2-\frac{\gamma^2}2$ quantum wedge is invariant in law under adding $c$ to its field, we conclude that the law of  $(\phi_0, (\cS_s)_{s>0})$ is $\LF_\bbH^{(-\frac\gamma2, 0), (\delta, \infty)} \times \mathrm{FL}^2_\kappa$. Proposition~\ref{prop-nonsimple-dms} thus gives Theorem~\ref{thm-forest-zipper} for $n=0$ and $\delta = Q + \frac\gamma4$. Next, Proposition~\ref{prop-reweight-infty} implies Theorem~\ref{thm-forest-zipper} for $n=0$. Finally,
	in Section~\ref{sec-8-general} the full Theorem~\ref{thm-sf-zipper} is obtained from its $n=0$ special case; this argument applies verbatim in our setting, except in the proof of Theorem~\ref{thm-sf-zipper}, the phrase ``conformal welding with a process $(\cC_s)_{s>0}$ with law $\mathrm{CRT}^2_\kappa$ (equivalently,  mating continuum random trees)'' should be replaced by ``conformal welding with a process $(\cS_s)_{s>0}$ with law $\mathrm{FL}^2_\kappa$ (equivalently, mating independent forested lines)''. Thus, the $n=0$ case of Theorem~\ref{thm-forest-zipper} implies the full Theorem~\ref{thm-forest-zipper}.
\end{proof}

In the proof of Proposition~\ref{prop-nonsimple-dms}, we used the fact that the mating of forested lines from Proposition~\ref{prop-mate-forests} is locally defined, which we now justify. The statement involves the  \emph{quantum cut point measure} on the set of cut points of a thin quantum wedge, which describes the time parameter of the Poisson point process description of the quantum wedge, see \cite[Section 2.4]{ahs-disk-welding} for details. 
\begin{lemma}\label{lem-local-mating-nonsimple}
	For $\kappa \in (4,8)$ and $\gamma =\frac4{\sqrt\kappa}$, there exists a (deterministic) function $\mathrm{Mate}$ such that the following holds. Let $0<t_1 < t_2$. Let $(D, h, 0, \infty)/{\sim_\gamma}$ be an embedding of a weight $2-\frac{\gamma^2}2$ quantum wedge and let $\eta$ be the curve from $0$ to $\infty$ obtained by  concatenating independent $\SLE_\kappa(\frac\kappa2-4; \frac\kappa2-4)$ curves in each component of $D$. Let $p_1, p_2 \in D$ be the cut points of $D$ at quantum cut point distances $t_1$ and $t_2$ from $0$, let $D' \subset D$ be the set of components of $D$ lying between $p_1$ and $p_2$, let $\eta'$ be the curve segment of $\eta$ from $p_1$ to $p_2$, and let $F_L'$ and $F_R'$ be the segments of forested lines obtained by cutting $\cD' = (D', h, \eta', p_1, p_2)/{\sim_\gamma}$ by $\eta'$.  
	Then $\cD' = \mathrm{Mate}(F_L', F_R')$ a.s..
\end{lemma}
\begin{proof}
	The weight $2-\frac{\gamma^2}2$ quantum wedge is an ordered Poissonian collection of beads (Definition~\ref{def-thin-wedge}), and the \emph{weight $2-\frac{\gamma^2}2$ quantum disk} can be defined as a random (interval) subset  of the weight $2-\frac{\gamma^2}2$ quantum wedge \cite[Definition 2.2]{ahsy-nonsimple}. By \cite[Proposition 3.25]{ahsy-nonsimple}, the forested line segments obtained by cutting a weight $2-\frac{\gamma^2}2$ quantum disk by independent $\SLE_\kappa(\frac\kappa2-4; \frac\kappa2-4)$ determine the decorated quantum disk a.s.; this gives the claim. 
\end{proof}

\section{The boundary BPZ equation for Liouville conformal field theory}\label{sec-bpz}

In this section we use the LCFT zipper to prove the boundary BPZ equations stated in Theorem~\ref{thm-bpz}. Let $m,n \geq 0$.  Let $(\alpha_j, z_j) \in \R \times {\mathbb H}$ for $j \leq m$ and assume $z_1, \dots, z_m$ are distinct. Let $-\infty = x_0 < x_1 < \dots < x_n < x_{n+1} = +\infty$, let $\beta_1, \dots, \beta_n \in \R$ and let $\delta \in \R$. Let $\beta_* \in \{ -\frac\gamma2, -\frac2\gamma\}$ and recall the LCFT correlation function $F_{\beta_*}(w, (z_j)_j, (x_k)_k)$ defined in~\eqref{eq-corr}. 

In Section~\ref{subsec-bpz-2gamma} we prove the BPZ equation holds for $\beta_* = -\frac2\gamma$, in Section~\ref{subsec-bpz-8} we handle the case $\beta_* = -\frac\gamma2$ and $\gamma \in (0,\sqrt2)$ and in Section~\ref{subsec-bpz-48} we settle the case $\beta_* = -\frac\gamma2$ and $\gamma \in (\sqrt2, 2)$ case. These sections use the coupling of LCFT with $\SLE_\kappa$ for $\kappa \in  (0,4], (8, \infty), (4,8)$ respectively.

\begin{proof}[Proof of Theorem~\ref{thm-bpz}]
	For $\gamma \in (0,2]$ the $\beta_*  = -\frac2\gamma$ BPZ equation is shown in Lemma~\ref{lem-bpz-simple}. 
	For $\gamma \in (0, \sqrt2)$ and $\beta_* = -\frac\gamma2$ it is shown in Lemma~\ref{lem-bpz-sf}, and for $\gamma \in (\sqrt2, 2)$ and $\beta_* = -\frac\gamma2$ see Lemma~\ref{lem-bpz-forr}. Finally, when $\gamma=2$ we have $- \frac\gamma2 = -\frac2\gamma$ so the $\beta_*=-\frac\gamma2$ BPZ equation has already been settled by the first case. 
\end{proof}

\subsection{Case: $\beta_* = -\frac2\gamma$}\label{subsec-bpz-2gamma}

Let $\kappa \leq 4$ and $\gamma = \sqrt \kappa$. 
Consider reverse $\SLE_\kappa$ where the driving function $W_t$ is given by $W_0 = w$ and $dW_t = \sqrt\kappa dB_t$. Let $(\eta_t)_{t \geq 0}$ be the family of curves and $g_t$ the corresponding Loewner maps, and let $T\leq \infty$ be the first time $t$ that $g_t(x_k) \in \eta_t$ for some $k$. For $t < T$ define 
\eqb\label{eq-Mt-simple}
M_t = \prod_j |g_t'(z_j)|^{2\Delta_{\alpha_j}} \prod_k |g_t'(x_k)|^{\Delta_{\beta_k}} F_{-\frac2\gamma} (W_t, (g_t(z_j))_j, (g_t(x_k))_k).
\eqe
\begin{lemma}\label{lem-Mt-simple-mtg}
	$M_t$ is a local martingale. 
\end{lemma}
\begin{proof}
	Let $\ol F_{-\frac2\gamma}(w, (z_j)_j, (x_k)_k )< \infty$ be defined via~\eqref{eq-corr} except with the integrand replaced by its absolute value, and let $\ol M_t = \prod_j |g_t'(z_j)|^{2\Delta_{\alpha_j}} \prod_k |g_t'(x_k)|^{\Delta_{\beta_k}} \ol F_{-\frac2\gamma}(w, (z_j)_j, (x_k)_k )$, so $\ol M_t \geq |M_t|$. 
	For $N>0$ let $T_N = T \wedge \inf \{t: \ol M_t \geq N \}$. Since $t\mapsto \ol M_t$ is continuous on $[0,T)$ we have $\lim_{N \to \infty} T_N= T$ almost surely. Thus we are done once we show $M_t$ stopped at $T_N$ is a martingale. 
	
	It suffices to show that for stopping times $\tau_1 \leq \tau_2 \leq T_N$ we have ${\mathbb E}[ M_{\tau_2} \mid \eta_{\tau_1}] = M_{\tau_1}$. Instead, to simplify notation, we show that if $\tau \leq T_N$ is a stopping time then ${\mathbb E}[M_\tau] = M_0$; the proof of the desired claim is identical.
	
	Sample $\phi_0 \sim  \LF_{\mathbb H}^{(\frac{\beta_*}2, w), (\alpha_j, z_j)_j, (\frac{\beta_k}2, x_k)_k, (\frac\delta2, \infty)}$, and define $(\phi_t, \eta_t)$ by conformally welding the boundary arcs to the left and right of $w$ as in Theorem~\ref{thm-simple-zipper}. Let $A_t = \cA_{\phi_t}({\mathbb H})$, let  $L_{k, t} = \cL_{\phi_t}(g_t(I_k))$ for $k \neq k_*$, and let $L_t = \cL_{\phi_t}((g_t(x_{k_*}), W_t))$  and $R_t = \cL_{\phi_t}((W_t, g_t(x_{k_*+1}))$.
	Since the conformal welding does not affect the quantum area nor quantum lengths of boundary segments not adjacent to $w$, the processes $A_t$ and $L_{k, t}$ are constant. Moreover, the conformal welding identifies segments of equal quantum length, so $L_t - R_t = L_0 - R_0$ for all $t$. Thus, $G_t = \exp( - A_t - \sum_{k \neq k_*} \mu_k L_{k, t} - \mu_L( L_t - R_t))$ is constant as $t$ varies. 
	Consequently, writing ${\mathbb E}$ to denote expectation with respect to $\rSLE_\kappa^\tau$ and using Theorem~\ref{thm-simple-zipper}, we have 
	\eqb\label{eq-mtg}
	\begin{aligned}
		M_0 &= \LF_{\mathbb H}^{(-\frac1{\sqrt\kappa}, w), (\alpha_j, z_j)_j, (\frac{\beta_k}2, x_k)_k, (\frac\delta2, \infty)}[G_0] \\
		&= {\mathbb E}[\prod_j |g_\tau'(z_j)|^{2\Delta_{\alpha_j}} \prod_k |g_\tau'(x_k)|^{\Delta_{\beta_k}} \LF_{\mathbb H}^{(-\frac1{\sqrt\kappa}, W_\tau), (\alpha_j, g_\tau(z_j))_j, (\frac{\beta_k}2, g_\tau(x_k))_k, (\frac\delta2, \infty)}[G_\tau]] = {\mathbb E}[M_\tau].
	\end{aligned}	 
	\eqe
	Here, the integrals converge absolutely by the definition of $T_N$. This completes the proof. 
\end{proof}

If we assume that $F_{-\frac2\gamma}$ is smooth, then setting the drift term of $dM_t$ to zero immediately gives Theorem~\ref{thm-bpz} for $\beta_* = -\frac2\gamma$. Since we do not have smoothness a priori, we only know $F_{-\frac2\gamma}$ is a weak solution to the BPZ equation. A hypoellipticity argument originally due to \cite[Lemma 5]{dubedat-virasoro} shows that weak solutions are also strong solutions; we instead follow \cite[Lemma 4.4, Proposition 2.6]{pw-multiple-SLE} which is closer to our setting. See also \cite[Appendix A]{ahsy-nonsimple} for a more detailed treatment of this argument.

\begin{lemma}\label{lem-bpz-simple}
	Theorem~\ref{thm-bpz} holds for $\beta_* = -\frac2\gamma$. 
\end{lemma}
\begin{proof}
	Consider the diffusion $X_t = (W_t, (g_t(z_j))_j, (g_t(x_k))_k, (|g_t'(z_j)|)_j, (|g_t'(x_k) |)_k)$ where $W_t$ and $g_t$ are defined above~\eqref{eq-Mt-simple}. For each $t$ we have $X_t \in \R \times {\mathbb H}^{m} \times \R^n \times \R^m \times \R^n$. We denote the coordinates of an element of $\R \times {\mathbb H}^{m} \times \R^n \times \R^m \times \R^n$ by $(\mathbf w, (\mathbf z_j)_j, (\mathbf x_k)_k, (\mathbf a_j)_j, (\mathbf b_k)_k)$; with this notation, since $d g_t(z) = -\frac2{g_t(z)-W_t}dt$ and $d|g_t'(z)| = \Re \frac{2|g_t'(z)|}{(g_t(z) - W_t)^2} dt$, the infinitesimal generator $A$ of $X_t$ is 
	\[A = \frac\kappa2 \partial_{\mathbf w}^2 - \sum_j (\frac{2}{\mathbf z_j - \mathbf w} \partial_{\mathbf z_j} + \frac{2}{\ol{\mathbf  z}_j - \mathbf w} \partial_{\ol{\mathbf  z}_j}) - \sum_k \frac2{\mathbf x_k - \mathbf w}\partial_{\mathbf x_k} + \sum_j \Re \frac{2\mathbf a_j}{(\mathbf z_j - \mathbf w)^2} \partial_{\mathbf a_j} + \sum_j \frac{2\mathbf b_k}{(\mathbf x_k - \mathbf w)^2} \partial_{\mathbf b_k}. \]
	Let $F(\mathbf w, (\mathbf z_j)_j, (\mathbf x_k)_k, (\mathbf a_j)_j, \mathbf b_k)_k) = \prod_j \mathbf a_j^{2\Delta_{\alpha_j}} \prod_k \mathbf b_k^{\Delta_{\beta_k}} F_{-\frac2\gamma}(\mathbf w, (\mathbf z_j)_j, (\mathbf x_k)_k)$. Since $M_t$ is a local martingale (Lemma~\ref{lem-Mt-simple-mtg}), $F$ is a weak solution to $AF = 0$. The product rule then yields
	\[0 = A \Big(\prod_j \mathbf a_j^{2\Delta_{\alpha_j}} \prod_k \mathbf b_k^{\Delta_{\beta_k}}  F_{-\frac2\gamma} (\mathbf w, (\mathbf z_j)_j, (\mathbf x_k)_k)\Big) = \prod_j \mathbf a_j^{2\Delta_{\alpha_j}} \prod_k \mathbf b_k^{\Delta_{\beta_k}}  \cD F_{-\frac2\gamma}(\mathbf w, (\mathbf z_j)_j, (\mathbf x_k)_k)\]
	where $\cD$ is the differential operator on $\R \times{\mathbb H}^m \times \R^n$ given by 
	\eqb\label{eq-cD}
	\cD :=  \frac\kappa2 \partial_{\mathbf w}^2 - \sum_j (\frac{2}{\mathbf z_j - \mathbf w} \partial_{\mathbf z_j} + \frac{2}{\ol{\mathbf  z}_j - \mathbf w} \partial_{\ol{\mathbf  z}_j}) - \sum_k \frac2{\mathbf x_k - \mathbf w}\partial_{\mathbf x_k} + \sum_j \Re \frac{4\Delta_{\alpha_j}}{(\mathbf z_j - \mathbf w)^2}  + \sum_j \frac{2 \Delta_{\beta_k}}{(\mathbf x_k - \mathbf w)^2}. 
	\eqe
	We conclude that $\cD F_{-\frac2\gamma} = 0$ in the distributional sense. 
	
	The operator $\cD$ is called \emph{hypoelliptic} if the weak solutions of $\cD F = 0$ are smooth.  H\"ormander's condition gives a criterion for hypoellipticity which we verify for $\cD$ in Lemma~\ref{lem-hypoelliptic}.
	Thus, $F_{-\frac2\gamma}$ is smooth, and $\cD F_{-\frac2\gamma} = 0$ is the desired BPZ equation. 
\end{proof}

\begin{lemma}\label{lem-hypoelliptic}
Let $U \subset \R \times \bbH^m \times \R^n$ be the set of tuples $(\mathbf w, (\mathbf z_j)_j, (\mathbf x_k)_k)$ such that $\mathbf z_1, \dots, \mathbf z_m$ are pairwise distinct, and  $\mathbf w,\mathbf x_1, \dots, \mathbf x_n$ are pairwise distinct.  The operator $\cD$ defined in~\eqref{eq-cD} is hypoelliptic on $U$. 
\end{lemma}
\begin{proof}
	Let $A_0 = - \sum_j (\frac{2}{\mathbf z_j - \mathbf w} \partial_{\mathbf z_j} + \frac{2}{\ol{\mathbf  z}_j - \mathbf w} \partial_{\ol{\mathbf  z}_j}) - \sum_k \frac2{\mathbf x_k - \mathbf w}\partial_{\mathbf x_k}$ and $A_1 = \sqrt{\frac\kappa2} \partial_{\mathbf w}$ be smooth vector fields on $U$, and define the smooth function $b = \sum_j \Re \frac{4\Delta_{\alpha_j}}{(\mathbf z_j - \mathbf w)^2}  + \sum_j \frac{2 \Delta_{\beta_k}}{(\mathbf x_k - \mathbf w)^2}$. H\"ormander's theorem states that if at every point in $U$ the vector fields $A_0$, $A_1$ and their commutators generate a full-rank vector space, then 
	 $\cD= A_1^2 + A_0 + b$ is hypoelliptic on $U$ \cite[Theorem 1.1]{hormander-hypoelliptic}. We will verify H\"ormander's condition for our choice of $A_0, A_1$. Let $A_0^{[0]} = A_0$ and for $\ell \geq 1$ inductively define $A_0^{[\ell]} := \frac1\ell [ \partial_{\mathbf w} , A_0^{[\ell-1]}] = \frac1\ell (\partial_{\mathbf w} A_0^{[\ell - 1]} - A_0^{[\ell-1]} \partial_{\mathbf w})$. Then
	 \begin{align*}
	 	A_0^{[\ell]} = -\sum_j (\frac{2}{(\mathbf z_j - \mathbf w)^{\ell+1}} \partial_{\mathbf z_j} + \frac2{(\ol{\mathbf z}_j - \mathbf w)^{\ell+1}} \partial_{\ol{\mathbf z}_j}) - \sum_k \frac2{(\mathbf x_k - \mathbf w)^{\ell+1}} \partial_{\mathbf x_k} \\
	 	= -\sum_j (\Re \frac{2}{(\mathbf z_j - \mathbf w)^{\ell+1}} \partial_{\mathbf u_j} + \Im \frac2{(\ol{\mathbf z}_j - \mathbf w)^{\ell+1}} \partial_{\mathbf v_j}) - \sum_k \frac2{(\mathbf x_k - \mathbf w)^{\ell+1}} \partial_{\mathbf x_k}
	 \end{align*}
 where we have changed coordinates from $(\partial_{\mathbf z_j},\partial_{\ol{\mathbf z}_j})$ to $(\partial_{\mathbf u_j},\partial_{\mathbf v_j})$ for $(\mathbf u_j, \mathbf v_j) := (\Re \mathbf z_j, \Im \mathbf z_j)$. 
 
	 We now verify that, evaluated at each point in $U$, the $(2m+n+1)$ vectors $A_1, A_0^{[0]}, \dots, A_0^{[2m+n-1]}$  are linearly independent in $\R^{2m+n+1}$. Since $A_1$ is the only vector with a $\partial_{\mathbf w}$ term, it suffices to show $A_0^{[0]}, \dots, A_0^{[2m+n-1]}$ are linearly independent. This is equivalent to $\det B \neq 0$, where  $B$ is the $(2m+n) \times (2m+n)$ matrix whose rows are
	 \begin{align*}
	 	\begin{bmatrix}
	 		- \Re \frac2{(\mathbf z_j - \mathbf w)^{1}} & \cdots & - \Re \frac2{(\mathbf z_j - \mathbf w)^{2m+n}}
	 	\end{bmatrix}& \quad \text{ for }j = 1, \dots, m, \\
 	\begin{bmatrix}
 		- \Im \frac2{(\ol{\mathbf z}_j - \mathbf w)^{1}} & \cdots & - \Im \frac2{(\ol{\mathbf z}_j - \mathbf w)^{2m+n}}
 	\end{bmatrix}& \quad \text{ for }j = 1, \dots, m, \\
 \begin{bmatrix}
 	- \frac2{(\mathbf x_k - \mathbf w)^{1}} & \cdots & - \frac2{(\mathbf x_k - \mathbf w)^{2m+n}}
 \end{bmatrix}& \quad \text{ for }k = 1, \dots, n.
 		 \end{align*}
Using row operations, $\det B \neq 0$ is equivalent to $\det \wt B \neq 0$, where $\wt B$ is the Vandermonde matrix whose  rows are $[ a^1 \, \,  \cdots \, \, a^{2m+n}]$ for $a = (\mathbf z_j - \mathbf w)^{-1}$ ($m$ rows),  $a = (\ol{\mathbf z}_j - \mathbf w)^{-1}$ ($m$ rows), and $a = (\mathbf x_k - \mathbf w)^{-1}$ ($n$ rows). Since the $2m+n$ complex numbers $((\mathbf z_j - \mathbf w)^{-1})_j, ((\ol{\mathbf z}_j - \mathbf w)^{-1})_j, ((\mathbf x_k - \mathbf w)^{-1})_k$ are nonzero and pairwise distinct, we conclude $\det \wt B \neq 0$. This completes the proof. 
\end{proof}

\subsection{Case: $\gamma \in (0,\sqrt2)$ and $\beta_* = -\frac\gamma2$ }\label{subsec-bpz-8}
The key to the BPZ equation is the following martingale which explains the coupling~\eqref{eq-cosmocoupling} between $\mu_L,\mu_R$. 

\begin{lemma}\label{lem-bm-mtg}
	For $\kappa > 4$, let $(X,Y)$ be correlated Brownian motions with covariance
	\[\Var(X_s) = \Var(Y_s) = \mathbbm a^2 s, \quad  \Cov(X_s, Y_s) = -\mathbbm a^2 \cos(\theta)s \quad  \text{ where }  \theta = \frac{4\pi}{\kappa}, \,\, \mathbbm a^2 = \frac{2}{\sin(\theta)}.\]
Let $x \in {\mathbb C}$ and set $\mu_L = \sqrt{1/\sin \theta} \cos x$ and $\mu_R = \sqrt{1/\sin \theta} \cos (x + \theta)$. 
Then the process $e^{- s - \mu_L X_s - \mu_R Y_s}$ is a martingale. 
\end{lemma}
\begin{proof}
	We need the trigonometric identity
	\eqb\label{eq-trig}
	\cos^2 x + \cos^2 (x+\theta) - 2 \cos x \cos(x+\theta) \cos \theta = \sin^2 \theta.
	\eqe
	To see that this holds, we compute
	\alb
	&2 \cos x \cos(x+\theta) \cos \theta = (\cos (2x +\theta) + \cos (\theta) ) \cos \theta \\ &= \frac12 \cos(2(x+\theta)) + \frac12 \cos 2x + 1 - \sin^2\theta
	= \cos^2 (x+\theta) + \cos^2 x - \sin^2\theta,
	\ale
	where the first equality uses the product-to-sum formula, the second  uses the product-to-sum formula and $\sin^2 + \cos^2 = 1$, and the last uses the double-angle formula. Now,~\eqref{eq-trig} gives
	\[{\mathbb E}[e^{-s - \mu_L L_s - \mu_R R_s}] = \exp(- s + \frac12 (\mu_L^2 + \mu_R^2 - 2 \mu_L \mu_R \cos \theta) \mathbbm a^2 s) = \exp(- s + \frac {1}{2 \sin \theta} \cdot \sin^2 \theta \cdot \mathbbm a^2 s) = 1. \]
	This and the strong Markov property of Brownian motion imply that $e^{-s - \mu_L L_s - \mu_R R_s}$ is a martingale. 
\end{proof}

	Since $(X,Y)$ has the law of the Brownian motion in Proposition~\ref{prop-mot}, 
	Lemma~\ref{lem-bm-mtg} immediately gives the following. Here we have the ``mismatched'' term $\mu_L R(\cC_s)$ because eventually $\partial_r^-\cC_s$ will be welded to the left boundary of another surface. 
\begin{lemma}\label{lem-mtg-8}
	Let $\kappa > 8$ and $\gamma = \frac4{\sqrt\kappa}$. 
	Sample $(\cC_s)_{s>0} \sim \mathrm{CRT}^2_\kappa$, let $A(\cC_s) = s$ be the quantum area of $\cC_s$, let $L(\cC_s)$ be the quantum length of $\partial_\ell^+ \cC_s$ minus the quantum length of $\partial_\ell^- \cC_s$, and let $R(\cC_s)$ be defined analogously with $\ell$ replaced by $r$.  
	Let $\sigma$ be a stopping time for the filtration generated by $(\cC_s)_{s>0}$. For $\mu_L, \mu_R$ as in Lemma~\ref{lem-bm-mtg},
	\[\E[e^{-A(\cC_s) - \mu_L R(\cC_s) - \mu_R L(\cC_s)}] = 1. \] 
\end{lemma}
\begin{proof}
	Using Lemma~\ref{lem-bm-mtg}, the claim is immediate from the fact that if $(X_s, Y_s)_{s>0}$ is the process in Proposition~\ref{prop-mot}, then by definition $(R(\cC_s), L(\cC_s)) = (X_s, Y_s)$. 
\end{proof}

Consider reverse $\SLE_\kappa$ where the driving function $W_t$ is given by $W_0 = w$ and $dW_t = \sqrt\kappa dB_t$. Let $(\eta_t)_{t \geq 0}$ be the family of curves, and let $T\leq \infty$ be the first time $t$ that $g_t(x_k) \in \eta_t$ for some $k$. For $t < T$ define 
\[M_t = \prod_j |g_t'(z_j)|^{2\Delta_{\alpha_j}} \prod_k |g_t'(x_k)|^{\Delta_{\beta_k}} F_{-\frac\gamma2} (W_t, (g_t(z_j))_j, (g_t(x_k))_k).\] 

\begin{lemma}\label{lem-Mt-sf-mtg}
	In the case $\kappa > 8$, $M_t$ is a local martingale.
\end{lemma}

\begin{proof}
	For $N>0$ define the stopping time $T_N$ as in Lemma~\ref{lem-Mt-simple-mtg}, then it suffices to show that $M_t$ stopped at $T_N$ is a martingale. 
	As before, we need to show that for stopping times $\tau_1 \leq \tau_2 \leq T_N$ we have ${\mathbb E}[ M_{\tau_2} \mid \eta_{\tau_1}] = M_{\tau_1}$. We instead show that if $\tau \leq T_N$ is a stopping time then ${\mathbb E}[M_\tau] = M_0$; the former is proved identically. Similarly, to lighten notation we assume that $\mu_k = 0$ for all $k \neq k_*$.
	
	Sample $(\phi_0, (\cC_s)_{s>0} ) \sim \LF_{\mathbb H}^{(\frac{\beta_*}2, w), (\alpha_j, z_j)_j, (\frac{\beta_k}2, x_k), (\frac\delta2, \infty)} \times \mathrm{CRT}^2_\kappa$, use conformal welding to glue $\cC_s$ to $(\bbH, \phi_0, w, \infty)/{\sim_\gamma}$ (identifying the curve starting point of $\cC_s$ with the boundary point $w$), and define $(\phi_t, \eta_t)$ as in Theorem~\ref{thm-sf-zipper}.
	For each $t \leq \tau$ let $s(t) := \cA_{\phi_t}(\eta_t((0,t)))$, and let $\sigma = s(\tau)$. The time $\sigma$ is a stopping time for the filtration $\wt\cF_s = \sigma(\phi_0,(\cC_\cdot)_{[0,s]})$, since for $s = s(t)$ the pair $(\phi_{t}, \eta_{t})$ is obtained from $(\phi_0,(\cC_\cdot)_{[0,s]})$ by conformal welding. 
	
	Let $A_t = \cA_{\phi_t}({\mathbb H})$, $L_t = \cL_{\phi_t}((g_t(x_{k_*}), W_t))$ and $R_t = \cL_{\phi_t}((W_t, g_t(x_{k_*+1})))$. Let $G_t = e^{-A_t - \mu_L L_t - \mu_R R_t}$. 	By the conformal welding construction of the process we have $(A_\tau - A_0, L_\tau - L_0, R_\tau - R_0) = (A(\cC_\sigma), R(\cC_\sigma), L(\cC_\sigma))$ where the latter quantities are as defined in Lemma~\ref{lem-mtg-8}. Since $\sigma$ is a stopping time for $\wt \cF_s$, Lemma~\ref{lem-mtg-8} gives 
	\[{\mathbb E}[G_\tau \mid \phi_0] =  G_0 \E[e^{-A(\cC_\sigma) - \mu_L R(\cC_\sigma) - \mu_R L(\cC_\sigma)} \mid \phi_0] = 1. \]
	Consequently, by Theorem~\ref{thm-sf-zipper}, Equation~\eqref{eq-mtg} holds in this setting, so ${\mathbb E}[M_\tau] = M_0$ as desired. 
\end{proof}

\begin{lemma}\label{lem-bpz-sf}
	Theorem~\ref{thm-bpz} holds for $\beta_* = -\frac\gamma2$ and $\gamma \in (0,\sqrt2)$. 
\end{lemma}
\begin{proof}
	Given Lemma~\ref{lem-Mt-sf-mtg} the argument of Lemma~\ref{lem-bpz-simple} implies the result. 
\end{proof}

\subsection{Case: $\gamma \in (\sqrt2,2)$ and $\beta_* = -\frac\gamma2$ }\label{subsec-bpz-48}
In this section we prove  Theorem~\ref{thm-bpz} for $\gamma \in (\sqrt2,2)$ and $\beta_* = -\frac\gamma2$. 

We set $\kappa = \frac{16}{\gamma^2}\in (4,8)$. 
In this regime, $\SLE_\kappa$ is self-hitting but not space-filling. There is a variant called \emph{space-filling $\SLE_\kappa$} \cite{ig4} with the property that if $\eta^\mathrm{SF}$ is space-filling $\SLE_\kappa$ in $({\mathbb H}, 0, \infty)$, and $T$ is the set of times $t$ that $\eta^\mathrm{SF}(t)$ is on the boundary of the unbounded connected component of ${\mathbb H}\backslash \eta^\mathrm{SF}([0,t])$, then the ordered collection of points $\{ \eta^\mathrm{SF}(t) : t \in T\}$ has the law of (non-space-filling\footnote{The $\SLE_\kappa$ curve for $\kappa \in (4,8)$ is non-space-filling, but in this section we will (redundantly) refer to it as non-space-filling $\SLE_\kappa$ to make a clear distinction from space-filling $\SLE_\kappa$.}) $\SLE_\kappa$ in $({\mathbb H}, 0, \infty)$. This gives a coupling of space-filling $\SLE_\kappa$ and non-space-filling $\SLE_\kappa$.

The weight $(2 - \frac{\gamma^2}2)$ quantum wedge is thin since $2 - \frac{\gamma^2}2 \in (0, \frac{\gamma^2}2)$ for $\gamma \in (\sqrt2,2)$. As in the case of non-space-filling $\SLE_\kappa$, we define space-filling $\SLE_\kappa$ on the thin quantum wedge to be the concatenation of independent space-filling $\SLE_\kappa$ curves between the marked points in each connected component.
We state the mating of trees theorem \cite[Theorem 1.4.1]{wedges} for this range of $\gamma$, see Figure~\ref{fig-mot-forr}. 
\begin{proposition}\label{prop-mot-forr}
	Let $\kappa \in (4,8)$ and $\gamma = \frac{4}{\sqrt\kappa}$. Sample a weight $(2-\frac{\gamma^2}2)$ quantum wedge decorated by an independent \emph{space-filling} $\SLE_\kappa$ $\eta^\mathrm{SF}$.  
	Parametrize $\eta^\mathrm{SF}$ by quantum area. On the counterclockwise (resp.\ clockwise) boundary arc of $\eta^\mathrm{SF}([0,s])$ from $0$ to $\eta(s)$, let $X_s^-$ and $X_s^+$ (resp.\ $Y_s^-$ and $Y_s^+$) be the quantum lengths of the boundary segments lying in the quantum wedge's boundary and bulk respectively. Then the law of $(X_s, Y_s) := (X_s^+ - X_s^-,Y_s^+ - Y_s^-)$ is Brownian motion with covariance as in Proposition~\ref{prop-mot}. Moreover, the curve-decorated quantum wedge is measurable with respect to $(X_s, Y_s)_{s \geq 0}$.
\end{proposition}
	Using the above, we obtain an analog of Lemma~\ref{lem-mtg-8} for $(\cS_s)_{s>0}$ by relating the mating of independent forested lines to that of the correlated continuum random trees as in Figure~\ref{fig-forr-vs-mot}. We have already defined $\partial_\ell^-$, and now define $\partial_\ell^+\cS_s$ to be the boundary arc of $\cS_s$ such that the disjoint union $\partial_\ell^+\cS_s \cup \partial_\ell^- \cS_s$ is the left boundary arc of $\cS_s$. We likewise define $\partial_r^-$. 
	\begin{lemma}\label{lem-mtg-48}
		Let $\kappa \in (4,8)$ and $\gamma = \frac4{\sqrt\kappa}$. For $(\cS_s)_{s>0} \sim \mathrm{FL}^2_\kappa$, 
		let $A(\cS_s)$ be the quantum area of $\cS_s$, let $L(\cS_s)$ be the quantum length of $\partial_\ell^+ \cS_s$ minus the quantum length of $\partial_\ell^- \cS_s$, and similarly define $R(\cS_s)$. For $\mu_L, \mu_R$ chosen as in Lemma~\ref{lem-bm-mtg}, 
		we have $\E[e^{- A(\cS_s) - \mu_L R(\cS_s) - \mu_R L(\cS_s)}] = 1$ for any stopping time $\sigma$ for the filtration  generated by $(\cS_s)_{s>0}$. 
	\end{lemma}
\begin{proof}
		In the setting of Proposition~\ref{prop-mot-forr}, 
	let $\eta$ be the non-space-filling $\SLE_\kappa$ coupled with $\eta^\mathrm{SF}$. Parametrize $\eta$ by quantum length, and use it to construct the process $(\cS_s)_{s > 0}$ as in Section~\ref{sec-intro-48}.
	For each $s>0$, let $t(s) = \inf \{ t \: : \: \eta^\mathrm{SF}(t) = \eta(s)\}$. By a local version of the measurability claim of Proposition~\ref{prop-mot-forr},
	the process $(X, Y)|_{[0,t(s)]}$ determines $\cS_s$ \cite[Lemma 2.17]{ag-disk}. Therefore, any  stopping time $\sigma$ for $(\cS_s)_{s > 0}$ gives a stopping time $t(\sigma)$ for $(X, Y)$. 
	
	By definition, $(A(\cS_s), R(\cS_s), L(\cS_s)) = (t(s), X_{t(s)}, Y_{t(s)})$. Since $t(\sigma)$ is a stopping time for $(X,Y)$, the claim follows from Lemma~\ref{lem-bm-mtg}.
\end{proof}

\begin{figure}
	\begin{center}
		\includegraphics[scale=0.48]{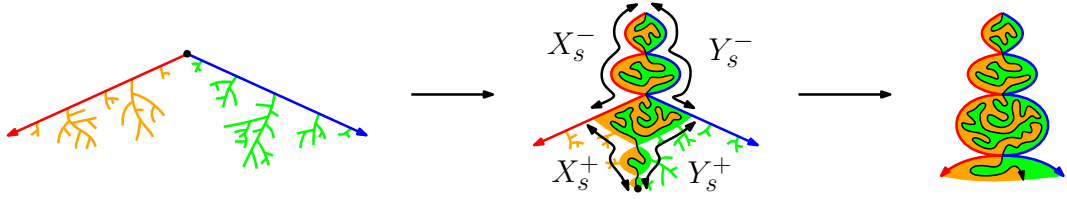}
	\end{center}
	\caption{\label{fig-mot-forr} Let $\kappa \in (4,8)$. \textbf{Left:} A pair of correlated continuum random trees described by $(X_s, Y_s)_{s \geq 0}$.  \textbf{Right:} Mating the trees gives a weight $(2-\frac{\gamma^2}2)$ quantum wedge decorated by an independent \emph{space-filling} $\SLE_\kappa$ curve.  \textbf{Middle:} The trees have been mated until the quantum area is $s$. We write $X_s^-, X_s^+, Y_s^-, Y_s^+$ for the quantum lengths of the four labelled boundary arcs.}
\end{figure}

\begin{figure}
	\begin{center}
		\includegraphics[scale=0.48]{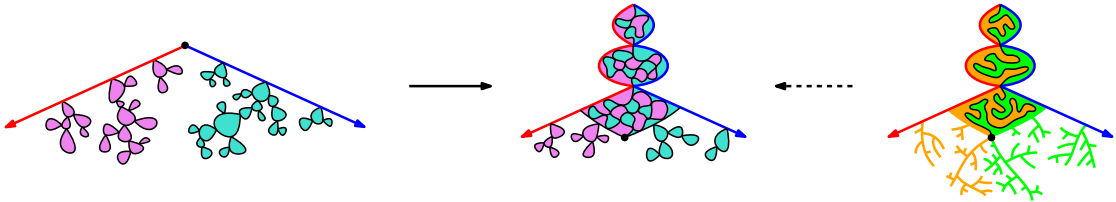}
	\end{center}
	\caption{\label{fig-forr-vs-mot} Let $\kappa \in (4,8)$. \textbf{Left:} A pair of independent forested lines.  \textbf{Middle:} The forested lines have been mated for some amount of time. \textbf{Right:} The pair of forested lines can be coupled with a sample from $\mathrm{CRT}^2_\kappa$ such that the middle picture is obtained by mating continuum random trees, then replacing the space-filling $\SLE_\kappa$ with its coupled non-space-filling $\SLE_\kappa$. }
\end{figure}

As before, consider reverse $\SLE_\kappa$ where the driving function $W_t$ is given by $W_0 = w$ and $dW_t = \sqrt\kappa dB_t$. Let $(\eta_t)_{t \geq 0}$ be the family of curves, and let $T\leq \infty$ be the first time $t$ that $g_t(x_k) \in \eta_t$ for some $k$. For $t < T$ define 
\[M_t = \prod_j |g_t'(z_j)|^{2\Delta_{\alpha_j}} \prod_k |g_t'(x_k)|^{\Delta_{\beta_k}} F_{-\frac\gamma2} (W_t, (g_t(z_j))_j, (g_t(x_k))_k).\]

\begin{lemma}\label{lem-Mt-forr-mtg}
	In the case $\kappa \in (4,8)$, $M_t$ is a local martingale.
\end{lemma}
\begin{proof}
	The argument of Lemma~\ref{lem-Mt-sf-mtg} carries over verbatim to this setting, using Lemma~\ref{lem-mtg-48} instead of Lemma~\ref{lem-mtg-8}.
\end{proof}

\begin{lemma}\label{lem-bpz-forr}
	Theorem~\ref{thm-bpz} holds for $\beta_* = -\frac\gamma2$ and $\gamma \in (\sqrt2,2)$. 
\end{lemma}
\begin{proof}
	Given Lemma~\ref{lem-Mt-forr-mtg} the argument of Lemma~\ref{lem-bpz-simple} implies the result. 
\end{proof}

\bibliographystyle{hmralphaabbrv}
\bibliography{cibib}

\end{document}